\documentclass[reqno,11pt]{amsart}
\usepackage{amsmath,amssymb,amsthm,graphicx,mathrsfs,url, enumerate,transparent}
\usepackage[usenames,dvipsnames]{color}
\usepackage[colorlinks=true,linkcolor=Red,citecolor=Green]{hyperref}
\usepackage[super]{nth}
\hypersetup{pdfstartview=FitH,bookmarksopen=true,pdfpagemode=UseOutlines}
\hypersetup{pdftitle=Ruelle-Taylor resonances of Anosov actions,pdfdisplaydoctitle=true}

\usepackage[font=footnotesize]{caption}
\usepackage{a4wide}

\newtheorem{theorem}{Theorem}
\newtheorem{prop}{Proposition}[section]
\newtheorem{lemma}[prop]{Lemma}
\newtheorem{cor}[prop]{Corollary}
\newtheorem*{lemma*}{Lemma}
\theoremstyle{definition}
\newtheorem{Def}[prop]{Definition}

\newtheorem{example}[prop]{Example}{\bf}{\it}
\theoremstyle{remark}
\newtheorem{rem}[prop]{Remark}
\numberwithin{equation}{section}

\newcommand{\N}{\mathbb{N}}

\newcommand{\R}{\mathbb{R}}
\newcommand{\Z}{\mathbb{Z}}
\newcommand{\C}{\mathbb{C}}

\newcommand{\T}{\text{T}}
\newcommand{\X}{\mathbf{X}}
\newcommand{\M}{\mathcal{M}}
\newcommand{\W}{\mathcal{W}}

\newcommand{\g}{\mathfrak g}

\newcommand{\n}{\mathfrak n}
\renewcommand{\a}{\mathfrak a}

\renewcommand{\Re}{\operatorname{Re}}

\DeclareMathOperator{\Id}{Id}

\DeclareMathOperator{\Res}{Res}

\DeclareMathOperator{\Ell}{ell}

\DeclareMathOperator{\Op}{Op}

\DeclareMathOperator{\supp}{supp}

\DeclareMathOperator{\WF}{WF}

\DeclareMathOperator{\ran}{ran}
\DeclareMathOperator{\Index}{index}
\DeclareMathOperator{\OpDiff}{Diff}
\DeclareMathOperator{\Diffeo}{Diffeo}

\newcommand{\mc}{\mathcal}
\newcommand{\pl}{\partial}
\newcommand{\cjg}{\langle}
\newcommand{\cjd}{\rangle}
\newcommand{\la}{\lambda}
\newcommand{\eps}{\epsilon}

\newcommand{\bbar}{\overline}

\newcommand{\acs}{\mathfrak{a}_{\mathbb C}^\ast}

\newcommand{\tu}[1]{\textup{#1}}
\newcommand{\Abb}[4]{\begin{cases} #1 & \rightarrow  #2 \\ #3 &\mapsto  #4\end{cases}}

\title[Ruelle-Taylor resonances of Anosov actions]{Ruelle-Taylor resonances of Anosov actions}

\author[Y. Guedes Bonthonneau]{Yannick Guedes Bonthonneau}
\email{bonthonneau@math.univ-paris13.fr}
\address{Universit\'e Sorbonne Paris Nord, CNRS, 93430, Villetaneuse, France}
\author[C. Guillarmou]{Colin Guillarmou}
\email{colin.guillarmou@universite-paris-saclay.fr}
\address{Universit\'e Paris-Saclay, CNRS,  Laboratoire de math\'ematiques d'Orsay, 91405, Orsay, France.}
\author[J. Hilgert]{Joachim Hilgert}
\email{hilgert@math.upb.de}
\address{Universit\"at Paderborn, Warburgerstr. 100, 33098 Paderborn, Germany}
\author[T. Weich]{Tobias Weich}
\email{weich@math.upb.de}
\address{Universit\"at Paderborn, Warburgerstr. 100, 33098 Paderborn, Germany}

\date{\today}

\begin{document}

\begin{abstract}
Combining microlocal methods and a cohomological theory developed by J. Taylor, we define for $\R^\kappa$-Anosov actions a notion of joint Ruelle resonance spectrum. We prove that these Ruelle-Taylor resonances fit into a Fredholm theory, are intrinsic and form a discrete subset of $\C^\kappa$, with $\lambda=0$ being always a leading resonance. The joint resonant states at $0$ give rise to some new measures of SRB type and the mixing properties of these measures are related to the existence of purely imaginary resonances. The spectral theory developed in this article applies in particular to the case of Weyl chamber flows and provides a new way to study such flows.
\end{abstract}

\maketitle
\sloppy
\section{Introduction}
If $P$ is a differential operator on a manifold $M$ that has purely discrete spectrum as an unbounded operator acting on $L^2(M)$ (e.g. an elliptic operator on a closed Riemannian manifold $M$), then the eigenvalues and eigenfunctions carry a huge amount of information about the dynamics generated by $P$. Furthermore,  if $P$ is a geometric differential operator (e.g. Laplace-Beltrami operator, Hodge-Laplacian or Dirac operators) the discrete spectrum encodes important topological and geometric invariants of the manifold $M$. 

Unfortunately, in many cases (e.g. if the manifold $M$ is not compact or if $P$ is non-elliptic) the $L^2$-spectrum of $P$ is not discrete anymore but consists mainly of essential spectrum. Still, there are certain cases where the essential spectrum of $P$ is non-empty, but where there is a hidden intrinsic discrete spectrum attached to $P$, called the \emph{resonance spectrum}.
To be more concrete, let us give a couple of examples of this theory: 
\begin{itemize}
\item Quantum resonances of Schr\"odinger operators $P=\Delta+V$ with $V\in C_c^\infty(\R^n)$ on $M=\R^n$ with $n$ odd  
(see for example \cite[Chapter 3]{DZ19} for a textbook account to this classical theory). 
\item Quantum resonances for the Laplacian on non-compact geometrically finite hyperbolic manifolds $M=\Gamma\backslash \mathbb{H}^{n+1}$: here $P=\Delta_M$ is the Laplace-Beltrami operator  on $M$ \cite{MM87,GZ97,GUMA}.
\item Ruelle resonances for Anosov flows  \cite{BL07,FS11,GLP13,DZ16a}: here $P=iX$ with $X$ being the vector field generating the Anosov  flow.
\end{itemize} 
The definition of the resonances can be stated in different ways (using meromorphically continued resolvents, scattering operators or discrete spectra on auxiliary function spaces), and also the mathematical techniques used to establish the existence of resonances in the above examples are quite diverse (ranging from asymptotics of special functions to microlocal analysis). 
Nevertheless, all three examples above share the common point that the existence of a discrete resonance spectrum can be proven via a parametrix construction, i.e. one constructs a meromorphic family of operators $Q(\lambda)$ (with $\lambda\in \C$) such that 
\[
(P-\lambda)Q(\lambda) = \Id+K(\lambda),
\] 
where $K(\lambda)$ is a meromorphic family of compact operators on some suitable Banach or Hilbert space. Once such a parametrix is established, the resonances are the $\lambda$ where $\Id+K(\lambda)$ is not invertible and 
the discreteness of the resonance spectrum follows directly from analytic Fredholm theory. 

In general, being able to construct such a parametrix and define a theory of resonances involve non-trivial analysis and pretty strong assumptions, but they lead to powerful results on the long time dynamics of the propagator $e^{itP}$, for example in the study of dynamical systems \cite{Liv04,NZ13,FT17a} or on evolution equations in relativity \cite{HintzVasy}. Furthermore, resonances form an important spectral invariant that can be related to a large variety of other mathematical quantities such as geometric invariants \cite{GZ97, SZ07}, topological invariants \cite{DR17b, DZ17, DGRS18, KW20} or arithmetic quantities \cite{BGS11}. They also appear in trace formulas and are the divisors of dynamical Ruelle and Selberg zeta functions \cite{BO99, PP01, GLP13, DZ16a, FT17a}.
\\

The purpose of this work is to  use analytic and microlocal methods to construct a theory of joint resonance spectrum for the generating vector fields of $\R^\kappa$-Anosov actions. 
In terms of PDE and spectral theory, this can be viewed as the construction of a good notion of joint spectrum for a family of $\kappa$ commuting vector 
fields $X_1,\dots,X_\kappa$, generating a rank-$\kappa$ subbundle $E_0\subset TM$, when their flow is transversally hyperbolic with respect to that subbundle. These operators do not form an elliptic family and finding a good notion of joint spectrum is thus highly non-trivial. Our strategy is to work on anisotropic Sobolev spaces to make the non-elliptic region ``small'' and then obtain Fredholmness properties. \\

However, this involves working in a non self-adjoint setting, even if the $X_k$'s were to preserve a Lebesgue type measure. 
We are then using Koszul complexes and a cohomological theory developed by Taylor \cite{Tay70, Tay70a} in order to define a proper notion of joint spectrum in these anisotropic spaces, and we will show that this spectrum is discrete.

We emphasize that, in terms of PDE and spectral theory, there are important new aspects to be considered and the results are far from being a direct extension of the $\kappa=1$ case  (the Anosov flows). But also outside the spectral theory of linear partial differential operator the developed theory might be helpful:  
the classical examples of such  $\R^\kappa$-Anosov actions are \emph{Weyl chamber flows} for  compact locally symmetric spaces of rank $\kappa\geq 2$, and it is conjectured by Katok-Spatzier \cite{KaSp94} that essentially all non-product $\R^\kappa$ actions are smoothly conjugate to homogeneous cases. 
Despite important recent advances \cite{SV19}, the conjecture is still widely open and it is important to extract as much information as possible on a general Anosov $\R^\kappa$ action in order to address this conjecture: 
for example, having an ergodic invariant measure with full support plays an important role in the direction of this conjecture (see e.g. \cite{KS07} where the existence of such a measure is a central assumption on which the results are base; see also the discussions in the recent preprint \cite{SV19}). Based on the spectral theory developed in this article we show in a following paper \cite{BGW21} the existence of such ergodic measures of full support for any positively transitive\footnote{See \cite[Definition 2.9]{BGW21}.} Anosov action.

Let us summarize the main novelties of this work and its first applications: 
\begin{enumerate}
 \item we construct a new theory of joint resonance spectrum for a family of commuting differential operators by combining the theory of Taylor \cite{Tay70} 
with the use of anisotropic Sobolev spaces for the study of resonances; as far as we know, this is the first result of joint spectrum in the theory of classical or quantum resonances.
\item All the Weyl chamber flows on locally symmetric spaces and the standard actions of Katok-Spatzier \cite{KaSp94} are included in our setting, and 
our results are completely new in that setting where representation theory is usually one of the main tools. This gives a new, analytic, way of studying homogeneous dynamics and spectral theory in higher rank.
\item We show that the leading joint resonance provides a construction of a new Sinai-Ruelle-Bowen (SRB) invariant measure $\mu$ for all $\R^\kappa$ actions. In a companion paper \cite{BGW21} based on this work, we show that our measure $\mu$ has all the properties of SRB measures of Anosov flows (rank $1$ case), and it has \emph{full support} if the Weyl chamber is positively transitive, an important step in the direction of the rigidity conjecture. 
\item We show in \cite{BGW21} that the periodic tori of the $\R^\kappa$ action are equidistributed in the support of $\mu$ and that $\mu$ can be written as an infinite sum over Dirac measures on the periodic tori, in a way similar to Bowen's formula in rank $1$. These results are new even in the case of locally symmetric spaces and give a new way to study periodic tori (also called flats) in higher rank.
\item Based on the present paper, the two last named authors together with L.~Wolf prove a \emph{classical-quantum correspondence} between the joint resonant states of Weyl chamber flows on compact locally symmetric spaces $\Gamma\backslash G/M$ of rank $\kappa$ and the joint eigenfunctions of the commuting algebra of invariant differential operators on the locally symmetric space $\Gamma\backslash G/K$ \cite{HWW21}. This gives a higher rank version of \cite{DFG15}. 
\end{enumerate}

Another expected consequence of this construction would be a proof of the exponential decay of correlations for the action and a gap of Ruelle-Taylor resonances under appropriate assumptions, with application to the local rigidity and the regularity of the invariant measure $\mu$. These questions will be addressed in a forthcoming work.

\subsection{Statement of the main results}
Let us now introduce the setting in more details and state the main results.
Let $\M$ be a closed manifold, $\mathbb A\simeq \R^\kappa$ be an abelian group and let $\tau: \mathbb A\to \Diffeo(\M)$ be a smooth locally free group action. 
If $\mathfrak a:= \tu{Lie}(\mathbb A)\cong \R^\kappa$, we can define a \emph{generating map} 
\[
 X:\Abb{\mathfrak a}{C^\infty(\M;T\M)}{A}{X_A:=\frac{d}{dt}_{|t=0}\tau(\exp(tA)),}
\]
so that for each basis $A_1,\dots,A_\kappa$ of $\a$, $[X_{A_j},X_{A_k}]=0$ for all $j,k$. For $A\in \a$ we denote by $\varphi_t^{X_A}$ the flow of the vector field $X_A$. Notice that, as a differential operator, we can view $X$ as a map 
\[X:C^\infty(\M)\to C^\infty(\M;\a^*),\quad (Xu)(A):=X_{A}u.\] 

It is customary to call the action \emph{Anosov} if there is an $A\in \a$ such that there is a continuous $d\varphi_t^{X_A}$-invariant splitting 
\begin{equation}\label{splittingA}
 T\M=E_0\oplus E_u\oplus E_s,
\end{equation}
where $E_0={\rm span}(X_{A_1},\dots,X_{A_\kappa})$, and there exists a $C>0,\nu>0$ such that for each 
$x\in \M$ 
\begin{align*}
&\forall w\in E_s(x),\forall t\geq 0, \,\,   \|d\varphi_t^{X_A}(x)w\|\leq Ce^{-\nu t}\|w\|, \\
&\forall w\in E_u(x),\forall t\leq 0, \,\,   \|d\varphi_t^{X_A}(x)w\|\leq Ce^{-\nu |t|}\|w\|.
\end{align*}
Here the norm on $T\M$ is fixed by choosing any smooth Riemannian metric $g$ on $\mc{M}$. We say that  such an $A$ is \emph{transversely hyperbolic}.  
It can be easily proved that the splitting is invariant by the whole action. However, we do not assume that all $A\in \a^*$ have this transversely hyperbolic behavior. In fact, there is a maximal open convex cone $\W\subset \a$ containing $A$ such that for all $A'\in\W$, $X_{A'}$ is also transversely hyperbolic with the same splitting as $A$ (see Lemma~\ref{lem:weyl_chamber}); $\W$ is called a \emph{positive Weyl chamber}. 
This name is motivated by the classical examples of such Anosov actions that are the Weyl chamber flows for locally symmetric spaces of rank $\kappa$ (see Example \ref{locsymsp}). There are also several other classes of examples (see e.g. \cite{KaSp94, SV19}). 

Since we have now a family of commuting vector fields, it is natural to consider a joint spectrum for the family $X_{A_1},\dots,X_{A_\kappa}$ of first order operators if the $A_j$'s are chosen transversely hyperbolic with the same splitting. 
Guided by the case of a single Anosov flow (done in \cite{BL07,FS11,DZ16a}), we define $E_u^*\subset T^*\mc{M}$ to be the subbundle such that $E_u^*(E_u\oplus E_0)=0$. We shall say that $\lambda=(\lambda_1\dots,\lambda_\kappa)\in \C^\kappa$ is a \emph{joint Ruelle resonance} for the Anosov action if there is a non-zero distribution $u\in C^{-\infty}(\M)$ with wavefront set ${\rm WF}(u)\subset E_u^*$  such that\footnote{See \cite[Chapter VIII.1]{Hoe03} for a definition and properties of the wavefront set.}
\begin{equation}
 \label{eq:intro_def_joint_res}
\forall j=1,\dots,\kappa, \quad  (X_{A_j}+\lambda_j)u=0.\end{equation}
The distribution $u$ is called a \emph{joint Ruelle resonant state} (from now on we will denote $C^{-\infty}_{E_u^*}(\M)$ the space of distributions $u$ with ${\rm WF}(u)\subset E_u^*$). 
In an equivalent but more invariant way (i.e. independently of the choice of basis $(A_j)_j$ of $\a$), we can define a \emph{joint Ruelle resonance} as an element $\lambda\in \a_\C^*$ of the complexified dual Lie algebra such that there is a non-zero $u\in C^{-\infty}_{E_u^*}(\M)$ with 
\[
\forall A\in\a, \quad  (X_{A}+\lambda(A))u=0.
\] 
We notice that we  also define a notion of generalized joint Ruelle resonant states and Jordan blocks in our analysis (see Proposition \ref{prop:jordan_blocks}).
It is a priori not clear that the set of joint Ruelle resonances is discrete -- or non empty for that matter -- nor that the dimension of joint resonant states is finite, but this is a consequence of our work:
\begin{theorem}\label{Theo1intro}
Let $\tau$ be a smooth abelian Anosov action on a closed manifold $\M$ with positive Weyl chamber $\W$. Then the set of joint Ruelle resonances $\lambda \in\a^*_\C$ is a discrete set
contained in 
\begin{equation}\label{notaylorintro}
 \bigcap_{A\in \mathcal W} \{\lambda\in\a_\C^*~|~\Re(\lambda(A))\leq 0\}.
 \end{equation}
Moreover, for each joint Ruelle resonance $\lambda\in \a^*_\C$ the space of joint Ruelle resonant states is finite dimensional.
\end{theorem}
\begin{figure}
\begin{centering}
\includegraphics[width=0.8\textwidth]{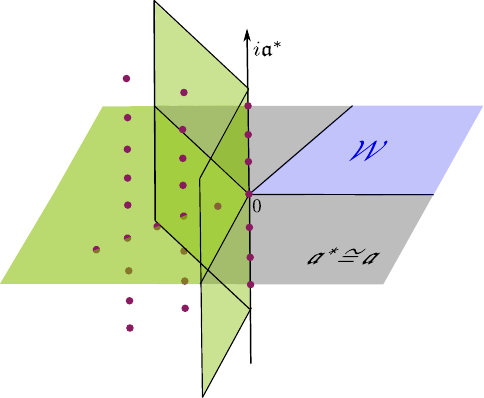}
\end{centering}
\caption{\label{fig:resonances} Schematic sketch of the location of the resonances in $\a^*_\C$. Note that in order to draw the at least 4 dimensional space $\a_\C^*$ the imaginary direction $i\a^*$ has been reduced in the drawing to a one dimensional line.
The blue cone depicts the positive Weyl chamber $\W$ 
and the green region illustrates the region \eqref{notaylorintro} in which the resonances can occur. The leading resonances discussed in Theorem~\ref{Theo3intro} are located at the tip of this region. } 
\end{figure}
We remark that this spectrum always contains $\lambda=0$ (with $u=1$ being the joint eigenfunction) and that for locally symmetric 
spaces it contains infinitely many joint Ruelle resonances, as is shown in \cite[Theorem 1.1]{HWW21}.

We also emphasize that this theorem is definitely not a straightforward extension of the case of a single Anosov flow. 
It relies on a deeper result based on the theory of joint spectrum and joint functional calculus developed by  Taylor \cite{Tay70, Tay70a}. 
This theory allows us to set up a good Fredholm problem on certain functional spaces by using Koszul complexes, as we now explain. 

Let us define $X+\lambda$, for $\lambda\in \a_\C^*$,  as an operator 
\[X+\lambda :C^\infty(\M)\to C^\infty(\M;\a_\C^*),\quad ((X+\lambda)u)(A):=(X_{A}+\lambda(A))u.\] 
We can then define for each $\lambda\in\a_\C^*$ the differential operators 
$d_{(X+\lambda)}: C^\infty(\M;\Lambda^j\a_\C^*)\to C^\infty(\M;\Lambda^{j+1}\a_\C^*)$ by setting
\[
d_{(X+\lambda)}(u\otimes\omega):= ((X+\lambda)u)\wedge\omega \textrm{ for }u\in C^\infty(\M),\, \omega\in \Lambda^j\a^*_\C.
\] 
Due to the commutativity of the family of vector fields $X_{A}$ for $A\in \a$, it can be easily checked that $d_{(X+\lambda)}\circ d_{(X+\lambda)}=0$ (see Lemma~\ref{lem:d_X_properties}). Moreover, as a differential operator, it extends to a continuous map
\[
d_{(X+\lambda)}: C^{-\infty}_{E_u^*}(\M;\Lambda^j\a_\C^*)\to C^{-\infty}_{E_u^*}(\M;\Lambda^{j+1}\a_\C^*)
\]
and defines an \emph{associated Koszul complex}
\begin{equation}\label{introtaylorcomplex}
  0\longrightarrow C^{-\infty}_{E_u^*}(\mc{M})\overset{d_{(X+\lambda)}}{\longrightarrow}C^{-\infty}_{E_u^*}\otimes \Lambda^1\a_\C^* \overset{d_{(X+\lambda)}}{\longrightarrow} \dots \overset{d_{(X+\lambda)}}{\longrightarrow} C^{-\infty}_{E_u^*}(\mc{M})\otimes \Lambda^\kappa \a_\C^*\longrightarrow 0,
\end{equation}
We prove the following results on the cohomologies of this complex: 
\begin{theorem}\label{Theo2intro}
Let $\tau$ be a smooth abelian Anosov action\footnote{We actually prove Theorem \ref{Theo1intro} and Theorem \ref{Theo2intro} in the more general setting of admissible lifts to vector bundles, as defined in Section \ref{admissiblelifts}.}
 on a closed manifold $\mc M$ with generating map $X$. Then for each $\lambda\in\a_\C^*$ and $j=0,\dots,\kappa$, the cohomology 
\[\ker d_{(X+\lambda)}|_{C^{-\infty}_{E_u^*}(\M)\otimes \Lambda^j\a^*_\C}/
\ran\, d_{(X+\lambda)}|_{C^{-\infty}_{E_u^*}(\M)\otimes \Lambda^{j-1}\a^*_\C}\]
is finite dimensional. It is non-trivial only at a discrete subset of $\{\lambda\in \a^*_\C\,|\, {\rm Re}(\lambda(A))\leq 0, \forall A\in \mc{W}\}$.
\end{theorem}
We want to remark that the statement about the cohomologies in Theorem~\ref{Theo2intro} is not only a stronger statement than Theorem~\ref{Theo1intro}, but that the cohomological setting is in fact a fundamental ingredient in proving the discreteness of the resonance spectrum and its finite multiplicity. Our proof relies on the theory of joint \emph{Taylor spectrum} (developed by J. Taylor in \cite{Tay70, Tay70a}), defined using such Koszul complexes carrying a suitable notion of Fredholmness. In our proof of Theorem~\ref{Theo2intro} we show that the Koszul complex furthermore provides a good framework for a parametrix construction via microlocal methods. More precisely, the parametrix construction is not done on the topological vector spaces $C^{-\infty}_{E_u^*}(\mc M)$ but on a scale of  Hilbert spaces $\mc H_{NG}$, depending on the choice of an escape function $G\in C^\infty(T^*\mc{M})$ and a parameter $N\in \R^+$, by which one can in some sense approximate $C^{-\infty}_{E_u^*}(\mc{M})$. 
The spaces $\mc H_{NG}$ are \emph{anisotropic Sobolev spaces} which roughly speaking allow $H^{N}(\mc{M})$ Sobolev regularity 
in all directions except in $E_u^*$ where we allow for $H^{-N}(\mc{M})$ Sobolev regularity. They can be rigorously defined using microlocal analysis, following the techniques of Faure-Sj\"ostrand \cite{FS11}. 
By further use of pseudodifferential and Fourier integral operator theory we can then construct a parametrix $Q(\lambda)$, which is a family of bounded operators on $\mathcal H_{NG}\otimes\Lambda \a_\C^*$ depending holomorphically on $\lambda\in\a^*_\C$ and fulfilling
\begin{equation}
 \label{eq:parametrix_intro}
d_{(X+\lambda)}Q(\lambda) + Q(\lambda)d_{(X+\lambda)} = \Id + K(\lambda).
\end{equation}
Here $K(\lambda)$ is a holomorphic family of compact operators on $\mc H_{NG}\otimes\Lambda\a_\C^*$ for $\lambda$ in a suitable domain of $\a_\C^*$ that can be made arbitrarily large letting $N\to\infty$. 
Even after having this parametrix construction, the fact that the joint spectrum is discrete and intrinsic (i.e. independent of the precise construction of the Sobolev spaces) is more difficult than for an Anosov flow (the rank $1$ case): 
this is because holomorphic  functions in $\C^\kappa$ do not have discrete zeros when $\kappa\geq 2$ and we are lacking a good notion of resolvent, while for one operator the resolvent is an important tool. 
Due to the link with the theory of the Taylor spectrum, we call $\lambda\in \a^*_\C$ a \emph{Ruelle-Taylor resonance} for the Anosov action if for some $j=0,\ldots,\kappa$ the $j$-th cohomology is non-trivial
\[
\ker d_{(X+\lambda)}|_{C^{-\infty}_{E_u^*}(\M)\otimes \Lambda^j\a^*_\C}/
\ran\, d_{(X+\lambda)}|_{C^{-\infty}_{E_u^*}\otimes \Lambda^{j-1}\a^*_\C}\not=0,
\]
and we call the non-trivial cohomology classes \emph{Ruelle-Taylor resonant states}. Note that the definition of joint Ruelle resonances precisely means that the $0$-th cohomology is non-trivial. Thus, any joint Ruelle resonance is a Ruelle-Taylor resonance. The converse statement is not obvious but turns out to be true, as we will prove in Proposition~\ref{prop:zero_cohomology_nonempty}: if the cohomology of degree $j>0$ is not $0$, 
then the cohomology of degree $0$ is not trivial.

We continue with the discussion of the \emph{leading resonances}. In view of \eqref{notaylorintro} and Figure~\ref{fig:resonances}, a resonance is called a leading resonance when its real part vanishes. We show that this spectrum carries important information about the dynamics: it is related to a special type of invariant measures as well as to mixing properties of these measures.

First, let $v_g$ be the Riemannian measure of a fixed metric $g$ on $\mc{M}$.
We call a $\tau$-invariant probability measure $\mu$ on $\M$, a \emph{physical measure} if there is $v\in C^\infty(\mc{M})$ non-negative such that for any continuous function $f$ and any  open cone $\mathcal{C}\subset \W$,
\begin{equation}\label{eq:Cone-averaging-physical-measure}
\mu(f) =\lim_{T\to \infty} \frac{1}{{\rm Vol}(\mc{C}_T)} \int_{A\in \mc{C}_T} \int_\M f(\varphi^{-X_A}_1(x)) v(x) dv_g(x) dA
\end{equation}
where $\mc{C}_T:=\{A\in \mc{C}\, |\, |A|\leq T\}$, and here $|\cdot|$ denotes a fixed Euclidean norm on $\a$. 
In other words, $\mu$ is the weak Cesaro limit of a Lebesgue type measure under the dynamics.
We prove the following result:
\begin{theorem}\label{Theo3intro}
Let $\tau$ be a smooth abelian Anosov action with generating map $X$ and let $\W$ be a positive  Weyl chamber. 
\begin{enumerate}[(i)]
\item \label{it:physM_are_res_states} The linear span over $\C$ of the physical measures is isomorphic (as a $\C$ vector space) to $\ker d_X|_{C^{-\infty}_{E_u^*}}$, the space of joint Ruelle resonant states at $\lambda=0\in \a^*_\C$; in particular, it is finite dimensional.\footnote{The dimension can more concretely be expressed in dynamical terms, see\cite[Theorem 3]{BGW21}.}
\item \label{it:WF_phys_meas} A probability measure $\mu$ is a physical measure if and only if it is $\tau$-invariant and $\mu$ has wavefront set
${\rm WF}(\mu)\subset E_s^*$, where $E_s^*\subset T^*\M$ is defined by $E_s^*(E_s\oplus E_0)=0$.
\item Assume that there is a unique physical measure $\mu$ (or by (\ref{it:physM_are_res_states}) equivalently that the space of joint resonant states at 0 is one dimensional). Then the following are equivalent:
\begin{itemize} 
 \item The only Ruelle-Taylor resonance on $i\a^*$ is zero.
 \item There exists $A\in \a$ such that $\varphi_t^{X_A}$ is \emph{weakly} mixing with respect to $\mu$.
 \item For any $A\in \mc{W}$, $\varphi_t^{X_A}$ is \emph{strongly} mixing with respect to $\mu$.
\end{itemize}
\item \label{it:complex_measures} $\lambda\in i\a^*$ is a joint Ruelle resonance if and only if there is a complex measure $\mu_\lambda$ with 
${\rm WF}(\mu_\lambda)\subset E_s^*$ satisfying for all $A\in \W, t\in \R$ the following equivariance under push-forwards of the action: $(\varphi^{X_{A}}_t)_*\mu_\lambda=e^{-\lambda(A)t}\mu_\lambda$. Moreover, such measures are absolutely continuous with respect to the physical measure obtained by taking $v=1$ in \eqref{eq:Cone-averaging-physical-measure}.
\item \label{it:res_at_zero} If $\mc M$ is connected and if there exists a smooth invariant measure $\mu$ with $\supp(\mu)=\mc M$, we have for any $j=0,\ldots,\kappa$
\[
\dim \left(\ker d_{X}|_{C^{-\infty}_{E_u^*}(\M)\otimes \Lambda^j\a^*_\C}/
\ran\, d_{X}|_{C^{-\infty}_{E_u^*}\otimes \Lambda^{j-1}\a^*_\C}\right)={\binom{\kappa}{j}}.
\] 
\end{enumerate}
\end{theorem}
We show that the isomorphism stated in (\ref{it:physM_are_res_states}) and the existence of the complex measures in (\ref{it:complex_measures}) can  be constructed explicitly in terms of spectral projectors built from the parametrix \eqref{eq:parametrix_intro}. We refer to Propositions~\ref{prop:SRBmeasures} and \ref{prop:mixing} for these constructions and for slightly more complete statements.

In the case of a single Anosov flow, physical measures are known to coincide with SRB measures (see e.g. \cite{You02} and references therein). The latter are usually defined as invariant measures that can locally be disintegrated along the stable or unstable foliation of the flow with absolutely continuous conditional densities.

We prove in a subsequent article \cite{BGW21} that the microlocal characterization Theorem~\ref{Theo3intro}(\ref{it:WF_phys_meas}) of physical measures via their wavefront set implies that the physical measures of an Anosov action are exactly those invariant measures that allow an absolutely continuous disintegration along the stable manifolds. We show in \cite[Theorem 2]{BGW21} that for each physical/SRB measure, there is a basin $B\subset \mc{M}$ of positive Lebesgue measure such that for all $f\in C^0(\M)$, all proper open subcones $\mathcal{C}\subset\W$ and all $x\in B$, we have the convergence
\begin{equation}\label{eq:Cone-averaging-ergodic-SRB-measure}
\mu(f) =\lim_{T\to \infty} \frac{1}{{\rm Vol}(\mc{C}_T)} \int_{A\in \mc{C}_T}f(\varphi^{-X_A}_1(x)) dA.
\end{equation}
Moreover, we prove in \cite[Theorem 3]{BGW21} that the measure $\mu$ can be written as an infinite weighted sums over Dirac measures on the periodic tori of the action, showing an equidistribution of periodic tori in the support of $\mu$.  
Finally, we show that this measure has full support in $\mc{M}$ if the action is positively transitive in the sense that there is a dense orbit $\cup_{A\in \mc{W}} \varphi_1^{X_A}(x)$ for some $x\in \mc{M}$. As mentioned before, the existence of such a measure is considered as one important step towards the resolution of the \cite{KaSp94} rigidity conjecture.

\subsection{Relation to previous results}
The notion of resonances for certain particular Anosov flows appeared in the work of Ruelle \cite{Rue76}, and was later extended by Pollicott \cite{Pol85}. The introduction of a spectral approach based on anisotropic Banach and Hilbert spaces came later and allowed to the definition of resonances in the general setting, first for Anosov/Axiom A diffeomorphisms \cite{BKL02,GL06,BT07,FRS08}, then for general Anosov/Axiom A flows \cite{Liv04,BL07,FS11,GLP13, DZ16a,DG16,Me21}. It was also applied to the case of pseudo Anosov maps \cite{FGL19}, Morse-Smale flows \cite{DR18}, geodesic flows for manifolds with cusps \cite{BW21} and billiards \cite{BDL18}. This spectral approach has been used to study SRB measures \cite{BKL02, BL07} but it led also to several important consequences on dynamical zeta function \cite{GLP13, DZ16a, FT17, DG16} of flows, and links with topological invariants \cite{DZ17, DR17b, DGRS18}.

Concerning the notion of joint spectrum in dynamics, there are several cases that have been considered but they correspond to a different context of systems with symmetries (e.g. \cite{BR01}).

Higher rank $\R^\kappa$-Anosov actions have in particular been studied mostly for their rigidity: they are conjectured to be always smoothly conjugated to several models, mostly of algebraic nature (see e.g. the introduction of \cite{SV19} for a precise statement and a state of the art on this question). 
The local rigidity of $\R^\kappa$-Anosov actions near \emph{standard Anosov actions}\footnote{This class, defined in \cite{KaSp94}, consists of Weyl chamber flows associated to rank $\kappa$ locally symmetric spaces and variations of those.} was proved in \cite{KaSp94}, and an important step of the proof relies on showing
\[
\ker d_X|_{C^\infty(\M)\otimes \Lambda^1\a_\C^*}/\ran\, d_X|_{C^\infty(\M)}=\C^\kappa.
\]
The main tools are based on representation theory to prove fast mixing with respect to the canonical invariant (Haar) measure.
It is also conjectured in \cite{KaKa} that, more generally, for such standard actions, one has  for $j=1,\dots,\kappa-1$
\[
\ker d_X|_{C^\infty(\M)\otimes \Lambda^j\a_\C^*}/\ran\, d_X|_{C^\infty(\M)\otimes \Lambda^{j-1}\a_\C^*}=\C^{\binom{\kappa}{j}}.
\]
This can be compared to (\ref{it:res_at_zero}) in Theorem \ref{Theo3intro}, except that there the functional space is different. 
Having a notion of Ruelle-Taylor resonances provides an approach to obtain exponential mixing for more general Anosov actions by generalizing microlocal techniques for spectral gaps \cite{NZ13, Tsu10} to a suitable class of higher rank Anosov action, and by using the functional calculus of Taylor \cite{Tay70a, Vas79}. We believe that such tools might be very useful to obtain new results on the rigidity conjecture.

We would like to conclude by pointing out a different direction: on rank $\kappa>1$ locally symmetric spaces $\Gamma\backslash G/K$, there is a commuting algebra of invariant differential operators that can be considered as a quantum analog to the Weyl chamber flows. If the locally symmetric space is compact, this algebra always has a discrete joint spectrum of $L^2$-eigenvalues. Its joint spectrum and relations to trace formulae have been studied in \cite{DKV79}. In \cite{HWW21}, it is shown that a subset of the Ruelle-Taylor resonances for the Weyl chamber flow are in correspondence with the joint discrete spectrum of the invariant differential operators on $\Gamma\backslash G/K$, giving a generalization of the classical/quantum correspondence of \cite{DFG15, GHW21, KW21} to higher rank.

\subsection{Outline of the article} In Section~\ref{sec:geom_prelim} we introduce the geometric setting of Anosov actions and the admissible lifts that we study. In Section~\ref{sec:Taylor-spectrum} we explain how to define the Taylor spectrum for a certain class of unbounded operators and discuss some properties of this Taylor spectrum. In Section~\ref{sec:Taylor-pseudors} we prove Theorem~\ref{Theo1intro} and Theorem \ref{Theo2intro}, using microlocal analysis. A sketch of the central techniques is given at the beginning of Section~\ref{sec:Taylor-pseudors}. The last Section~\ref{sec:resonance-at-zero} is devoted to the proof of Theorem \ref{Theo3intro}. In Appendix~\ref{sec:microlocal}, we recall some classical results of microlocal analysis needed in the paper.\\

\textbf{Acknowledgements.}
This project has received funding from the European Research Council (ERC) under the European Union’s Horizon 2020 research and innovation programme (grant agreement No. 725967) and from the Deutsche Forschungsgemeinschaft (DFG) through the Emmy Noether group ``Microlocal Methods for Hyperbolic Dynamics''(Grant No. WE 6173/1-1). 
This material is based upon work supported by the National Science Foundation under Grant No. DMS-1440140 while C.G. was in residence at the Mathematical Sciences Research Institute in Berkeley, California, during the Fall 2019 semester. 
We thank F. Rodriguez Hertz, A. Brown  and R. Spatzier for useful discussions, as well as B. K\"uster and L. Wolf for valuable feedback on an earlier version of the manuscript.

\section{Geometric preliminaries}\label{sec:geom_prelim}
\subsection{Anosov actions}
We first want to explain the geometric setting of Anosov actions and the admissible lifts that we will study. 

Let $(\M, g)$ be a closed, smooth Riemannian manifold (normalized with volume $1$) equipped with a smooth locally 
free action $\tau : \mathbb A \to \Diffeo(\M)$ for an abelian Lie group 
$\mathbb A\cong \R^\kappa$. Let $\mathfrak a:= \tu{Lie}(\mathbb A)\cong \R^\kappa$ be the associated commutative Lie algebra and $\exp:\mathfrak a\to\mathbb A$ the Lie group exponential map.
After identifying $\mathbb A\cong \mathfrak a\cong \R^\kappa$, this exponential map is simply the identity, but it will be  quite useful to have a notation that distinguishes between transformations and infinitesimal transformations. 
Taking the derivative of the $\mathbb A$-action one obtains the infinitesimal action, called an \emph{$\a$ action}, which is an injective Lie algebra homomorphism 
\begin{equation}\label{XAnosov}
 X:\Abb{\mathfrak a}{C^\infty(\M;T\M)}{A}{X_A:=\frac{d}{dt}_{|t=0}\tau(\exp(At)).}
\end{equation}
Note that $X$ can alternatively be seen as a Lie algebra morphism into the space ${\rm Diff}^1(\M)$ of first order differential operators.
By commutativity of $\mathfrak a$, $\ran (X)\subset C^\infty(\M;T\mc M)$ is a $\kappa$-dimensional subspace of commuting vector fields which span a $\kappa$-dimensional smooth subbundle which we call the \emph{neutral subbundle} $E_0\subset T\mc M$. 
Note that this subbundle is tangent to the $\mathbb A$-orbits on $\mc M$. It is often useful to study the one-parameter flow generated by a vector field $X_A$ which we denote by $\varphi^{X_A}_t$. One has the obvious identity $\varphi^{X_A}_t = \tau(\exp(At))$ for $t\in \R$. 
The Riemannian metric on $\M$ induces norms  on $T\M$ and $T^*\M$, both denoted by $\|\cdot\|$.
\begin{Def}\label{transhyp}
An element $A\in \mathfrak a$ and its corresponding vector field $X_A$ are called \emph{transversely hyperbolic} if there is a continuous splitting
\begin{equation}\label{eq:invariant-splitting}
 T\M = E_0\oplus E_u\oplus E_s,
\end{equation}
that is invariant under the flow $\varphi^{X_A}_t$ and such that there are $\nu>0, C>0$ with 
\begin{equation}
 \label{eq:stable}
 \|d\varphi_t^{X_A}v\|\leq Ce^{-\nu |t|}\|v\|, \quad  \forall v\in E_s,~ \forall t\geq 0,
\end{equation}
\begin{equation}\label{eq:unstable}
 \|d\varphi_t^{X_A}v\|\leq Ce^{-\nu |t|}\|v\|, \quad  \forall v\in E_u,~ \forall t\leq 0.
\end{equation}
We say that the $\mathbb A$-action is \emph{Anosov} if there exists an $A_0\in \mathfrak a$ such that $X_{A_0}$ is transversely hyperbolic.
\end{Def}
Given a transversely hyperbolic element $A_0\in \mathfrak a$ we define the \emph{positive Weyl chamber} $\mathcal W\subset \mathfrak a$
to be the set of  $A \in \mathfrak a$  which are transversely hyperbolic with the same stable/unstable bundle as $A_0$. 
\begin{lemma}\label{lem:weyl_chamber}
Given an Anosov action and a transversely hyperbolic element $A_0\in \a$, the positive Weyl chamber $\mathcal W\subset \mathfrak a$ is an open convex cone.
\end{lemma}
\begin{proof}
Let us first take the $\varphi^{X_{A_0}}_t$-invariant splitting $E_0\oplus E_u\oplus E_s$ and show that it is in fact invariant under the Anosov action $\tau$: let $v \in E_u$ and  $A\in \mathfrak a$ . Using $[X_{A_0},X_{A}]=0$, for each $t_0\in \R$ fixed and all $t\in\R$ we find
\begin{equation}
d\varphi_{-t}^{X_{A_0}}d\varphi_{t_0}^{X_A}v = d\varphi_{t_0}^{X_A}d\varphi_{-t}^{X_{A_0}}v.
\end{equation}
In particular, $\|d\varphi_{-t}^{X_{A_0}}d\varphi_{t_0}^{X_A}v\|$ decays exponentially fast as $t\to + \infty$. This implies that $d\varphi_{t_0}^{X_A}v\in E_u$ and the same argument works with $E_s$. 
Next, we choose an arbitrary norm on $\mathfrak a$. There exist $C,C'>0$ such that for each $v\in E_u$ we have for $t\geq 0$
\[
\|d\varphi_{-t}^{X_A}v\| \leq \|d\varphi_{-t}^{X_{A-A_0}}d\varphi_{-t}^{X_{A_0}}v\|\leq 
C\|v\|e^{-\nu t}\|d\varphi_{-t}^{X_{A-A_0}}\|\leq C\|v\|e^{-\nu t}e^{C't\|A-A_0\|}.
\]
This implies that by choosing $\|A-A_0\|$ small enough, $E_u$ is an unstable 
 bundle for $A$ as well. The same construction works for $E_s$ and we have thus shown that $\mathcal W$ is open.

By re-parametrization, it is clear that $\W$ is a cone, so that only the convexity is left to be proved.
Now, take $A_1, A_2 \in \W$ and let $C_1,\nu_1,C_2,\nu_2$ be the corresponding constants for the transversal hyperbolicity estimates \eqref{eq:stable} and \eqref{eq:unstable}. Then for $s\in [0,1]$ and $v\in E_u$ we can again use the commutativity and obtain,
\begin{equation}
\| d\varphi_{-t}^{X_{s A_1 + (1-s) A_2}} v \| \leq C_1 C_2 e^{-\nu_1 s t - \nu_2 (1-s) t} \|v\|
\end{equation}
and this shows that $sA_1 + (1-s) A_2 \in \mathcal W$.
\end{proof}
Here we emphasize that the Weyl chamber $\mc{W}$ only depends on the Anosov splitting 
associated to $A_0$ but not on $A_0$ itself. Notice also that in general there are other Weyl chambers $\mc{W}'$ associated to a different Anosov 
splitting. In the standard example of Weyl chamber flows they are images of $\mc{W}$ by the Weyl group of the higher rank locally symmetric space, explaining the terminology Weyl chambers (see for example \cite{HWW21} for details). In general the structure of Weyl chambers can be quite complicated (see for example the example of non total Anosov actions given in \cite[Section 6.3.4.]{SV19}). In that case, the Ruelle-Taylor spectrum that we shall define has no reason to be the same for $\mc{W}$ and for $\mc{W}'$.

There is an important class of examples given by the Weyl chamber flow on Riemannian locally symmetric spaces.
\begin{example}\label{locsymsp}
Consider a real semi-simple Lie group $\mathbb G$, connected and of non-compact type, and let $\mathbb G=\mathbb K\mathbb A\mathbb N$ be an Iwasawa decomposition with $\mathbb A$ abelian, 
$\mathbb K$ the compact maximal subgroup and $\mathbb N$ nilpotent. Then $\mathbb A\cong \R^\kappa$ and $\kappa$ is called the \emph{real rank} of $\mathbb G$. Let $\a$ be the Lie algebra of $\mathbb A$ and consider the adjoint action of $\mathfrak a$ on $\mathfrak g$ which leads to the definition of a finite set of \emph{restricted roots} $\Delta\subset \a^*$. For $\alpha \in \Delta$ let $\mathfrak g_\alpha$ be the associated root space. It is then possible to choose a set of positive roots $\Delta_+\subset\Delta$ and with respect to this choice there is an algebraic definition of a positive Weyl chamber 
\[
\mathcal W:= \{A\in\a\, |\, \alpha(A)>0 \text{ for all } \alpha\in\Delta_+\}.
\]
If one now considers $\Gamma < \mathbb{G}$ a torsion free, discrete, co-compact subgroup one can define the biquotient  $\M := \Gamma \backslash \mathbb G / \mathbb M$ where $\mathbb M\subset \mathbb K$ is the centralizer of $\mathbb{A}$ in $\mathbb K$. As $\mathbb A$ commutes with $\mathbb M$, the space $\mathcal M$ carries a right $\mathbb A$-action. Using the definition of roots, it is direct to see that this is an Anosov action: all elements of the positive Weyl chamber $\mathcal W$ are transversely hyperbolic elements sharing the same stable/unstable distributions given by the associated vector bundles:
 \[
  E_0 = \mathbb{G}\times_{\mathbb{M}}\a,~~E_s=\mathbb{G}\times_\mathbb{M}\n,~~E_u = \mathbb{G}\times_\mathbb{M} \bbar{\n}.
 \]
Here $\n:=\sum_{\alpha\in\Delta_+}\g_\alpha$ and $\bbar\n:=\sum_{-\alpha\in\Delta_+}\g_\alpha$ are the sums of all positive, respectively negative root spaces, and $\n$ coincides with the Lie algebra of the nilpotent group $\N$. 
\end{example}
Note that there are various other constructions of Anosov actions and we refer to \cite[Section 2.2]{KaSp94} for further examples.

\subsection{Admissible lifts}\label{admissiblelifts}
We want to establish the spectral theory not only for the commuting vector fields $X_A$ that act as first order differential operators on $C^\infty(\mc M)$ but also for first order differential operators on Riemannian vector bundles $E\to\mc M$ which lift the Anosov action.
\begin{Def}\label{def:admissible_lift}
 Let $\mc M$ be a closed manifold with an Anosov action of $\mathbb A\cong \R^\kappa$ and generating map $X$. Let $E\to\mc M$ be the complexification of a smooth Riemannian vector bundle over $\mc M$. Denote by ${\rm Diff}^1(\M;E)$ the Lie-algebra of first order differential operators with smooth coefficients and scalar principal symbol, acting on sections of $E$. Then we call a Lie algebra homomorphism 
 \[
  \X:\a\to {\rm Diff}^1(\M;E),
 \]
an \emph{admissible lift} of the Anosov action if it satisfies the following Leibniz rule: for any section $s\in C^\infty(\M;E)$ and any function $f\in C^\infty (\mc M)$ one has  for all $A\in \mathfrak{a}$
\begin{equation}\label{eq:admissible_lift_assumption}
 {\bf  X}_A (fs) = (X_Af)s + f\X_A s.
\end{equation}
\end{Def}
A typical example to have in mind would be when $E$ is a tensor bundle, (e.g.  exterior power of the cotangent bundle $E=\Lambda^{m}T^*\M$ or symmetric tensors $E=\otimes^m_ST^*\M$), and 
\[
\X_A s:= \mc{L}_{X_A}s
\] 
where $\mc{L}$ denotes the Lie derivative. This admissible lift can be restricted to any subbundle that is invariant under the differentials $d\varphi^{X_A}_t$ for all $A\in \a, t>0$. Another class of examples comes from flat connections.
More generally, the above examples can be seen as a special case where the $\mathbb A$-action $\tau$ on $\mc M$ lifts to an action $\widetilde{\tau}$ on $E$ which is fiberwise linear. Then one can define an infinitesimal action
 \begin{equation}\label{liftedaction}
  \X_As(x) := \partial_t\tilde{\tau}(\exp(-At))s(\tau(\exp(At))x)|_{t=0}
 \end{equation}
 which is an admissible lift. 

\section{Taylor spectrum and Fredholm complex}
\label{sec:Taylor-spectrum}

The Taylor spectrum was introduced by Taylor in \cite{Tay70, Tay70a} as a joint spectrum for commuting bounded operators, using the theory of Koszul complexes. 
While there are different competing notions of joint spectra (see e.g. the lecture notes \cite{Cur88}), the Taylor spectrum is from many perspectives the most natural notion. 
Its attractive feature is that it is defined in terms of operators acting on Hilbert spaces 
and does not depend on a choice of an ambient commutative Banach algebra.
Furthermore, it comes with a satisfactory analytic functional calculus developed by Taylor and Vasilescu \cite{Tay70a, Vas79}.

\subsection{Taylor spectrum for unbounded operators}\label{sec:Taylor_commuting}

Most references introduce the Taylor spectrum for tuples of bounded operators. In our case, we need to deal with unbounded operators. Additionally, working with a tuple implies choosing a basis, which should not be necessary. 
Let us thus explain how the notion of Taylor spectrum can easily be extended to an important class of abelian actions by unbounded operators.

We start with $E\to \mc M$ a smooth complex vector bundle over a smooth manifold $\mc M$ (not necessarily compact), $\a\cong \R^\kappa$ an abelian Lie algebra
and $\X:\a\to \OpDiff^1(\mc M;E)$ a Lie algebra morphism.  
For the moment we do not have to assume that $\mc M$ possesses an Anosov action. 
Note that $\X$ extends by linearity to  $\X:\a_\C\to \OpDiff^1(\mc M;E)$ and for the definition of the spectra we will need to work with this complexified version. Using the map $\X$ we define
\[
 d_{\X}:\Abb{C_c^\infty(\M;E)}{ C_c^\infty(\M;E)\otimes \a^*_\C}{u}{\X u}
\]
where we have set $(\X u)(A):=\X_{A} u$ for each $A\in \a_\C$. 
This will be the central ingredient to define the Koszul complex which will lead to the definition of the Taylor spectrum. In order to do this we need some more notation: we denote by $\Lambda\acs :=\bigoplus_{\ell=0}^\kappa \Lambda^\ell\acs$  the exterior algebra of $\acs$ --- this is just a coordinate-free version of $\Lambda \C^\kappa$. 
Given a topological vector space $V$ we use the shorthand notation $V\Lambda^\ell:= V\otimes \Lambda^\ell \a_\C^*$ and $V\Lambda:=  V \otimes \Lambda\acs$. As $\Lambda\acs$ is finite dimensional $V\Lambda$ is again a topological vector space. 
We notice that since $\Lambda\acs$ is a finite dimensional vector space, we can view it as a trivial bundle $\mc{M}\times \Lambda\acs\to \mc{M}$, and when $V=C_c^\infty(\mc{M};E)$, $V=L^p(\mc{M};E)$ or $V=\mc{D}'(\mc{M};E)$,  elements in $V\otimes \Lambda^\ell \a_\C^*$ can be identified respectively to sections of $V=C_c^\infty(\mc{M};E\otimes \Lambda \a_\C^*)$, $L^p(\mc{M};E\otimes \Lambda \a_\C^*)$ 
or $\mc{D}'(\mc{M};E\otimes \Lambda \a_\C^*)$. We shall freely make this identitifcation as this will  sometime be useful when we use pseudo-differential operators.

We have the  \emph{contraction} and \emph{exterior product} maps
\[
 \iota:\Abb{\a_\C\times V\Lambda^\ell }{V\Lambda^{\ell-1}}{(A,v\otimes \omega)}{\iota_A(v\otimes \omega) := v\otimes (\iota_A \omega)} \text{ and }
 ~\wedge:\Abb{V\Lambda^\ell \times\Lambda^r\acs}{V\Lambda^{\ell+r}}{(v\otimes \omega,\eta)}{v\otimes(\omega\wedge\eta).}
\]
We can then extend $d_\X$ to a continuous map on the spaces $C_c^\infty \Lambda:=C_c^\infty(\M;E)\otimes \Lambda\a^*_\C$ (resp. $C^{-\infty}\Lambda:=C^{-\infty}(\M;E)\otimes \Lambda\a^*_\C$) 
by setting for each $u\in C_c^\infty(\M;E)$ (resp. $u\in C^{-\infty}(\M;E)$) and  
$\omega \in \Lambda^\ell\acs$
\[
 d_\X:u\otimes \omega\mapsto (d_\X u)\wedge \omega.
\]
Similarly, for each $A\in\a$ we will also extend $\X_A$ on these spaces by setting
\[
\X_A(u\otimes \omega):=\X_Au\otimes \omega=(\iota_{A}d_{\X}u)\otimes \omega.
\]
\begin{rem}\label{rem:coordinate_free_and_corrdinates}
Choosing a basis $A_1,\ldots,A_\kappa \in\a$ provides an isomorphism 
$\Lambda \a^* \cong \Lambda \R^\kappa$. One checks that under this isomorphism the coordinate free version $d_\X: V\otimes \Lambda^\ell\a^*\to V\otimes \Lambda^{\ell+1}\a^*$ of the Taylor differential transforms to the Taylor differential $d_X: V\otimes \Lambda^\ell\R^\kappa
\to V\otimes \Lambda^{\ell+1}\R^\kappa$ of the operator tuple $X=(\mathbf X_{A_1},\ldots,\mathbf X_{A_\kappa})$ defined as 
\begin{equation}\label{eq:def-d_Y-coordinates}
d_X ( u\otimes e_{i_1}\wedge\dots\wedge e_{i_j}) := \sum_{k=1}^\kappa (\mathbf X_{A_k}u)\otimes e_k \wedge e_{i_1}\wedge\dots\wedge e_{i_j}
\end{equation}
if the basis $(e_j)_j$ of $\R^\kappa$ is identified to the dual basis of $(A_j)_j$ in $\a^*$. 
\end{rem}

\begin{lemma}\label{lem:d_X_properties}
 For each $A\in \a_\C$ one has the following identities as continuous operators on 
 $C_c^\infty \Lambda$ and $C^{-\infty}\Lambda$:
 \begin{enumerate}[(i)]
  \item $\iota_A d_\X+d_\X\iota_A= \X_A,$
  \item $\X_A d_\X = d_\X \X_A,$
  \item $d_\X d_\X =0.$
 \end{enumerate}
\end{lemma}
\begin{proof}
 Let $u\otimes \omega\in C_c^\infty\Lambda$ or $u\otimes\omega\in C^{-\infty}\Lambda$. Then by definition
 \[
 \iota_A d_\X(u\otimes \omega) = \iota_A ((d_\X u)\wedge \omega) = 
 (\X_Au)\otimes \omega - d_\X u\wedge (\iota_A \omega) = (\X_A - d_\X\iota_A)(u\otimes \omega)
 \]
 which yields (i). In order to prove (ii) it suffices, by definition of $d_\X$, to prove the identity as a map $C_c^\infty(\M;E) \to C_c^\infty(\M;E)\otimes \a^*_\C$. 
 Take arbitrary $A, A'\in \a_\C$ and note that by definition $\iota_{A'}\mathbf X_A = \mathbf X_A\iota_{A'}$. Then for $u\in C_c^\infty(\M;E)$ we get 
 \[
  \iota_{A'}(\X_A d_\X - d_\X\X_A)u = (\X_{A} \X_{A'} - \X_{A'}
  \X_A) u = 0
 \]
 which proves the statement. Note that we crucially use the commutativity of the differential operators $\X_A$ in this step. 
 
 For (iii) we first conclude from (i) and (ii) that $\iota_A d_\X d_\X= d_\X d_\X\iota_A$. Using this identity we deduce that for $u\in C_c^\infty \Lambda^\ell$ and arbitrary $A_1,\ldots A_{\ell + 1}\in \a_\C$
 \[
  \iota_{A_1}\ldots\iota_{A_{\ell+1}} d_\X d_\X  u =0,
 \]
which implies $d_\X d_\X u=0$.
\end{proof}
As a direct consequence of Lemma~\ref{lem:d_X_properties}(iii) we conclude that 
\begin{equation}\label{eq:def_smooth_complex}
  0\longrightarrow C_c^\infty\Lambda^0\overset{d_\X}{\longrightarrow}C_c^\infty\Lambda^1\overset{d_\X}{\longrightarrow} \dots \overset{d_\X}{\longrightarrow} C_c^\infty\Lambda^\kappa\longrightarrow 0
\end{equation}
and 
\begin{equation}\label{eq:def_distributional_complex}
  0\longrightarrow C^{-\infty}\Lambda^0\overset{d_\X}{\longrightarrow} C^{-\infty}\Lambda^1\overset{d_\X}{\longrightarrow} \dots \overset{d_\X}{\longrightarrow}  C^{-\infty}\Lambda^\kappa\longrightarrow 0
\end{equation}
are complexes. 

We now want to construct a complex of bounded operators on Hilbert spaces which lies between the complexes on $C_c^\infty\Lambda$ and $C^{-\infty}\Lambda$. For this, we consider $\mathcal{H}$ a Hilbert space with continuous embeddings $C_c^\infty(\M; E) \subset \mc H \subset C^{-\infty}(\mc M; E)$ such that $C_c^\infty(\M; E)$ is a dense subspace of $\mc H$. 
If we fix a non-degenerate Hermitian inner product $\langle \cdot,\cdot\rangle_{\acs}$, then this induces a scalar product $\langle\cdot,\cdot\rangle_{\mc H\Lambda}$ and gives a Hilbert space structure on $\mc H \Lambda$. 
While the precise value of $\langle\cdot,\cdot\rangle_{\mc H\Lambda}$ obviously depends on the choice of the Hermitian product on $\acs$, the finite dimensionality of $\acs$ implies that all Hilbert space structures on $\mc H\Lambda$ obtained in this way are equivalent. 
Note that on the Hilbert spaces $\mc H\Lambda^\ell$ the operators $d_\X$ will in general be unbounded operators. However, we have:
\begin{lemma} 
For any choice of a non-degenerate Hermitian product on $\acs$, the vector space $\mc{D}(d_\X):= \{u \in\mc H\Lambda\, |\, d_\X u\in\mc H\Lambda\}$ becomes a Hilbert space when endowed with the scalar product
 \begin{equation}\label{eq:scalar_prod_dom_dX}
  \langle \cdot,\cdot \rangle_{\mc{D}(d_\X)} := \langle \cdot,\cdot \rangle_{\mc H\Lambda}
  + \langle d_\X \cdot,d_\X \cdot\rangle_{\mc H\Lambda}.
 \end{equation}
 Furthermore, all scalar products obtained this way are equivalent and induce the same topology on $\mc{D}(d_\X)$. Finally, $d_\X$ is \emph{bounded} on $\mathcal{D}(d_\X)$.
\end{lemma}

\begin{proof}
We have to check that $\mathcal D(d_\X)$ is complete with respect to the topology of 
$\langle\cdot,\cdot\rangle_{\mathcal{D}(d_\X)}$: suppose $u_n$ is a Cauchy sequence in 
$\mc D(d_\X)$, then $u_n$ and $d_\X u_n$ are Cauchy sequences in $\mc H$ and we denote by
$v_0,v_1 \in \mathcal H\Lambda$ their respective limits. By the continuous embedding 
$\mathcal H\subset  C^{-\infty}(\mc M;E)$ and the continuity of $d_\X$ on $C^{-\infty}(\mc M;E)$ we deduce 
\[
v_{1} = \lim_{n\to\infty} d_\X u_n = d_\X \lim_{n\to\infty}u_n = d_\X v_0,
\]
in $C^{-\infty}(\M;E)$ which proves the completeness. For the boundedness, we take $u\in \mathcal{D}(d_\X)$ and we compute
\[
\| d_\X u\|^2_{\mathcal{D}(d_\X)} = \| d_\X u \|^2_{\mathcal{H}\Lambda} + \| d_\X d_\X u\|^2_{\mathcal{H}\Lambda} \leq \| u \|^2_{\mathcal{D}(d_\X)}.
\qedhere\]
\end{proof}
 
To be able to use the usual techniques, it is crucial that $C_c^\infty(\mathcal{M};E)$ is not only dense in $\mathcal{H}$ but also in $\mathcal{D}(d_\X)$ --- on this level of generality, this is not a priori guaranteed. For this reason, we say the $\a$-action $\X$ has a \emph{unique extension} to $\mc H$ if 
\begin{equation}
 \label{eq:unique_extension_cf}
 \bbar{C_c^\infty(\mc M; E) \Lambda}^{\mc D(d_\X)} = 
\mc D(d_\X).
\end{equation}
 We note that by \cite[Lemma A.1]{FS11}, if $\M$ is a closed manifold and if 
 $\mc{H}=\mc{A}(L^2(\M, E) \Lambda)$ for some invertible pseudo-differential operator $\mc{A}$ on $\M$ so that $\mc{A}^{-1}d_\X \mc{A}\in \Psi^1(\M;E\otimes \Lambda)$ (see Appendix~\ref{sec:microlocal} for the notation), then  $C^\infty(\M;E\Lambda)\Lambda$ is dense in the domain $\mc{D}(d_\X)$ and there is only one closed extension for $d_\X$.

In order to finally define the Taylor spectrum in an invariant way, 
we consider $\lambda\in\acs$ as a Lie algebra morphism
\[
 \lambda:\a_\C\to \OpDiff^0(\mc M; E)\subset \OpDiff^1(\mc M; E), 
 \quad \lambda(A)(u):= \lambda(A)u.
\]
In this way we can define $\X-\lambda:\a_\C\to \OpDiff^1(\mc M; E)$ and the associated operator $d_{\X-\lambda}$ on $C_c^\infty \Lambda$ and $C^{-\infty}\Lambda$. Since $d_{\X - \lambda} = d_\X - d_\lambda$, and $d_\lambda$ is bounded on $\mathcal{H}\Lambda$,  $\mathcal{D}(d_{\X-\lambda})$ does not depend on $\lambda$. Furthermore, note that from Lemma~\ref{lem:d_X_properties} we know that $d_{\mathbf X-\lambda}^2 = 0$ on $C_c^\infty \Lambda$ and by density of $C_c^\infty(\mc M, E) \Lambda\subset\mc D(d_\mathbf{X}) $ and boundedness of $d_{\mathbf X-\lambda}:\mc D(d_\mathbf{X}) \to \mc D(d_\mathbf{X})$  we deduce $d_{\X-\lambda}^2 = 0$ on $\mc D(d_\X)$. For $k=0,\ldots,\kappa$, we write $\mc{D}^k(d_\X):= \mc{D}(d_\X)\cap \mc H\Lambda^k$ and we gather the results above in the following:
\begin{lemma}
For an $\a$-action $\mathbf X$ with unique closed extension to $\mc H$, for any $\lambda\in\acs$
 \begin{equation}\label{eq:def_complex_coordinatefree}
  0\longrightarrow \mc{D}^0(d_\X)\overset{d_{\X-\lambda}}{\longrightarrow}\mc{D}^1(d_\X)\overset{d_{\X-\lambda}}{\longrightarrow}\dots \overset{d_{\X-\lambda}}{\longrightarrow} \mc{D}^\kappa(d_\X)\longrightarrow 0
 \end{equation}
 defines a complex of bounded operators, and the operators $d_{\X-\lambda}$ depend holomorphically on $\lambda\in \acs$.
\end{lemma}
Recall from the discussion above that the unique extension property was crucially used to get $d_{\mathbf X-\lambda}\circ d_{\mathbf X-\lambda} =0$, thus to have a well defined complex of bounded operators. 

We introduce the notation
\begin{equation}\label{notationcoho}
\begin{split}
\ker_{\mathcal H\Lambda} d_{\X-\lambda} := \ker_{\mathcal D(d_\X)\to \mathcal D(d_\X)}d_{\X-\lambda} &,\quad   \ran_{\mc H\Lambda}d_{\X-\lambda}:= \ran_{\mc D(d_\X)\to\mc D(d_\X)} d_{\X-\lambda}\\
\ker_{\mathcal H\Lambda^j} d_{\X-\lambda} := \ker_{\mathcal D^j(d_\X)\to \mathcal D^{j+1}(d_\X)}d_{\X-\lambda} &,\quad  \ran_{\mc H\Lambda^j}d_{\X-\lambda}:= \ran_{\mc D^{j-1}(d_\X)\to\mc D^j(d_\X)} d_{\X-\lambda}.
\end{split}
\end{equation}
Now, following the previous discussion of the Taylor spectrum, we can define
\begin{Def}
Let $\X:\a\to \OpDiff^1(\mc M;E)$ be a Lie algebra morphism and $\mc H$ a Hilbert space such that the $\mathfrak a$ action $\mathbf X$ has a unique extension to $\mc H$. 
Then we define the \emph{Taylor spectrum} $\sigma_{\T, \mc H}(\X)\subset \acs$ by
\[
 \lambda\in\sigma_{\T, \mc H}(\X) \Longleftrightarrow \ran_{\mc H\Lambda}(d_{\X-\lambda}) \not= \ker_{\mc H\Lambda}(d_{\X-\lambda}).
\]
This is equivalent to saying that the sequence \eqref{eq:def_complex_coordinatefree} is not an exact sequence. The complex is said to be \emph{Fredholm} if $\ran_{\mc H\Lambda}(d_{\mathbf X-\lambda})$ is closed and the cohomology $\ker_{\mc H\Lambda}(d_{\X-\lambda})/\ran_{\mc H\Lambda}(d_{\X-\lambda})$ has finite dimension. In this case we say that $\lambda$ is not in the \emph{essential Taylor spectrum} $\sigma_{{\rm T},\mc{H}}^{\rm ess}(\X)$ of $\X$ and define the \emph{index} by
\begin{equation}\label{indexX}
\Index(\X-\lambda)  := \sum_{\ell=0}^\kappa (-1)^\ell \dim (\ker_{\mc H\Lambda^{\ell}} d_{\X-\lambda}/\ran_{\mc H\Lambda^\ell} 
d_{\X-\lambda}).
\end{equation}
\end{Def}
As the usual Fredholm index for a single operator, the Fredholm index in the Taylor complex is also a locally constant function of $\lambda$ (see Theorem 6.6 in \cite{Cur88}).

Note that the non-vanishing of the $0$-th cohomology 
$\ker_{\mc H\Lambda^0} d_{{\bf X}-\lambda}$ of the complex is equivalent to 
\[ 
\exists u\in \mc{D}^0(d_{\bf X})\setminus\{0\},\,\quad ({\bf X}_{A_j}-\lambda_j)u=0,
\]
which corresponds to $(\lambda_1,\dots,\lambda_\kappa)$ being a joint eigenvalue of $({\bf X}_{A_1},\dots,{\bf X}_{A_\kappa})$. Obviously, on infinite dimensional vector spaces the joint eigenvalues do not provide a satisfactory notion of joint spectrum. Recall that for a single operator, $\lambda\in \C$ is in its spectrum if ${\bf X}-\lambda$ is either not injective or not surjective. In terms of the Taylor complex for a single operator $(\kappa=1)$ the non-injectivity corresponds to the vanishing of the zeroth cohomology group whereas the surjectivity corresponds to the vanishing of the first cohomology group.

\begin{rem}\label{rem:Taylorcomplex} 
So far we always started with a Lie algebra morphism $\X:\a\to {\rm Diff}^1(\mc{M};E)$, then considered the action of 
${\rm Diff}^1(\mc{M};E)$ on some topological vector space $V$ (e.g. $C_c^\infty(\mc{M})$) in order to define Taylor complex and Taylor spectrum. This will also be our main case of interest. However we notice that the construction of the operator $d_{\X}$ and the complex associated to $d_\X$ works exactly the same if we take instead any Lie algebra morphism
\[ \X: \a \to \mc{L}(V)\]
where $V$ is a topological vector space and $\mc{L}(V)$ denotes the Lie algebra of continuous linear operators on $V$ 
with Lie bracket $[A,B]:=AB-BA$. We shall call the complex induced by $d_{\X}$ on $V\Lambda$ the \emph{Taylor complex} of $\X$ on $V$. 
If $V$ is a Hilbert space we define the Taylor spectrum of $\X$ on $V$ by 
\[
 \lambda\in\sigma_{\T, V}(\X) \Longleftrightarrow \ran_{V\Lambda}(d_{\X-\lambda}) \not= \ker_{V\Lambda}(d_{\X-\lambda}).
\]
Such Lie algebra morphism that are not directly coming from differential operators will occasionally show up within the parametrix constructions in Sections~\ref{sec:Taylor-pseudors} and \ref{sec:resonance-at-zero}.  
\end{rem}

\subsection{Useful observations}

For the reader not familiar with the Taylor spectrum, and for our own use, we have gathered in this section several observations that are helpful when manipulating these objects.
First, we shall say that an operator $P:C^{-\infty}(\M;E)\otimes \Lambda\a^*_\C\to C^{-\infty}(\M;E)\otimes \Lambda\a^*_\C$ is 
\emph{$\Lambda$-scalar} if there is an operator 
$P': C^{-\infty}(\M;E)\to C^{-\infty}(\M;E)$ such that 
\[\forall u \in  C^{-\infty}(\M;E), \omega \in \Lambda\a^*_\C, \quad P(u\otimes \omega)=(P'u)\otimes \omega.\]

As usual with differential complexes, we have a dual notion of divergence complex. 
For this, we need a way to identify $\a$ with $\a^\ast$, i.e. a scalar product $\langle\cdot,\cdot\rangle$ on $\a$, extended to a $\C$-bilinear two form. If one chooses a basis, the implicit scalar product is given by the standard one in that basis. In any case, $A \mapsto A':= \langle A, \cdot\rangle$ is an isomorphism between $\a$ and $\a^\ast$. If 
\[{\bf Y} : \a \to \mc{L}(C_c^\infty(\M;E))\]
satisfies $[{\bf Y}_{A_1},{\bf Y}_{A_2}]=0$ for any $A_1,A_2\in \a$, then we can define the action ${\bf Y}':\a^*\mapsto \mc{L}(C_c^\infty(\M;E))$ 
by: for $u\in C_c^\infty(\M;E)$ and $A'\in \a^*$, if $A$ is dual to $A'$ we set 
\[
{\bf Y}'_{A'}u := {\bf Y}_A u.
\]
In this fashion, $d_{{\bf Y}'} u := {\bf Y}' u$ is an element of $C_c^\infty(\M;E)\otimes\a$ while 
$d_{\bf Y} u$ is an element of $C_c^\infty(\M;E)\otimes\a^\ast$. We can thus define the divergence operator associated to ${\bf Y}$
\begin{equation}\label{divergencedef}
\delta_{\bf Y} : \left\{\begin{array}{ccc}
C_c^\infty(\M;E)\otimes \Lambda^j\a^\ast_\C & \to & C_c^\infty(\M;E)\otimes \Lambda^{j-1}\a^\ast_\C \\
u\otimes \omega & \mapsto & - \imath_{{\bf Y}' u} \omega.
\end{array}\right.\end{equation} 
In an orthonormal basis $(e_j)_j$ of $\a$ for $\langle\cdot,\cdot\rangle$ and $(e_j')_j$ the dual basis in $\a^*$, we get for $u\in C_c^\infty(\M;E)$ and $\omega=e'_{i_1}\wedge\dots \wedge e'_{i_\ell}$ 
\[
\delta_{\bf Y} (u \otimes \omega)= \sum_{j=1}^\ell (-1)^{j}({\bf Y}_{e_{i_j}}u) e'_{i_1}\wedge\dots\wedge \widehat{e'_{i_j}}\wedge \dots\wedge e'_{i_\ell}.
\]
We get directly that for $A'\in \a^*$
\[
A'\wedge\delta_{\bf Y}(u\otimes\omega) +  \delta_{\bf Y} (A'\wedge (u\otimes\omega)) = -A'\wedge \imath_{{\bf Y}' u} \omega - \imath_{{\bf Y}'u}(A'\wedge \omega) =- (A'({\bf Y}'u))\otimes \omega.
\]
It follows from similar arguments as before that
\[
{\bf Y}_A \delta_{\bf Y} = \delta_{\bf Y} {\bf Y}_A,\quad \delta_{\bf Y} \delta_{\bf Y} = 0.
\]
We have the following
\begin{lemma}\label{trick}
Let ${\bf X}:\a\to {\rm Diff}^1(\M;E)$ be an admissible lift and ${\bf Y}:\a\to \mc{L}(C_c^\infty(\M;E))$ satisfying 
${\bf Y}_{B_1}{\bf Y}_{B_2}={\bf Y}_{B_2}{\bf Y}_{B_1}$ for any $B_1,B_2\in\a$, 
such that  ${\bf X}_A {\bf Y}_B= {\bf Y}_B {\bf X}_A$ for any $A,B\in\a$. If we fix an inner product $\langle\cdot,\cdot\rangle$ on $\a$ and a corresponding orthonormal basis $(e_j)_j$ and if ${\bf X}_{i}:={\bf X}_{e_i}$ and ${\bf Y}_j:={\bf Y}_{e_j}$, we then have as continuous 
operators on $C_c^\infty(\mc{M};E)\Lambda$
\[
\delta_{\bf Y} d_{\bf X} + d_{\bf X} \delta_{{\bf Y}} = - \Big(\sum_{k=1}^\kappa {\bf X}_k{\bf Y}_k\Big) \otimes \Id.
\]
The sum does not depend on the choice of basis, because it is the trace of the matrix representing ${\bf X}{\bf Y}$ with $\langle\cdot,\cdot\rangle$.
\end{lemma}
\begin{proof} Let $e_i'$ be the dual basis to the chosen orthogonal basis $e_i$. For $I=(i_1,\dots,i_\ell)$ let $e'_I:=e'_{i_1}\wedge\dots \wedge e'_{i_\ell}$, then for $u\in C_c^\infty(\M;E)$ we compute
\begin{align*}
d_{{\bf X}}\delta_{{\bf Y}}(u\otimes e_I')&=\\
	-&\Big(\sum_{k\in I}({\bf X}_k{\bf Y}_ku)\otimes e_I'+\sum_{k\notin I,j}(-1)^{j-1}({\bf X}_k{\bf Y}_{i_j}u)\otimes e'_{k}\wedge e'_{i_1}\wedge \dots\wedge \widehat{e'_{i_j}}\wedge \dots\wedge e'_{i_\ell} \Big) ,\\
\delta_{{\bf Y}}d_{{\bf X}}(u\otimes e_I')&=\\
	-&\Big(\sum_{k\notin I}({\bf Y}_k{\bf X}_ku)\otimes e_I'+
\sum_{k\notin I,j}(-1)^{j}({\bf Y}_{i_j}{\bf X}_ku)\otimes
e_{k}'\wedge e_{i_1}'\wedge \dots\wedge \widehat{e}_{i_j}'\wedge \dots\wedge e_{i_\ell}'\Big).
\end{align*}
Using the commutation of $[{\bf X}_i,{\bf Y}_j]=0$, we obtain the result.
\end{proof}
As an illustration, let us recall the following classical fact:
\begin{lemma}\label{lem:taylor_finite_dim}
 Let $X_1,\ldots,X_\kappa$ be commuting operators on a finite dimensional vector space $V$. Then
 $\sigma_{\T, V}(X) = \{ {\rm joint}\, {\rm eigenvalues}\, {\rm of}\, X_1,\dots,X_\kappa \}\subset \C^\kappa$.
\end{lemma}
\begin{proof}
 By the basic theory of weight spaces (see e.g. \cite{Kna02}[Proposition 2.4]) $V$ can be decomposed into generalized weight spaces, i.e. there are finitely many $\lambda^{(j)}=(\lambda^ {(j)}_1,\ldots \lambda^{(j)}_\kappa)\in\C^\kappa$ and a direct sum decomposition $V=\oplus_jV_j$ which is invariant under all $X_1,\ldots,X_\kappa$ and there are $n_j$ such that 
 \[
  (X_i - \lambda^{(j)}_i)^{n_j}_{|V_j} =0,\quad \forall i=1,\ldots,\kappa, \,\, \forall j
 \]
 Commutativity and the Jordan normal form then imply that the $\lambda^{(j)}$ are precisely the joint eigenvalues of the tuple $X$. 
 Now let $\mu\neq\lambda^{(j)}$ for all $j$. We have to prove that $\mu\notin \sigma_{\T, V}(X)$. 
 By $\mu\neq\lambda^{(j)}$ we deduce that for any $j$ there is at least one $1\leq k_j\leq\kappa$ such that $\mu_{k_j}\neq \lambda^{(j)}_{k_j}$ and again by Jordan normal form $(X_{k_j}-\mu_{k_j}):V_j\to V_j$ is invertible. 
 Setting $\widetilde V_k:= \oplus_{k_j=k}V_j$, we can thus find an $X$ invariant decomposition $V= \oplus_{k=1}^\kappa \widetilde V_k$ such that $X_k-\mu_k:\widetilde V_k\to \widetilde V_k$ is invertible.
Let $\Pi_k$ be the projection onto $\widetilde V_k$ w.r.t. the above direct sum decomposition. Now set $Y_k:=(X_k-\mu_k)^{-1}\Pi_k:V\to V$. Then the $Y_k$ satisfy all the assumptions of Lemma~\ref{trick} and
 \[
 \delta_Y d_{(X-\mu)} + d_{(X-\mu)}\delta_Y = -\Id .
 \]
 Consequently the Taylor complex~\eqref{eq:def_complex_coordinatefree} is exact.
\end{proof}

In the particular case that $X=(X_1, \dots, X_\kappa)$ are symmetric matrices, using the spectral theorem, we can reduce the problem to the case that $X_1, \dots, X_\kappa$ are scalars acting on some $\R^m$. From this we deduce that for $\lambda \in \sigma_{\T,\R^m}(X)$, 
\[
\dim( \ker_{\Lambda^j} d_{X-\lambda} / \ran_{\Lambda^j} d_{X-\lambda} ) = \dim (\R^m \otimes \Lambda^j \R^\kappa) = m {\binom{\kappa}{j}},
\]
and we check that
\[
\Index (X-\lambda) = m \sum_{j=1}^\kappa (-1)^j {\binom{\kappa}{j} } = 0.
\]

Our next step is to give a criterion for $d_{\X-\lambda}$ to be Fredholm. We first notice that since $\ran d_{\X-\lambda} \subset\ker d_{\X-\lambda}$, the closedness of $\ran d_{\X-\lambda}$  in ${\mc D}(d_\X)$ or in $\mc{H}\Lambda$ is equivalent.
Below, if $F\subset C^{-\infty}(\mc{M};E)\Lambda$ 
is a vector subspace, we shall denote
$\ran_{F}d_{\X}:=\{ d_{\X}u\,|\, u\in F\}$ and $\ker_F d_{\X}:=\{u\in F\, |\, d_\X u=0\}.$
We shall use the following criterion for the $d_\X$-complex to be Fredholm. 
\begin{lemma}\label{lem:critereFredholm}
Let $\X$ be an $\a$-action with unique extension to $\mc{H}$. Assume that there are bounded operators $Q$, $R$ and $K$ on  $\mc{H}\Lambda$, acting continuously on $C^{-\infty}(\M;E)\Lambda$, such that $K$ is compact, $\|R\|_{\mc{L}(\mc{H}\Lambda)}<1$, and 
\[
Qd_\X+d_\X Q=\Id + R + K.
\] 
Then the complex defined by $d_\X$ is Fredholm. Denote by $\Pi_0$ the projector on the eigenvalue $0$ of the Fredholm operator $\Id+R+K$; it is bounded on $\mathcal{D}(d_\X)$ and commutes with $d_\X$. Then the map $u \mapsto \Pi_0 u$ from $\ker d_\X\cap\mathcal{D}(d_\X)$ to $\ker d_\X \cap \ran \Pi_0$ factors to an isomorphism
\begin{equation}\label{eq:isomorphism-finite-dimension}
\Pi_0: \ker_{\mathcal{D}(d_{\X})}d_{\X} / \ran_{\mathcal{D}(d_{\X})} d_{\X} \to \ker_{\ran \Pi_0} d_{\X} / \ran_{\ran \Pi_0} d_{\X}.
\end{equation}
\end{lemma}

\begin{proof}
First, since $Q$, $R$ and $K$ are continuous on distributions, it makes sense to write $d_\X Q + Q d_\X = \Id + R + K$ in the distribution sense. Further, from this relation, we deduce that $Q$ is bounded on $\mathcal{D}(d_\X)$. Additionally, without loss of generality (by modifying $R$) we can assume that $K$ is a finite rank operator. 

Let us prove that the range of $d_\X$ is closed. Consider $u \in (\ker d_\X)^\perp\cap \mc{D}(d_\X)$. Since $d_\X Q u \in \ran(d_\X) \subset \ker d_\X$
\begin{equation}
\langle ({\rm Id} + R + K) u, u \rangle_{\mc{H}\Lambda} = \langle Qd_\X u , u \rangle_{\mc{H}\Lambda}.
\end{equation}
We get that there is $C>0$ such that for each $u\in (\ker d_\X)^\perp\cap \mc{D}(d_\X)$ we have 
\begin{equation}
(1-\| R\|) \| u\|_{\mc{H}\Lambda} -\| Ku \|_{\mc{H}\Lambda}  \leq  C \| d_\X u \|_{\mc{H}\Lambda}. 
\end{equation}
Since $K$ is of finite rank, we obtain by a standard argument that $d_\X$ has closed range (both in $\mathcal{H}\Lambda$ and $\mathcal{D}(d_\X)$).

The operator $F:=\Id + R + K$ is Fredholm of index $0$ and, since $F d_\X = d_\X Q d_\X = d_\X F$ \emph{on distributions}, we deduce that 
$F: \mathcal{D}(d_\X)\to \mathcal{D}(d_\X)$ is bounded. Since $F$ is Fredholm of index $0$, we know that for $s\mapsto (F-s)^{-1}\in \mc{L}(\mc{H}\Lambda)$ is meromorphic in $\C\setminus B(1,\|R\|_{\mathcal L(\mc H\Lambda)})$ and for
 $s\in\C^\ast=\C\setminus \{0\}$ close to $0$ it is analytic. Note that 
\[ \|Fu\|_{\mc{D}(d_\X)}^2= \|Fu\|_{\mc{H}\Lambda}^2+\|d_\X Fu\|_{\mc{H}\Lambda}^2=
\|Fu\|_{\mc{H}\Lambda}^2+\|Fd_\X u\|_{\mc{H}\Lambda}^2\leq \|F\|^2_{\mc{L}(\mc{H}\Lambda)} \|u\|_{\mc{D}(d_\X)}^2 \]
thus $(F-s)$ is invertible on $\mc{D}(d_\X)$ for $|s|> \|F\|_{\mc{L}(\mc{H}\Lambda)}$.
This implies that $(F-s)^{-1}:\mc{D}(d_\X)\to \mc{D}(d_\X)$ is itself bounded in $\{|s|>\|F\|_{\mc{L}(\mc{H}\Lambda)}\}$ with $d_\X(F-s)^{-1}=(F-s)^{-1}d_\X$, and it extends meromorphically to $s\in \C\setminus B(1,\|R\|_{\mathcal L(\mc H\Lambda)})$ 
as an operator $(F-s)^{-1}:\mc{D}(d_\X)\to \mc{H}\Lambda$. By meromorphic continuation we have for all $u\in \mc{D}(d_\X)$ and 
$s\in \C\setminus B(1,\|R\|_{\mathcal L(\mc H\Lambda)})$ not a pole of $(F-s)^{-1}$,
\[ d_\X(F-s)^{-1}u=(F-s)^{-1}d_\X u  \textrm{ in }C^{-\infty}(\mc{M};E)\Lambda.\]
In particular we deduce that for all $s$ close to $0$, $d_\X(F-s)^{-1}u\in \mc{H}\Lambda$ with $\|d_\X(F-s)^{-1}u\|_{\mc{H}\Lambda}\leq \|(F-s)^{-1}\|_{\mc{L}(\mc{H}\Lambda)}\|u\|_{\mc{D}(d_\X)}$ , i.e. $(F-s)^{-1}:\mc{D}(d_\X)\to \mc{D}(d_\X)$ is bounded, and $d_\X(F-s)^{-1}=(F-s)^{-1}d_\X$ on $\mc{D}(d_\X)$.

In that case, the spectral projector $\Pi_0$ of $F$ for the eigenvalue $0$ commutes with $d_{\bf X}$, is bounded on $\mathcal{D}(d_\X)$, and since $\mathcal{D}(d_\X)$ is dense in $\mathcal{H}\Lambda$ and $\Pi_0$ has finite rank, its image is contained in $\mathcal{D}(d_\X)$. Further, we can write $F= (F+\Pi_0)({\rm Id}-\Pi_0)$, and $\widetilde{F}:=F+\Pi_0$ is invertible on $\mathcal{H}\Lambda$ and $\mathcal{D}(d_\X)$, and commuting with $d_\X$, so that on $\mc{D}(d_\X)$
\begin{equation}\label{Id-Pi0}
d_\X \widetilde{F}^{-1}Q + \widetilde{F}^{-1}Q d_\X = {\rm Id} - \Pi_0.
\end{equation}
In particular, for $u\in\ker d_\X\cap \mc{D}(d_\X)$, we have 
\begin{equation}\label{injectiviteisom}
u = d_\X \widetilde{F}^{-1}Q u + \Pi_0 u.
\end{equation}
Since $\Pi_0$ and $d_\X$ commute, $u \mapsto \Pi_0 u$ factors to a homomorphism between the cohomologies in \eqref{eq:isomorphism-finite-dimension}. This map in cohomologies is obviously surjective since $\ran \Pi_0\subset \mc{D}(d_\X)$.
To prove that the map is injective, we need to prove that if $\Pi_0 u \in d_\X \ran \Pi_0$ for $u\in \ker d_\X\cap \mathcal{D}(d_\X)$, then $u\in d_\X \mathcal{D}(d_\X)$. This fact actually follows directly from \eqref{injectiviteisom} by using that both $\widetilde{F}^{-1}$ and $Q$ are bounded on $\mc{D}(d_\X)$.
\end{proof}
We can also deduce the following:
\begin{lemma}\label{lem:spectrum-and-eigenvalues}
Under the assumptions of Lemma \ref{lem:critereFredholm}, if $F:={\rm Id}+K+R$ is of the form $F=F'\otimes \Id$ where $F'$ is an operator on $\mathcal{H}$ (i.e. $F$ is \emph{$\Lambda$-scalar}), then $0\in \sigma_{\T,\mathcal{H}}(\X)$ if and only if there exists a non-zero $u \in \mathcal{D}(d_\X)\cap \mc{H}$ such that $\X u = 0$.
\end{lemma}
\begin{proof}
From Lemma \ref{lem:critereFredholm}, we deduce that $0\in\sigma_{\T,\mathcal{H}}(\X)$ if and only if the complex given by $d_\X$ is not exact on $\ran \Pi_0$ (recall that $d_\X$ commutes with $\Pi_0$). However, if $F$ is $\Lambda$-scalar, then $\Pi_0=\Pi_0'\otimes {\rm Id}$ with $\Pi_0'$ the spectral projector at $0$ of $F'$ on $\mc{H}$, and $\ran \Pi_0=(\ran\Pi_0')\otimes \Lambda \a^*_\C$. It follows that $d_\X$ restricted to $\ran \Pi_0$ is exactly the Taylor complex of the operator $\X$ on $\ran \Pi_0'$ in the sense of Remark \ref{rem:Taylorcomplex}. We are thus reduced to finite dimension and we can apply Lemma \ref{lem:taylor_finite_dim}.
\end{proof}

The version of the Analytic Fredholm Theorem  for the Taylor spectrum is the following statement:

\begin{prop}\label{pr:complex-analytic}
Let $\X$ be an $\a$-action with unique extension to $\mc{H}$.
Then $\sigma_{\T, \mc H}(\X) \setminus \sigma_{{\rm T},\mc{H}}^{\rm ess}(\X)$ is a complex analytic submanifold of $\C^\kappa\setminus \sigma_{{\rm T},\mc{H}}^{\rm ess}(\X)$.
\end{prop}
\begin{proof}
 As the complex \eqref{eq:def_complex_coordinatefree} is an analytic Fredholm complex of bounded operators 
 on $\C^\kappa\setminus \sigma_{{\rm T},\mc{H}}^{\rm ess}(\X)$ the statement is classical and a proof can be found in \cite[Theorem 2.9]{Muller-00}.
\end{proof}

In general, the question of whether the spectrum is discrete  does not seem to have a very simple answer. For example, a characterization can be found in \cite[Corollary 2.6 and Lemma 2.7]{Ambrozie-Mueller-09}.
Such a criterion is particularly adapted to microlocal methods and it can actually be used in our setting. However, it turns out that an even simpler criterion is sufficient for us:
\begin{lemma}\label{lemma:discrete}
Under the assumptions of Lemma \ref{lem:critereFredholm}, assume in addition that $Q=\delta_{{\bf Q}}$ for some Lie algebra morphism  ${\bf Q}:\a\to \mc{L}(\mc{H})$ such that ${\bf Q}_A$ acts continuously on $C^{-\infty}(\M;E)$ and $[{\bf Q}_A,\X_B]=0$ for all $A,B\in \a$. Then, Lemma \ref{lem:spectrum-and-eigenvalues} applies, and the Taylor spectrum of $\X$ on $\mc{H}$ is discrete in a neighborhood of $0$.
\end{lemma}

\begin{proof} Let $A_1,\dots,A_\kappa\in \a$ be an orthonormal basis for $\cjg\cdot,\cdot\cjd$ and let $Q_j:={\bf Q}_{A_j}$.
We observe from Lemma \ref{trick} that for $\lambda\in \a_\C^*$ the following identity holds on $\mc{D}(d_\X)$
\[
d_{\X -\lambda} Q + Q d_{\X-\lambda} = (\underset{=F'}{\underbrace{-{\bf X}_{A_1} Q_1 - \dots - {\bf X}_{A_\kappa} Q_\kappa}} + \underset{=\lambda\cdot Q}{\underbrace{\lambda_1 Q_1 + \dots + \lambda_\kappa Q_\kappa}})\otimes \Id.
\]
Thus, denoting $F'(\lambda):=F'+\lambda\cdot Q$ on $\mc{H}$ and $F(\lambda):=F'(\lambda)\otimes {\rm Id}$ on $\mc{H}\Lambda$, we see that Lemma \ref{lem:spectrum-and-eigenvalues} indeed applies.

Next, we observe two things. The first is that $F'$ and $\lambda\cdot Q$ commute. The second is that for $\lambda$ small enough, $F'(\lambda)$ can still be decomposed in the form $\Id + R(\lambda) + K(\lambda)$ with $\|R(\lambda)\|_{\mc{L}(\mc{H})}< 1$ and $K(\lambda)$ compact, because $Q$ is bounded. It follows that $d_{\X-\lambda}$ is Fredholm for $\lambda$ close enough to $0$.

From Lemma \ref{lem:critereFredholm}, we know that the cohomology of $d_{\X-\lambda}$ on $\mathcal{D}(d_\X)$ is isomorphic to
\[
\ker_{\ran \Pi_0(\lambda)} d_{\X-\lambda} / \ran_{\ran \Pi_0(\lambda)} d_{\X-\lambda},
\]
and the isomorphism is given by $[u] \mapsto [\Pi_0(\lambda)u]$ if $\Pi_0(\lambda)$ denotes the spectral projector of $F(\lambda)$ at $0$ and $[\cdot]$ denotes cohomology classes. Let us now describe a sort of \emph{sandwiching procedure}. Assume that we have  a projector $\Pi_2$ bounded on $\mathcal{D}(d_\X)$, commuting with $d_{\X-\lambda}$. Then, the mapping $[u] \mapsto [\Pi_2 u]$ is well-defined and surjective as a map 
\begin{equation}\label{isomPi2Dom}
 \ker_{\mc{D}(d_\X)} d_{\X-\lambda} / \ran_{\mc{D}(d_\X)} d_{\X-\lambda}\to \ker_{\ran \Pi_2} d_{\X-\lambda} / \ran_{\ran \Pi_2} d_{\X-\lambda}.
\end{equation}
In general, there is no reason for this map to be \emph{injective}. However, if we further assume that $\Pi_2$ and $\Pi_0(\lambda)$ commute, and that $\ran(\Pi_0(\lambda))\subset \ran(\Pi_2)$, then we can see $\Pi_0(\lambda)$ as a projector on $\ran(\Pi_2)$. The mapping $[\Pi_2 u] \mapsto [\Pi_0(\lambda)u]$ for $u\in \ker_{\mc{D}(d_\X)} d_{\X-\lambda} / \ran_{\mc{D}(d_\X)} d_{\X-\lambda}$ is well-defined as a map 
\[\ker_{\ran \Pi_2} d_{\X-\lambda} / \ran_{\ran \Pi_2} d_{\X-\lambda}\to \ker_{\ran \Pi_0(\lambda)} d_{\X-\lambda} / \ran_{\ran \Pi_0(\lambda)} d_{\X-\lambda}\]
by using $\ker \Pi_2\subset \ker \Pi_0(\lambda)$, and it has to be surjective. Using this and the surjectivity of \eqref{isomPi2Dom} we deduce the bounds 
\begin{align*}
& \dim \ker_{\mc{D}(d_\X)} d_{\X-\lambda} / \ran_{\mc{D}(d_\X)} d_{\X-\lambda} \\ 
& \qquad \leq \dim \ker_{\ran \Pi_2} d_{\X-\lambda} / \ran_{\ran \Pi_2} d_{\X-\lambda}\leq\ker_{\ran \Pi_0(\lambda)} d_{\X-\lambda} / \ran_{\ran \Pi_0(\lambda)} d_{\X-\lambda}. 
\end{align*}
Since we have proved above that the lower and upper bound are equal, then \eqref{isomPi2Dom} is actually an isomorphism.

Let us write, with $\widetilde{F}':=F'+\Pi_0'$ where $\Pi_0'$ is the spectral projector of $F'$ at $0$ and we have the following identity on $\mc H$
\[\widetilde{F}'^{-1}F'(\lambda)={\rm Id} - \Pi'_0 + \widetilde{F}'^{-1} \lambda\cdot  Q.\] 
For $u\in \ker F'(\lambda)$, we have $({\rm Id}- \Pi'_0) u =- \widetilde{F}'^{-1}\lambda \cdot  Q u$. Since $\widetilde{F}'^{-1} \lambda \cdot Q $ commutes with $\Pi'_0$ (as $F'$ does commute with $\lambda\cdot Q$), we obtain for $u\in \ker F'(\lambda) \subset\mc H$
\[
({\rm Id}-\Pi'_0)u = ({\rm Id}-\Pi'_0)^2 u = -\widetilde{F}'^{-1}\lambda\cdot Q ({\rm Id}-\Pi'_0)u.
\]
For $\lambda$ small enough ${\rm Id}+\widetilde{F}'^{-1}\lambda \cdot Q$ is invertible on $\mc{H}$, which implies that $({\rm Id}-\Pi'_0)u=0$. In particular, $u\in \ran \Pi'_0$, so that $\ker F'(\lambda) \subset \ker F'$ and $\ran \Pi'_0(\lambda)\subset \ran \Pi'_0$. But certainly, $\Pi_0'$ and $\Pi_0'(\lambda)$ commute. So we can apply the argument above with $\Pi_0'$ playing the role of $\Pi_2$, and deduce that for $\lambda$ sufficiently small,
\[
\ker_{\mathcal{D}(d_\X)} d_{\X-\lambda}/ \ran_{\mathcal{D}(d_\X)} d_{\X-\lambda} \simeq \ker_{ \ran \Pi_0'} d_{\X-\lambda}/ \ran_{ \ran \Pi_0'} d_{\X-\lambda}.
\]
Since $\ran \Pi_0'$ is a fixed finite dimensional space, the Taylor spectrum of ${\bf X}$ is discrete near $0$ by Lemma \ref{lem:taylor_finite_dim}.
\end{proof}

\section{Discrete Ruelle-Taylor resonances via microlocal analysis}
\label{sec:Taylor-pseudors}
In this section, $\mc{M}$ is a compact manifold, equipped with a vector bundle $E\to\mc M$ and an admissible lift $\X$ of an Anosov action (see Definition~\ref{def:admissible_lift}). We have seen in Section~\ref{sec:Taylor_commuting} how to define the Taylor differential $d_\X$ which acts in its coordinate free form on $C^\infty(\mc M; E)\otimes\Lambda \a^*$. We have furthermore seen how $d_\X$ can be used to define a Taylor spectrum $\sigma_{\T, \mc H}(\mathbf X)\subset \a^*_{\mathbb C}$. We take coordinates whenever it is convenient. In that case, we will use the notation $d_X$ to avoid confusion. In the sequel it will be convenient to pass back and forth between these versions and we will mostly use the shorthand notation $C^\infty\Lambda$, leaving open which version we currently consider. 

The Ruelle-Taylor resonances that we will introduce will correspond to a discrete spectrum of $-\mathbf X$ on some anisotropic Sobolev spaces. 
From a spectral theoretic point of view this sign convention might seem unnatural. However, from a dynamical point of view this convention is very natural: 
given the flow $\varphi_t^X$ of a vector field $X$, the one-parameter group that propagates probability densities with respect to an invariant measure is given by $(\varphi^ X_{-t})^*$ and thus generated by the differential operator $-X$. 
We will therefore from now on consider the holomorphic family of complexes generated by $d_{\mathbf X+\lambda}$ for $\lambda \in \a_{\mathbb C}^*$ (respectively $\lambda\in \C^\kappa$ after a choice of coordinates). Let us denote by $e^{-t\mathbf X_A}$ the $1$-parameter family generated by $\mathbf X_A$, solving $\partial_t e^{-t{\mathbf X}_A}f=-{\mathbf X}_Ae^{-t{\bf X}_A}f$ with $e^{-t{\mathbf X}_A}f|_{t=0}=f$. Since we work with spaces that are deformations of $L^2(\mc{M})$, we will compare our results with the growth rate of the action on $L^2(\mc{M})$, defined for $A\in \a$ as
\begin{equation}\label{eq:def-CL^2(A)}
C_{L^2}(A) := \limsup \frac{1}{t}\log \| e^{-t\mathbf{X}_A} \|_{\mc{L}(L^2)}. 
\end{equation}
Naturally, the spectrum of $\mathbf{X}_A$ on $L^2$ is contained in $\{s\in\C\ |\ \Re s \leq C_{L^2}(A)\}$. 

The goal of this section is to show the following:
\begin{theorem}\label{theoHn}
Let $\tau$ be a smooth abelian Anosov action with generating map $X$ and ${\bf X}$ an admissible lift. Let $A_0\in \W$ be in the positive Weyl chamber. 
There exists $c>0$, locally uniformly with respect to $A_0$, such that for each $N>0$, there is a Hilbert space $\mc{H}_{N}$ containing $C^\infty(\M)$ and contained in
$C^{-\infty}(\M)$ such that the following holds true:
\begin{enumerate}[1)]
	\item\label{it:essential} $-X$ has no essential Taylor spectrum on the Hilbert space $\mc{H}_N$ in the region 
\[
\mc{F}_N:=\{\lambda\in \a_\C^*\, |\,\Re(\lambda(A_0))>-cN + C_{L^2}(A_0)\}.
\]
	\item For each $\lambda\in \mc{F}_N$ one has an isomorphism of finite dimensional spaces
\[ 
\ker d_{{\bf X}+\lambda}|_{\mc{D}_N^j(d_{\bf X})}/
\ran\, d_{{\bf X}+\lambda}|_{\mc{D}_N^{j-1}(d_{\bf X})}=
\ker d_{{\bf X}+\lambda}|_{C^{-\infty}_{E_u^*}(\M)\otimes \Lambda^j\a^*_\C}/
\ran\, d_{{\bf X}+\lambda}|_{C^{-\infty}_{E_u^*}(\M)\otimes \Lambda^{j-1}\a^*_\C}
\]
with $\mc{D}_N^j(d_{\bf X}):=\{u\in \mc{H}_N\otimes \Lambda^j\a^*_\C\, |\, d_{\bf X}u\in \mc{H}_N\otimes \Lambda^{j+1}\a^*_\C\}$, showing that the cohomology dimension is independent of $N$ and $A_0$.
	\item The Taylor spectrum of $-{\bf X}$ contained in $\mc{F}_N$ is discrete and contained in
	\[ \bigcap_{A\in \mathcal W} \{\lambda\in\a_\C^*~|~\Re(\lambda(A))\leq C_{L^2}(A)\}.\]
	\item\label{it:always_joint_eigenvalue} An element $\lambda\in \mc{F}_N$ is in the Taylor spectrum of $-{\bf X}$ on $\mc{H}_N$ if and only if $\lambda$ is a joint Ruelle resonance of ${\bf X}$.	
\end{enumerate}
\end{theorem}
The Hilbert space $\mc{H}_N$ will be rather written $\mc{H}_{NG}$ below, where $G$ is a certain weight function on $T^*\mc{M}$ giving the rate of Sobolev differentiability in phase space. We use this notation in order to emphasize the dependence of the space on $G$.

The central point of the proof will be a parametrix construction for the exterior differential $d_{\mathbf X+\lambda}$. We will prove in Proposition~\ref{prop:Fredholm} that there are holomorphic families of operators $Q(\lambda), F(\lambda): C^{-\infty}\Lambda \to C^{-\infty}\Lambda$ such that 
\[
 Q(\lambda)d_{\mathbf X+\lambda} + d_{\mathbf X+\lambda}Q(\lambda)=F(\lambda).
\]
The operators $Q(\lambda)$ and $F(\lambda)$ will be  Fourier integral operators and independent of any Hilbert space on which the operators act.
However, the crucial fact is that for these operators there exists a scale of Hilbert spaces $C^\infty\subset \mc H_{NG}\subset C^{-\infty}$ (with $N\geq 0$ and $G\in C^\infty(T^*\M)$ a weight function) and domains $\mathcal F_{NG}\subset \a^*_{\mathbb C}$ with $\a^*_{\mathbb C}=\cup_{N>0} \mathcal F_{NG}$ such that for $\lambda \in \mathcal F_{NG}$ the operators $Q(\lambda):\mc H_{NG}\to \mc H_{NG}$ are bounded and the operators $F(\lambda): \mc H_{NG}\to \mc H_{NG}$ are Fredholm and can be decomposed as $F(\lambda)={\rm Id}+R(\lambda)+K(\lambda)$ with $K(\lambda)$ compact and $\|R(\lambda)\|_{\mc{L}(\mc{H}_{NG}\Lambda)}<1/2$.
Then by Lemma~\ref{lem:critereFredholm} we directly conclude that the Taylor complex on $\mathcal H_{NG}\Lambda$ is Fredholm on $\lambda\in \mathcal F_{NG}$. 
The fact that the construction of the operator family $F(\lambda):C^\infty\Lambda\to C^{-\infty}\Lambda$ is independent of the specific Hilbert spaces on which they act will be the key for proving in Section~\ref{sec:intrinsic_spec} that the Taylor spectrum of $d_{\mathbf X+\lambda}$ is intrinsic to the Anosov action, i.e. independent of the constructed spaces $\mc H_{NG}$. 
The flexibility which we will have in the construction of the escape function $G$ will furthermore allow to identify this intrinsic spectrum with the spectrum of $d_{\mathbf X+\lambda}$ on the space $C^{-\infty}_{E_u^*}\Lambda$ of distributions with wavefront set contained in the annihilator $E_u^*\subset T^*\M$ of $E_u\oplus E_0$ (see Proposition~\ref{prop:D_Eu}). 
Finally, we will see that the choice of $Q(\lambda)$ can be made more \emph{geometric}, to enable the use of Lemma~\ref{lemma:discrete} and prove that this intrinsic spectrum is discrete in $\a_{\C}^*$.

The construction of the parametrix $Q(\lambda)$ and the Hilbert spaces $\mc{H}_{NG}$ 
will be done using microlocal analysis. Appendix~\ref{sec:microlocal} contains a brief summary of the necessary microlocal tools. 
Section~\ref{sec:anisotropic} will be devoted to the construction of the anisotropic Sobolev spaces. 
With these tools at hand we will construct the parametrix (Section~\ref{sec:parametrix}), and prove that the spectrum is intrinsic (Section~\ref{sec:intrinsic_spec}) as well as discrete (Section~\ref{sec:discrete_spec}).

\subsection{Escape function and anisotropic Sobolev space}\label{sec:anisotropic}
In this section we define the anisotropic Sobolev spaces. 
Their construction will be based on the choice of a so-called escape function for the given Anosov action. 
We first give a definition for such an escape function and then prove the existence of escape functions with additional useful properties. 

Given any smooth vector field $X\in C^\infty(\mc M; T\mc M)$ with flow $\varphi_t^X$ we define the \emph{ symplectic lift} of the flow and the corresponding vector field by
\begin{equation}
  \label{eq:symplectic_lift}
  \Phi^X_t:\Abb{T^*\mc M}{T^*\mc M}{(x,\xi)}{(\varphi_t^X(x),(({d\varphi^X_t})^{-1})^T\xi)}~~~\tu{ and }~~~ X^H:= \frac{d}{dt}_{|t=0}\Phi^X_t \in C^\infty(T^*\mc M; T(T^*\mc M))
\end{equation}
where $(({d\varphi^X_t})^{-1})^T$ is the transpose of the inverted differential $({d\varphi^X_t})^{-1}$. The notation $X^H$ is chosen because it is the Hamilton vector field of the principal symbol $\sigma_p^1(X)(x,\xi) =i\xi(X(x)) \in C^\infty(T^*\mc M)$ of $X$ (see Example~\ref{exmpl:principal_symbols}). Recall from Example~\ref{exmpl:principal_symbols} that for an admissible lift of an Anosov action, the principal symbols of the lifted differential operator $\mathbf X_A$ and that of the vector field $X_A$ tensorized with $\Id_E$ coincide. This will turn out to be the reason why we do not have to care about the admissible lifts for the construction of the escape function.
We will denote by $\{0\}:=\{(x,0)\in T^*\M\,|\, x\in M\}$ the zero section.
\begin{Def}\label{def:escape_fct}
 Let $c_{X}>0$, $A\in \W$, $\Gamma_{E_0^*}\subset T^*\mc M$, an open cone containing $E_0^*$ satisfying $\overline{\Gamma}_{E_0^*}\cap (E^*_u\oplus E^*_s) = \{0\}$. Then a function $G\in C^\infty(T^*\mc M,\R)$ is called an \emph{escape function} for $A$ compatible with $c_{X}$, $\Gamma_{E_0^*}$ if there is $R>0$ so that
 \begin{enumerate}
  \item\label{it:escfct_homog_cond} for $|\xi|\leq R/2$ one has $G(x,\xi)=1$ and for
  $|\xi|>1$ one can write, $G(x,\xi)=m(x,\xi)\log(1+ f(x,\xi))$. 
  Here   $m\in C^\infty(T^*\M;[-1/2,8])$ and for $|\xi|>R$, $m$ is positively homogeneous of degree $0$, with $m\leq -1/4$ in a conic neighborhood of $E_u^*$ and  $m\geq 4$ in a conic neighborhood of $E^*_s$. Furthermore, $f\in C^\infty(T^*\M,\R^+)$ is positive homogeneous of degree $1$ for $|\xi|>R$. We call $m$ the \emph{order function} of $G$. 
\item\label{it:escfct_m_decay} $X_A^H m(x,\xi)\leq 0$ for all $|\xi|>R$, 
\item\label{it:escfct_G_decay} for $\xi\notin \Gamma_{E_0^*}, |\xi|>R$ one has $X_A^H G(x,\xi)\leq -c_{X}$. 
\end{enumerate}
\end{Def}
Below (see Proposition~\ref{prop:escape_fct_exist}), we will prove the existence of escape functions for Anosov actions. Before coming to this point let us explain how we can build the anisotropic Sobolev spaces based on the escape function: given an escape function $G$, Property (\ref{it:escfct_homog_cond}) of Definition~\ref{def:escape_fct} implies that $m\in S^0_1(\mc M)$ and for any $N>0$, $e^{NG}\in 
S^{Nm}_{1-}(\mc M)$ is a real elliptic symbol. According to \cite[Lemma 12 and Corollary 4]{FRS08} there exists a pseudodifferential operator
\begin{equation}\label{defANG}
\hat{\mc{A}}_{NG} \in \Psi^{Nm}_{1-}(\mc M; E)
\end{equation} 
such that
\begin{enumerate}
 \item $\sigma_p^{Nm}(\hat{\mc{A}}_{NG}) = e^{NG}\Id_E \mod S^{Nm-1+}_{1-}$,
 \item $\hat{\mc{A}}_{NG}: C^\infty(\mc M; E) \to C^\infty(\mc M; E)$ is invertible,
 \item $\hat{\mc{A}}_{NG}^{-1} \in \Psi^{-Nm}_{1-}(\mc M; E)$ and 
 $\sigma_p^{-Nm}(\hat{\mc{A}}_{NG}^ {-1}) = e^{-NG}\Id_{E} \mod S^{-Nm-1+}_{1-}$.
 \end{enumerate}
We can now define the \emph{anisotropic Sobolev spaces} 
\[
 \mc H_{NG} := \hat{\mc{A}}_{NG}^{-1} L^2(\mc M;E) \tu{ with scalar product } \langle u ,v\rangle_{\mc H_{NG}}:=\langle\hat{\mc{A}}_{NG}u, \hat{\mc{A}}_{NG}v\rangle_{L^2}.
\]
Note that the scalar product $\langle u,v\rangle_{\mc H_{NG}}$ depends not only on the choice of the escape function but also on the choice of its quantization $\hat{\mc{A}}_{NG}$.
However, by $L^2$-continuity (Proposition~\ref{prop:l2-bound}), these different choices all yield equivalent scalar products  on the given vector space $\mc H_{NG}$. For that reason we can suppress this dependence in our notation. 

We want to study the Taylor spectrum of the admissible lift of the Anosov action on these anisotropic Sobolev spaces. Recall from Section~\ref{sec:Taylor_commuting} that due to the unboundedness of the differential operators we have to verify the unique extension property:
\begin{lemma}
For any escape function $G$ the $\a$-action of an admissible lift has a unique extension (in the sense that Equation~\eqref{eq:unique_extension_cf} holds) to the 
 anisotropic Hilbert space $\mc H_{NG}$.
\end{lemma}
\begin{proof}
 Let us consider the Taylor differential $d_\X$ as an unbounded operator on $\mc H_{NG}\Lambda$ with domain $C^\infty(\mc M; E\otimes\Lambda)$. Then, in the language of closed extensions, the desired equality \eqref{eq:unique_extension_cf} corresponds to the uniqueness of possible closed extensions. By unitary equivalence we can study the conjugate operator $P:=\hat{\mc{A}}_{NG}\, d_\X\hat{\mc{A}}_{NG}^{-1}$ acting as an unbounded operator on $L^2(\mc M; E\otimes\Lambda)$, instead.
 We want to apply \cite[Lemma A.1]{FS11} which states that any operator in $\Psi^1_1(\M;E\otimes \Lambda)$ has a unique closed extension as an unbounded operator on $L^2$ with domain $C^\infty$. By Proposition~\ref{prop:pseudo_product} and since $\hat{A}_{NG}$ has scalar principal symbol  we can write 
 $P=d_\mathbf X +[\hat{\mc{A}}_{NG}, d_\X]\hat{\mc{A}}_{NG}^{-1}$, where the first summand is obviously in $\Psi^1_1(\mc M; E\otimes \Lambda)$ and the second one in $\Psi^{0+}_{1-}(\mc M; E\otimes \Lambda)$.
 Now, by Definition~\ref{def:symbols} of symbol spaces, one checks that $S^{0+}_{1-}(\mc M; E\otimes \Lambda)\subset S^1_1(\mc M; E\otimes \Lambda)$. We conclude that $P\in \Psi^1_1(\mc M; E\otimes \Lambda)$ and are able to apply \cite[Lemma A.1]{FS11} which completes the proof.
\end{proof}

Let us now come to the existence of the escape function: 
\begin{prop}\label{prop:escape_fct_exist}
Fix an arbitrary $A_0\in \mc W\subset \a$, an open cone $\Gamma_{\tu{reg}}\subset T^*\M$ which is disjoint from $E_u^*$, and $\Gamma_0$ a small conic neighborhood of $E_0^*$ so that $\bbar{\Gamma}_0\cap (E_s^*\oplus E_u^*)=\{0\}$. Then there is a $c_{X}>0$, an open conic neighborhood $\Gamma_{E_0^*}\subset \Gamma_0$ of $E_0^*$, and $R>0$ such that there is an escape function $G$ for $A_0$ compatible with $c_X$ and $\Gamma_{E_0^*}$ with the additional property that the order function satisfies
\begin{equation}\label{mgeq1/2}
 m(x,\xi)\geq 1/2 \tu{ for }(x,\xi)\in \Gamma_\tu{reg}\tu{ and } |\xi|>R. 
\end{equation}
\end{prop}
\begin{proof}
The proof follows from \cite[Lemma 3.2]{DGRS18}: indeed, first we note that the proof there only uses the continuity of the decomposition $T^*\M=E_0^*\oplus E_u^*\oplus E_s^*$ and the contracting/expanding properties of $E_s^*$, $E_u^*$ but not the fact that $E_0^*$ is one dimensional. It suffices to take, in the notations of \cite{DGRS18}, $N_1=4$, $N_0=1/4$ and 
$\Gamma_{\rm reg}=T^*\M\setminus C^{uu}(\alpha_0)$ with $\alpha_0>0$ small enough. Although it is not explicitly written in the statement of \cite[Lemma 3.3]{DGRS18}, the order function $m$ constructed there satisfies $X_{A_0}^Hm(x,\xi)\leq 0$ for $|\xi|$ large enough and $\Gamma_{E_0^*}$ is arbitrarily small if $\alpha_0>0$ is small (see \cite[Section A.2]{DGRS18}).
\end{proof}

In the proof of the fact that the Ruelle-Taylor spectrum is discrete, we shall also need an escape function that works for all $A$ in a neighborhood $\mc{U}\subset \mc{W}$ of a fixed element $A_1\in \mc{W}$.
\begin{lemma}\label{unifG}
Let $A_1\in \W$ be fixed. Then there is an escape function $G$ for $A_1$, a conic neighborhood $\Gamma_{E_0}^*\subset T^*\M$ of $E_0^*$ such that $\bbar{\Gamma}_{E_0^*}\cap(E_u^*\oplus E_s^*)=\{0\}$, a constant $c_X>0$ and a neighborhood $\mc{U}\subset \W$ of $A_1$ such that 
$G$ is an escape function for all $A\in \mc{U}$ compatible with $c_X>0$ and $\Gamma_{E_0^*}$. Moreover, $G$ can be chosen to satisfy $X^H_{A} G\leq 0$ in $\{|\xi|\geq R\}$ for some $R\geq 1$.  
\end{lemma}
\begin{proof} In a first step we need to construct an order function $m$ that fulfills all properties of Definition~\ref{def:escape_fct}\ref{it:escfct_homog_cond} but additionally $X^H_{A}m\leq 0$ for $|\xi|\geq R$ for all $A$ close enough to $A_1$. To obtain this order function we can follow exactly the construction for Anosov flows given in \cite[Section 2]{Bon18}. It works mutatis mutandis in our case as the proof simply uses the continuity of the decomposition 
$T^*\M=E_0^*\oplus E_s^*\oplus E_u^*$ and the expanding/contracting properties of $E_s^*$ and $E_u^*$, but not the fact that $\dim E_0^*=1$.

We can then define the function $G$ as in \cite[Lemma 1.2]{FS11} by setting $G(x,\xi)=m(x,\xi)\log (1+f(x,\xi))$, where $f>0$ is positively homogeneous of degree $1$ in $\xi$ for $|\xi|>R$, satisfies $f(x,\xi)=|\xi(X_{A_1})|$ near $E_0^*\cap \{|\xi|\geq 1\}$, and 
\begin{equation}\label{XHAf}
X^H_{A}f<-c_1(1+f), \quad (\textrm{resp. }X^H_{A}f>c_1/(1+f))
\end{equation}
in a conic neighborhood of $E_s^*$ (resp. of $E_u^*$) for some $c_1>0$. To construct such $f$ near $E_s^*$, we can use the construction from
\cite[Lemma C.1]{DZ16a}: for $(x,\xi)$ in a conic neighborhood $N_s$ of $E_s^*$, set
\[ f(x,\xi):= \int_0^T |e^{-tX^H_{A_1}}(x,\xi)|dt, \quad T>0\]
so that, if $A=A_1+\eps A'$ with $|A'|_\a\leq 1$, one has $X^H_{A}=X^H_{A_1}+\eps X^{H}_{A'}$,
\[
X_A^Hf(x,\xi)= |\xi|-|e^{-TX^H_{A_1}}(x,\xi)|+\eps \int_0^T X^{H}_{A'}|e^{-tX^H_{A_1}}(x,\xi)| dt=  |\xi|-|e^{-TX^H_{A_1}}(x,\xi)|+ \mc{O}(\eps e^{CT}|\xi|)
\]
for some $C>0$ uniform with respect to $A'$ as above, the last term following from the classical estimate $\max_{|\alpha|+|\beta|\leq 1}\sup_{(x,\xi), |\xi|=1}\pl_x^{\alpha}\pl^\beta_\xi |e^{-tX^H_{A}}(x,\xi)|\leq Ce^{C|t|}$ and the homogeneity in $\xi$. Fix $T$ large enough so that we have 
$|\xi|-|e^{-TX^H_{A}}(x,\xi)|\leq -2|\xi|$ for all $|\xi|>1$ in $N_s$. Once $T$ has been fixed, one can choose $0<\eps<e^{-CT}$  so that $X_A^Hf(x,\xi)\leq -|\xi|$ in $N_s\cap \{|\xi|>1\}$. Since $(1+f(x,\xi))>c_1^{-1}|\xi|$ in $N_s\cap \{|\xi|>1\}$ for some $c_1>0$, 
we obtain \eqref{XHAf}. The same construction applies near $E_u^*$. We then extend $f$ in a positively homogeneous function of degree $1$ in $\{|\xi|\geq R\}$ in a smooth fashion (its value far from $E_u^*\cup E_s^*\cup E_0^*$ will not matter).
The proof of \cite[Lemma 1.2]{FS11} (using the fact that $X_{A}^H|\xi(X_{A_1})|=0$ as $[X_A,X_{A_1}]=0$) shows that $X^H_{A}G\leq 0$ for all $|\xi|\geq R$ if $R$ is large enough and that $G$ is an escape function for all $A\in \mc{U}:=A_1+\{\eps A'\in \a\, |\, |A'|_{\a}\leq 1\}$ 
compatible with some $c_X>0$ and some $\Gamma_{E_0^*}$.
\end{proof} 

\subsection{Parametrix construction}\label{sec:parametrix}
The goal of this section is to construct an operator $Q(\lambda)$ as in Lemma \ref{lem:critereFredholm} for the complex $d_{{\bf X}+\lambda}$, and so that $Q$ will be bounded on the anisotropic Sobolev spaces $\mc{H}_{NG}\Lambda$. The construction will be microlocal in the elliptic region and dynamical near the characteristic set. In Section \ref{sec:discrete_spec} we will provide an alternative construction of a $Q(\lambda)$ which is purely dynamical, i.e. which is a function of the operators ${\bf X}_{A_j}$. 

Recall the notation $E\otimes \Lambda = E\otimes \Lambda \a^*$. 
We will also freely identify operators $P: C^\infty(\mc M; E)\to C^{-\infty} (\mc M; E)$ with their $\Lambda$-scalar extension on sections of $E\otimes \Lambda$. 
\begin{lemma}\label{ellipticregion}
Let $P\in \Psi^0(\M; E)$ be such that ${\rm WF}({\rm Id}-P)$ does not intersect a conic neighborhood of $E_u^*\oplus E_s^*$, and we make it act as a $\Lambda$-scalar operator. 
There exists a holomorphic family of pseudo-differential operators $Q_{\Ell}(\lambda)\in \Psi^{-1}(\M;E\otimes \Lambda \a_\C^*)$ for $\lambda\in \a^*_\C$ such that  $Q_{\Ell}(\lambda): 
C^\infty(\mc M; E\otimes \Lambda^k\a_\C^*)\to C^\infty(\mc M; E\otimes \Lambda^{k-1}\a_\C^*)$
for all $k$ and
\begin{equation}\label{eq:elliptic_param} 
d_{\mathbf X+\lambda} Q_{\Ell}(\lambda)+Q_{\Ell}(\lambda)d_{\mathbf X+\lambda}=({\rm Id}-P)+S_1(\lambda)+S_2(\lambda) 
\end{equation}
with $S_1(\lambda)\in \Psi^{-1}(\M,E\otimes \Lambda)$ holomorphic in $\lambda$ satisfying $\WF(S_1(\lambda))\subset\WF(P)\cap \WF({\rm Id}-P)$ and $S_2(\lambda)\in \Psi^{-\infty}(\M;E\otimes \Lambda)$, also holomorphic in $\lambda$. 
\end{lemma}
\begin{proof}
We will use an arbitrary choice of basis $A_1,\ldots, A_\kappa$ in $\a$ and consider the commuting differential operators $\mathbf X_{A_1},\ldots, \mathbf X_{A_\kappa}$.
Recall that the corresponding divergence operator $\delta_{\mathbf X+\lambda}$ on  $C^\infty(\M;E)\otimes\Lambda \a^*$ is defined by 
\[
\delta_{\mathbf X+\lambda}(u \otimes e_{i_1}\wedge \dots \wedge e_{i_\ell})=\sum_{j=1}^\ell (-1)^{j}(\mathbf X_{A_{i_j}}+\lambda_{i_j})u \otimes e_{i_1}\wedge\dots\wedge \widehat e_{i_j}\wedge \dots\wedge e_{i_\ell},
\] 
where $\lambda_j:=\lambda(A_j)\in \C$ (here $(e_j)_j$ is a dual basis to $A_j$ in $\a^*$).
Thus, using the commutations $[\mathbf X_{A_j}+\lambda_j,\mathbf X_{A_k}+\lambda_k]=0$ and Lemma \ref{trick} with 
$\mathbf Y_j=\mathbf X_{A_j}+\lambda_j$, 
we obtain that the operator 
$\Delta_{\mathbf X+\lambda}:=d_{\mathbf X+\lambda}\delta_{\mathbf X+\lambda}+\delta_{\mathbf X+\lambda}d_{\mathbf X+\lambda}$ is $\Lambda$-scalar and given for each $\omega\in \Lambda \a^*$ by the expression
\[
\Delta_{\mathbf X+\lambda}(u\otimes \omega)=-\Big(\sum_{k=1}^\kappa (\mathbf X_{A_k}+\lambda_k)^2u\Big)\otimes \omega.
\]
This shows that $\Delta_{\mathbf X+\lambda} \in \Psi^2(\mc M; E\otimes \Lambda)$ with principal symbol given by (see Example~\ref{exmpl:principal_symbols}) 
\[  
\sigma^2_p(\Delta_{\mathbf X+\lambda})(x,\xi)= \|\xi_{E_0}\|^2\Id_{E\otimes \Lambda} \textrm{ with }\|\xi_{E_0}\|^2:=\sum_{k=1}^\kappa \xi(X_{A_k})^2.
\]
It is an operator which is microlocally elliptic outside $E_u^*\oplus E_s^*$ (i.e. $\Ell^2(\Delta_{\mathbf X+\lambda})= T^*\mc M\setminus(E_u^*\oplus E_s^*)$). Thus, by Proposition~\ref{prop:elliptic_param}, if $P'\in \Psi^0(\M, E\otimes \Lambda)$ is $\Lambda$-scalar and has ${\rm WF}(P')$ contained in a conic open set of $T^*\M$ not intersecting $E_u^*\oplus E_s^*$, then there exists a $\Lambda$-scalar pseudo-differential operator $Q_\Delta(\lambda)\in \Psi^{-2}(\M; E\otimes \Lambda)$ holomorphic in 
$\lambda$ with $\WF(Q_\Delta(\lambda))\subset \WF(P')$ such that 
\[
\Delta_{\mathbf X+\lambda}Q_\Delta(\lambda)=P'+S'(\lambda)
\]
with $S'(\lambda)\in \Psi^{-\infty}(\M;E\otimes \Lambda)$ holomorphic in $\lambda$ and $\Lambda$-scalar. We now choose $P'$ so that ${\rm WF}(P')\cap (E_u^*\oplus E_s^*)=\emptyset$ and ${\rm WF}({\rm Id}-P')\cap {\rm WF}({\rm Id}-P)=\emptyset$; in other words, $P'={\rm Id}$ microlocally on ${\rm WF}({\rm Id}-P)$.
Note that $d_{\mathbf X+\lambda}\Delta_{\mathbf X+\lambda}=\Delta_{\mathbf X+\lambda}d_{\mathbf X+\lambda}$ implies that
\[
\Delta_{\mathbf X+\lambda}\left(Q_\Delta(\lambda)d_{\mathbf X+\lambda}-d_{\mathbf X+\lambda}Q_\Delta(\lambda)\right)=[P',d_{\mathbf X+\lambda}]+ [S'(\lambda),d_{\mathbf X+\lambda}].
\]
Using  microlocal ellipticity of $\Delta_{\mathbf X+\lambda}$ outside $E_u^*\oplus E_s^*$ and the fact that 
\[\WF([P',d_{\mathbf X+\lambda}])=\WF([{\rm Id}-P',d_{\mathbf X+\lambda}])\subset \WF(P')\cap \WF({\rm Id}-P'), \quad \WF([S'(\lambda),d_{\mathbf X+\lambda}])=\emptyset,\]
we deduce from \eqref{eq:ell_wf_estimate} that $\WF([Q_\Delta(\lambda),d_{\mathbf X+\lambda}])\subset \WF(P')\cap \WF({\rm Id}-P')$. In particular, since $P'={\rm Id}$ microlocally on $\WF({\rm Id}-P)$, 
this implies that 
$[Q_\Delta(\lambda),d_{\mathbf X+\lambda}]({\rm Id}-P)\in \Psi^{-\infty}(\M; E\otimes \Lambda)$.
Thus,  with $Q_{\Ell}(\lambda):=\delta_{\mathbf X+\lambda}Q_\Delta(\lambda)({\rm Id}-P)$ (mapping $C^\infty\Lambda^k$ to $C^{\infty}\Lambda^{k-1}$) we obtain
\[
\begin{split}
d_{\mathbf X+\lambda}Q_{\Ell}(\lambda)+Q_{\Ell}(\lambda)d_{\mathbf X+\lambda}=& 
\Delta_{\mathbf X+\lambda}Q_\Delta(\lambda)({\rm Id}-P)+
\delta_{\mathbf X+\lambda}[Q_\Delta(\lambda),d_{\mathbf X+\lambda}]({\rm Id}-P)\\
& +\delta_{\mathbf X+\lambda}Q_\Delta(\lambda)[d_{\mathbf X+\lambda},P]\\
=& ({\rm Id}-P)+S_1(\lambda) +S_2(\lambda)
\end{split}
\]
with $S_2(\lambda)\in \Psi^{-\infty}(\M; E\otimes \Lambda)$ and 
\[
S_1(\lambda):=\delta_{\mathbf X+\lambda}Q_\Delta(\lambda)[d_{\mathbf X+\lambda},P]=-\delta_{\mathbf X+\lambda}Q_\Delta(\lambda)[d_{\mathbf X+\lambda},{\rm Id}-P]\in \Psi^{-1}(\M;E\otimes \Lambda)
\]
has wavefront set contained in $\WF(P)\cap \WF({\rm Id}-P)$.
\end{proof}

A second ingredient for the construction of the parametrix will be the following 
estimates of the essential spectral radius of the propagator on the anisotropic Sobolev spaces.
We recall that if $Y$ is a bounded operator on a Hilbert space $\mc{H}$,
\[r_{\tu{ess}}(Y):= \max \{|\lambda| \,|\, \lambda\in \sigma_{\rm ess}(Y)\}. \]
The proof of the following Lemma is inspired by the argument in \cite{FRS08} for Anosov diffeomorphisms.
\begin{lemma}\label{smallnorm}
Let $P\in \Psi^{0}(\M; E)$ be such that $\WF(P)$ is disjoint from $E_0^*$, and choose an arbitrary constant $C_P'> C_P:=\limsup_{|\xi|\to\infty}\|\sigma^0_p(P)(x,\xi)\|$ and some $T>0$.
Let $A\in \mc W\subset \a$,  $\Gamma_{\tu{reg}}$ be an open cone disjoint from $E_u^*\subset T^*\M$, and $\Gamma_0\subset T^*\M$ be a small conic neighborhood of $E_0^*$. 
By Proposition \ref{prop:escape_fct_exist}, associated to $\Gamma_{\rm reg}$ there exists an escape function $G$ for $A_0:=A$ and 
an open conic set $\Gamma_{E_0^*}\subset \Gamma_0$ so that $G$ is 
compatible with $c_X$ and  $\Gamma_{E_0^*}$ in the sense of Definition \ref{def:escape_fct}.
If in addition $ \overline{\Gamma}_{E_0^*}\cap \Phi^{X_A}_t(\WF(P))=\emptyset$ for all $0\leq t \leq T$, then for all $0\leq t \leq T$ the operator
\[
 e^{-t\mathbf X_A}P: \mc H_{NG}\to\mc H_{NG}
\]
is bounded and can be decomposed under the form 
\[
e^{-t\mathbf X_A}P=R_{N,G}(t)+K_{N,G}(t)
\]
with $\|R_{N,G}(t)\|_{\mc{L}(\mc{H}_{NG})}\leq C'_P e^{-c_XNt}\|e^{-t\mathbf X_A}\|_{\mc{L}(L^2)}$ and $K_{N,G}(t)$ compact on $\mc{H}_{NG}$. Both $R_{N,G}(t),K_{N,G}(t)$ depend on $N,G$.
As a consequence the essential spectral radius of $e^{-t\mathbf X_A}P: \mc H_{NG}\to\mc H_{NG}$ is bounded by $ C'_P e^{-c_XNt}\|e^{-t\mathbf X_A}\|_{\mc{L}(L^2)}$.
\end{lemma}
\begin{proof}
Let $m\in C^\infty(T^*\M; \R)$ be the order function of the escape function $G$ (see Definition~\ref{def:escape_fct}(\ref{it:escfct_homog_cond})). Instead of studying $e^{-t\mathbf X_A}P$ on $\mc H_{NG}$ we consider the operator $\hat{\mc{A}}_{NG}e^{-t\mathbf X_A}P\hat{\mc{A}}_{NG}^{-1}$ on $L^2(\mc M; E)$ which is a Fourier integral operator. We write this operator as 
\begin{equation}\label{eq:proof_smallnorm_1}
 \hat{\mc{A}}_{NG}e^{-t\mathbf X_A}P\hat{\mc{A}}_{NG}^{-1} = e^{-t\mathbf X_A}\underbrace{e^{t\mathbf X_A}\hat{\mc{A}}_{NG}e^{-t\mathbf X_A}}_{=:B_t}P\hat{\mc{A}}_{NG}^{-1}.
\end{equation}
For the newly introduced operator $B_t$ we apply Egorov's Lemma (Lemma~\ref{prop:egorov}) and deduce that it is a pseudodifferential operator  $B_t\in \Psi^{N(m\circ\Phi_t^{X_A})}_{1-}(\mc M; E)$ with principal symbol
\[
 \sigma_p^{N(m\circ\Phi^{X_A}_t)}(B_t) = e^{N(G\circ \Phi_t^{X_A})} \mod S_{1-}^{N(m\circ\Phi_t^ {X_A}))-1+}. 
\]
Consequently, $B_tP\hat{\mc{A}}_{NG}^{-1} \in \Psi^{N(m\circ\Phi_t^{X_A} -m)}_{1-}$ and by Definition~\ref{def:escape_fct}(\ref{it:escfct_m_decay}) $m\circ\Phi_t^{X_A}(x,\xi) -m(x,\xi)\leq 0$ for $|\xi|$ large enough. Thus $B_tP\hat{\mc{A}}_{NG}^{-1}\in \Psi^0_{1-}(\mc M; E)$ is bounded on $L^2$, and we can apply Proposition~\ref{prop:l2-bound} to this operator. We calculate its principal symbol
\[
 \sigma^0_p(B_tP\hat{\mc{A}}_{NG}^{-1}) = e^{N(G\circ \Phi^{X_A}_t- G)}\sigma^0_p(P).
\]
Now, using Definition \ref{def:escape_fct}(\ref{it:escfct_G_decay}), our assumption that $\overline{\Gamma}_{E_0^*}\cap \Phi^{X_A}_t(\WF(P))=\emptyset$ for $0\leq t\leq T$ ensures that, for any $(x,\xi)\in \WF(P)$ and $|\xi|$ sufficiently large, $\partial_t(G\circ \Phi_t^{X_A})\leq -c_{X}$ for all $0\leq t\leq T$. Thus 
\[
 \limsup_{R\to\infty} \sup_{(x, \xi)\in \WF(P), |\xi|>R} \|e^{N(G\circ \Phi^{X_A}_t(x,\xi)- G(x,\xi))}\sigma^0_p(P)(x,\xi)\|\leq C_Pe^{-Nc_{X}t}. 
\]
By closedness of $\overline{\Gamma}_{E_0^*}$ and $\WF(P)$ this estimate can also be extended to a small conical neighborhood of $\tu{WF}(P)$. On the complement of this neighborhood, by the definition of the wavefront set, we deduce $\limsup_{|\xi|\to \infty}\|\sigma_p^0(P)(x,\xi)\|= 0$. We have seen above that $e^{N(G\circ\Phi_t^{X_A} - G)}\in S^0_{1-}$. In particular this factor is uniformly bounded. Putting everything together we get 
\[
 \limsup_{|\xi|\to\infty} \|\sigma^0_p(B_tP\hat{\mc{A}}_{NG}^{-1})(x,\xi)\|\leq C_Pe^{-Nc_Xt}.
\]
Using Proposition~\ref{prop:l2-bound} we can write $B_tP\hat{\mc{A}}_{NG}^{-1} = \widetilde{R}_N(t)+\widetilde{K}_N(t)$ with $\widetilde{K}_N(t)\in \Psi^{-\infty}(\mc M; E)$ and $\|\widetilde{R}_N(t)\|_{\mc{L}(L^2)}\leq C_P'e^{-Nc_{X}t}$. Now, by \eqref{eq:proof_smallnorm_1}, our operator of interest can be written as
\[
 \hat{\mc{A}}_{NG} e^{-t\mathbf X_A}P\hat{\mc{A}}_{NG}^{-1} = e^{-t\mathbf X_A}(\widetilde{R}_N(t)+\widetilde{K}_N(t)), 
\]
and we get the desired property by setting $R_{N,G}(t):=e^{-t\mathbf X_A}\widetilde{R}_N(t)$ and $K_{N,G}(t):=e^{-t\mathbf X_A}\widetilde{K}_N(t)$. 
\end{proof}

Recall that $C_{L^2}(A)$ was defined in Formula \eqref{eq:def-CL^2(A)}. We can now come to the construction of our full parametrix for the Taylor complex:

\begin{prop}\label{prop:Fredholm}
For any $A_0\in \mathcal W$, any open cone $\Gamma_{0}\subset T^*\mc M$ containing $E_0^*$ and satisfying $\overline{\Gamma_0}\cap (E_u^*\oplus E_s^*) = \{0\}$, there are families of operators
$Q(\lambda),F(\lambda):C^\infty(\M;E\otimes \Lambda )\to C^{-\infty}(\M;E\otimes \Lambda)$ depending holomorphically on $\lambda\in\a_{\C}^*$ such that 
\[ 
Q(\lambda)d_{\mathbf X+\lambda}+d_{\mathbf X+\lambda} Q(\lambda)=F(\lambda).
\]
Furthermore, for any escape function $G$ for $A_0$ compatible with $c_{X}>0$, 
and $\Gamma_{E_0^*}\subset \Gamma_0$, one has for any $N>0$ and $\delta>0$ that:
\begin{enumerate}
 \item $Q(\lambda): \mc H_{NG}\Lambda^j\to \mc H_{NG}\Lambda^{j-1}$ is bounded for any $\lambda \in \a^*_\C$ and $0\leq j\leq\kappa$,
 \item $F(\lambda)$ can be decomposed as $F(\lambda)={\rm Id}+R_{N,G}(\lambda)+K_{N,G}(\lambda)$ where $K_{N,G}(\lambda)$ is a compact operator on $\mc{H}_{NG}\Lambda$,
and $R_{N,G}(\lambda): \mc H_{NG}\Lambda\to \mc H_{NG}\Lambda$ is bounded with norm $\|R_{N,G}(\lambda)\|_{\mc{L}(\mc{H}_{NG})}<1/2$ for 
 \[
 \lambda \in \mc F_{NG, A_0, \delta} :=\{\lambda\in \a_\C^*\, |\,\Re(\lambda(A_0))>-N c_{X} + C_{L^2}(A_0)+\delta\}\subset \a_\C^*.
 \]
 Both operators $R_{N,G}(\lambda),K_{N,G}(\lambda)$ depend on $N,G$, while $Q(\lambda)$ and $F(\lambda)$ do not. 
\end{enumerate}
\end{prop}
\smallskip

\begin{rem}
~
\begin{enumerate}
 \item If there is a smooth volume density $\mu$ preserved by the Anosov action (e.g. the Haar measure for Weyl chamber flows), and if we consider the scalar case $\mathbf X_A = X_A$, then $e^{tX_A}$ is unitary on $L^2(\mc M,\mu)$ and the constant $C_{L^2}(A)$ vanishes. 
 \item For proving that the Ruelle-Taylor spectrum is independent of the choice of $\mc H_{NG}$ it will be essential that the operators $Q(\lambda), F(\lambda)$ will only depend on the choice of $A_0$ and $\Gamma_{E_0^*}$ but not on the choice of the anisotropic Sobolev space $\mc H_{NG}$ as long as the escape function $G$ satisfies the required compatibility conditions. 
\end{enumerate}

\end{rem}
\begin{proof}[Proof of Proposition~\ref{prop:Fredholm}]
From the definition of $C_{L^2}(A)$, we deduce that there exists $T_0>0$ so that for $T\geq T_0$, $\| e^{-T\mathbf{X}_{A_0}}\| \leq e^{T( C_{L^2}(A_0) + \delta/2)}$; we choose $T$ so that both $T>T_0$ and $T\geq 2\ln(3)/\delta$. For $\lambda\in \a_\C^*$ we define the operators $\mathbf X_{A_0}(\lambda):=\mathbf X_{A_0}+\lambda(A_0)$ and let
\[
Q'_T(\lambda):=\int_0^T e^{-t\mathbf X_{A_0}(\lambda)}dt : C^\infty(\M; E)\to C^\infty(\M; E).
\]
We have the relations 
\begin{equation}\label{relationQ_T} 
\mathbf X_{A_0}(\lambda) Q'_T(\lambda)=Q'_T(\lambda)\mathbf X_{A_0}(\lambda)=1- e^{-T\mathbf X_{A_0}(\lambda)} , \quad  [\mathbf X_A , Q'_T(\lambda)]=0 \tu{ for all }A\in\a.
\end{equation}
Now we extend $Q'_T(\lambda)$ to an operator on $C^\infty(\mc M; E)\otimes\Lambda^\ell \a^*\to C^\infty(\mc M; E)\otimes\Lambda^{\ell-1}\a^*$ for each $\ell$ as follows: define the linear map ${\bf Q}_T(\la): \a\to \mc{L}(C^\infty(\mc{M};E))$ by ${\bf Q}_T(\la)A_0
=Q_T'(\la)$ and ${\bf Q}'_T(\la)A=0$ if $\cjg A,A_0\cjd=0$ (recall $\cjg \cdot,\cdot\cjd$ is a fixed scalar product on $\a$), and let
\[
Q_T(\lambda)(u\otimes \omega):=-\delta_{{\bf Q}_T(\lambda)}(u\otimes \omega)= (Q'_T(\lambda)u)\otimes \iota_{A_0}\omega
\]
for $u\in C^\infty(\M; E)$ and $\omega\in\Lambda^\ell\a^*$.
Using the relations of \eqref{relationQ_T} and Lemma \ref{trick} we get
\begin{equation}\label{Qd+dQ}
\begin{split} 
(Q_T(\lambda)d_{\mathbf X+\lambda}+d_{\mathbf X+\lambda} Q_T(\lambda))(u\otimes \omega)=& \big((1 - e^{-T\mathbf X_{A_0}(\lambda)})u\big)\otimes\omega.
\end{split}\end{equation}
We observe that by the commutativity of the Anosov action  
$[\mathbf X_A,e^{-T\mathbf X_{A_0}(\lambda)}]=0$, and therefore on $C^\infty(\M;E\otimes \Lambda)$ we have 
\begin{equation}\label{eq:commut_dX_eT}
[d_{\mathbf X+\lambda},e^{-T\mathbf X_{A_0}(\lambda)}]=0.
\end{equation}
Next, we use the microlocal parametrix in the elliptic region from Lemma~\ref{ellipticregion} with a carefully chosen microlocal cutoff function: 
By our assumption that $\overline{\Gamma}_{0}\cap (E_u^*\oplus E_s^*)=\{0\}$ and the fact that $E_u^*\oplus E_s^*$ is a $\Phi_t^{X_{A_0}}$-invariant subset, there exists a conic neighborhood $\Gamma_1\subset T^*\mc M$ of $E_u^*\oplus E_s^*$ such that $\Phi_t^{X_{A_0}}(\Gamma_1) \cap\overline{\Gamma_0}=\emptyset$ of $0\leq t\leq T$. 
Let us choose a second, smaller conical neighborhood $E_u^*\oplus E_s^*\subset \Gamma_2\Subset\Gamma_1$. Now we fix a microlocal cutoff $P=\Op(p)\in \Psi^0(\M, \C)$ which is microsupported in $\Gamma_1$ (i.e. $\WF(P)\subset \Gamma_1$) and microlocally equal to one on $\Gamma_2$ (i.e. $\WF({\rm Id}-P)\cap \Gamma_2=\emptyset$) and which furthermore has globally bounded symbol $\sup_{(x,\xi)}|p(x,\xi)|\leq 1$.
We apply Lemma~\ref{ellipticregion} with this choice of $P$ and multiply \eqref{eq:elliptic_param} from the left with $e^{-T\mathbf X_{A_0}(\lambda)}$. Using \eqref{eq:commut_dX_eT}, we get
\begin{equation}\label{Qelld+dQell}
d_{\mathbf X+\lambda} e^{-T\mathbf X_{A_0}(\lambda)}Q_{\Ell}(\lambda)+e^{-T\mathbf X_{A_0}(\lambda)}Q_{\Ell}(\lambda)d_{\mathbf X+\lambda} = e^{-T\mathbf X_{A_0}(\lambda)}\big({\rm Id}-P+S_1(\lambda)+S_2(\lambda)\big).
\end{equation}
We define $Q(\lambda):= Q_T(\lambda)+e^{-T\mathbf X_{A_0}(\lambda)}Q_{\Ell}(\lambda)$ and obtain by adding up \eqref{Qd+dQ} and \eqref{Qelld+dQell}
\[
d_{\mathbf X+\lambda} Q(\lambda)+Q(\lambda)d_{\mathbf X+\lambda} = F(\lambda)\tu{ with }F(\lambda):={\rm Id}-e^{-T\mathbf X_{A_0}(\lambda)}\big(P-S_1(\lambda)-S_2(\lambda)\big).
\]
Let us now show that $Q(\lambda)$ and $F(\lambda)$ have the required properties. 
By precisely the same argument as in Lemma~\ref{smallnorm} (using that $X_{A_0}^Hm(x,\xi) \leq 0$ for $|\xi|$ large enough) we deduce that  
$e^{-t\mathbf X_{A_0}}$ is bounded on $\mc H_{NG}$ uniformly for $t\in [0,T]$ for any escape function $G$ associated to $A_0$ compatible with $c_X>0$ and $\Gamma_{E_0^*}\subset \Gamma_0$.
Consequently $Q_T(\lambda)$ and $e^{-T\mathbf X_{A_0}(\lambda)}$ are bounded operators on $\mc H_{NG}\Lambda$. As $\hat{\mc{A}}_{NG}Q_{\Ell}(\lambda)\hat{\mc{A}}_{NG}^{-1}\in \Psi^{-2}(\mc M; E\otimes \Lambda)$, this operator is a bounded operator on $L^2$, thus $Q_{\Ell}(\lambda)$ is bounded on $\mc H_{NG}\Lambda$ as well. 
Putting everything together we deduce that $Q(\lambda)$ is bounded on $\mc H_{NG}\Lambda$ for any $\lambda\in \a_\C^*$. 
As $Q_T(\lambda)$ and $Q_{\Ell}(\lambda)$ decrease the order in $\Lambda\a^*$ by one, $Q(\lambda)$ has this property as well.

Let us deal with $F(\lambda)$: by our choice of $\Gamma_1$ we can apply Lemma~\ref{smallnorm} to 
$e^{-T\mathbf X_{A_0}(\lambda)}P = e^{-T\lambda(A_0)}e^{-T\mathbf X_{A_0}}P $ and deduce that the $e^{-T\mathbf X_{A_0}(\lambda)}P=R_N'(\lambda)+K_N'(\lambda)$ for some $R_N'(\lambda)$ bounded on
$\mc H_{NG}$ and $K_N'(\lambda)$ compact on that space, with bound 
\[
\|R_N'(\lambda)\|_{\mc{L}(\mc{H}_{NG})} \leq (1+\varepsilon)e^{T\left(-Nc_X - \Re(\lambda(A_0))+ C_{L^2}(A_0) + \delta/2\right)},
\]
for some $\varepsilon>0$. Consequently, by our choice of $T> 2 \ln(3)/\delta$ and for $\lambda \in \mathcal F_{NG, A_0,\delta}$  we get $\|R_N'(\lambda)\|_{\mc{L}(\mc{H}_{NG})}\leq (1+\varepsilon)/3$. Note that 
$S_1(\lambda)+S_2(\lambda)\in \Psi^{-1}(\mc M; E\otimes \Lambda)$ is compact on $\mc{H}_{NG}$ (this can be easily checked by conjugating it with $\hat{\mc{A}}_{NG}$ to obtain an operator in $\Psi^{-1}(\mc M; E\otimes \Lambda)$, thus compact on $L^2$). This completes the proof of Proposition~\ref{prop:Fredholm}
by setting $R_N(\lambda):=-R'_N(\lambda)$ and $K_N(\lambda):=-K'_N(\lambda)+e^{-T\mathbf X_{A_0}(\lambda)}(S_1(\lambda)+S_2(\lambda))$.
\end{proof}

As a consequence  we get:
\begin{prop}\label{prop:fredholm_spec}
 For $A_0\in \mathcal W$ there exists an escape function $G$ such that
 for any $N>0$ the operator $d_{\mathbf X+\lambda}$ on $\mc H_{NG}\Lambda$ defines a Fredholm complex for $\lambda\in\mathcal F_{NG, A_0, 0}$, i.e. we have 
  \[
  \sigma^{\tu{ess}}_{\T, \mc H_{NG}}(-\mathbf X)\cap \mc F_{NG, A_0,0}=\emptyset.
  \]
\end{prop}
\begin{proof} By Proposition~\ref{prop:escape_fct_exist}, there is an escape function $G$ that allows to obtain 
Proposition~\ref{prop:Fredholm}. Then we can use Lemma~\ref{lem:critereFredholm} applied to $\X+\lambda$ to deduce Fredholmness.
\end{proof}

\subsection{Ruelle-Taylor resonances are intrinsic}\label{sec:intrinsic_spec}
So far we have shown that the admissible lift of an Anosov action $X$ acting as differential operators on $\mc H_{NG}$ has a Fredholm Taylor spectrum on $\mathcal F_{NG, A}:=\mathcal F_{NG, A,0}\subset \a^*_\C$, where $A\in \W$ and $G$ is an escape function associated to $A$. Further, we have seen that $\mathcal F_{NG, A}$ can be made arbitrarily large by letting $N\to\infty$. 
However, it is not yet clear if this Fredholm spectrum is intrinsic to $\mathbf X$ or if it depends on the choice of the anisotropic Sobolev spaces $\mathcal H_{NG}$, i.e. in particular on the choices of $N$ or $G$. 

Let us denote by $C^{-\infty}_{E_u^*}(\mc M; E)$ the space of distributions in $C^{-\infty}(\mc M; E)$ with wavefront set contained in $E_u^*$. Equipped with a suitable topology this space becomes a topological vector space \cite[Chapter 8]{Hoe03}, and the lift $\mathbf X$ acts continuously on $C^{-\infty}_{E_u^*}(\mc M; E)$. In particular, it makes sense to consider the complex generated by the operator $d_{\mathbf X+\lambda}$ on $C^{-\infty}_{E_u^*}(\mc M; E\otimes \Lambda)$. The main result of this section is the following:
\begin{prop}\label{prop:D_Eu}
 Let $A_0\in \mc W$, $N\geq0$ and $G$ be an escape function for $A_0$. 
 Then for any $\lambda \in \mc F_{NG,A_0}$ one has vector space isomorphisms 
 \[
  \ker_{\mc H_{NG}\Lambda^j} d_{\mathbf X+\lambda}/\ran_{\mc H_{NG}\Lambda^j} d_{\mathbf X+\lambda} \cong \ker_{C^{-\infty}_{E_u^*}\Lambda^j} d_{\mathbf X+\lambda}/\ran_{C^{-\infty}_{E_u^*}\Lambda^j} d_{\mathbf X+\lambda}.
 \]
\end{prop}
Using this result, we see that the Ruelle-Taylor spectrum is independent of $A_0$ and of the anisotropic space $\mc{H}_{NG}\Lambda$ in the region $\mc{F}_{NG,A_0}$ of $\lambda \in \a^*_\C$ where the Taylor complex $d_{\mathbf X+\lambda}$ is Fredholm on $\mc{H}_{NG}\Lambda$. We can then define the notion of a Ruelle-Taylor resonance as follows:
\begin{Def}\label{def:ruelle-taylor-res}
We define the \emph{Ruelle-Taylor resonances} of $\mathbf X$ to be the set
\[ {\rm Res}_\X:=\{ \lambda\in \a^*_\C\, |\, \exists j, \, \ker_{C^{-\infty}_{E_u^*}\Lambda^j} d_{\mathbf X+\lambda}/\ran_{C^{-\infty}_{E_u^*}\Lambda^j} d_{\mathbf X+\lambda}\not=0\}\]
and the \emph{Ruelle-Taylor resonant cohomology space} of degree $j$ of $\lambda\in {\rm Res}_\X$ to be 
\[{\rm Res}_{\X,\Lambda^j}(\lambda):=\ker_{C^{-\infty}_{E_u^*}\Lambda^j} d_{\mathbf X+\lambda}/\ran_{C^{-\infty}_{E_u^*}\Lambda^j} d_{\mathbf X+\lambda}.\]  
\end{Def}

Another consequence of Proposition~\ref{prop:D_Eu} is:
\begin{cor}[Location of Ruelle-Taylor resonances]\label{cor:location}
 One has 
 \[
  \mathrm{Res}_\X \subset \bigcap_{A\in \mathcal W} \{\lambda\in\a_\C^*~|~\Re(\lambda(A))\leq C_{L^2}(A)\}
 \]
\end{cor}
\begin{proof}
 Assume that there exists an $A\in \mathcal W$ such that $\tu {Re}(\lambda(A))>C_{L^2}(A)$. Then for some $\delta>0$, $\lambda\in \mathcal F_{0G, A,\delta}$ and consequently $\lambda\in \mathrm{Res}_\X$ iff $\ker_{L^2\Lambda}d_{\mathbf X+\lambda}/\ran_{L^2\Lambda}d_{\mathbf X+\lambda} \neq 0$.
 However, by \eqref{Qd+dQ} there is a bounded operator $Q_T(\lambda): L^2(\mc M; E\otimes \Lambda) \to L^2(\mc M; E\otimes \Lambda)$ such that 
 \[
  d_{\mathbf X+\lambda}Q_T(\lambda) + Q_T(\lambda) d_{\mathbf X+\lambda} = {\rm Id} + e^{-T\mathbf X_A}e^{-T\lambda(A)}.
 \]
Since $\Re(\lambda(A))>C_{L^2}(A)$, the right hand side is invertible on $L^2(\mc M; E\otimes \Lambda)$ provided $T>0$ is large enough. As furthermore ${\rm Id}+e^{-T\mathbf X_A}e^{-T\lambda(A)}$ and its inverse commute with $d_{\mathbf X+ \lambda}$ we conclude that
$\ker_{L^2\Lambda}d_{\mathbf X+\lambda}/\ran_{L^2\Lambda}d_{\mathbf X+\lambda} = 0$.
\end{proof}
The strategy to prove Proposition~\ref{prop:D_Eu} is to show that in each cohomology class in $\ker_{\mc H_{NG}\Lambda} d_{\mathbf X+\lambda}/\ran_{\mc H_{NG}\Lambda} d_{\mathbf X+\lambda}$ one can find a representative that lies already in $\ker_{C^{-\infty}_{E_u^*}} d_{\mathbf X+\lambda}$. To this end we will construct for fixed $\lambda$ a projector $\Pi_0(\lambda)$ of finite rank such that we can find in each cohomology class a representative in the range of $\Pi_0(\lambda)$. The fact that the range of $\Pi_0(\lambda)$ is independent of the anisotropic Sobolev spaces and contained in $C^{-\infty}_{E_u^*}$ then follows very similarly to the corresponding characterization of Anosov flows \cite[Theorem 1.7]{FS11} by the flexibility in the choice of the escape function. 
\begin{proof}[Proof of Proposition~\ref{prop:D_Eu}]
Given $A_0, N, G$ and $\lambda\in\mc F_{NG,A_0}$ let us first fix $\delta>0$ such that $\lambda\in\mc F_{NG, A_0, \delta}$ and an open cone $\Gamma_0 \subset T^*\mc M$ containing $\Gamma_{E_0^*}$ (the conic set in Proposition \ref{prop:escape_fct_exist})  and such that $\overline{\Gamma}_0\cap (E_s^*\oplus E_u^*)=\{0\}$. Then Proposition~\ref{prop:Fredholm} provides  operators $Q(\lambda),F(\lambda):C^{-\infty}(\mc M; E\otimes \Lambda) \to  C^{-\infty}(\mc M; E\otimes \Lambda)$ which only depend on $\delta,\lambda, A_0, \Gamma_0$ and satisfy
\begin{equation}\label{eq:intrinsic_dQ+Qd}
 d_{\mathbf X+\lambda}Q(\lambda)+Q(\lambda)d_{\mathbf X+\lambda}= F(\lambda).
\end{equation}
We can thus apply Lemma \ref{lem:critereFredholm}, and deduce that if $\Pi_0(\lambda)$ is the spectral projector of $F(\lambda)$ on its kernel,  
\begin{equation}\label{eq:isom_P_0}
\Pi_0(\lambda): \ker_{\mc H_{NG}\Lambda} d_{\mathbf X+\lambda}/\ran_{\mc H_{NG}\Lambda}d_{\mathbf X+\lambda}\to \ker_{\ran \Pi_0(\lambda)} d_{\mathbf X+\lambda}/\ran_{\ran \Pi_0(\lambda)}d_{\mathbf X+\lambda}
\end{equation}
is an isomorphism. Here, $\ran(\Pi_0(\lambda)) = \Pi_0(\lambda)\mathcal{H}_{NG}$. But since $C^\infty$ is dense in $\mathcal{H}_{NG}$, and $\Pi_0(\lambda)$ has finite rank, we deduce that it is equal to $\Pi_0(\lambda)C^\infty(\mathcal{M}; E\otimes \Lambda)$. We now need the following lemma:
\begin{lemma}\label{lem:P1_properties}~
The projector $\Pi_0(\lambda)$ satisfies 
 $\Pi_0(\lambda)(C^\infty(\M;E\otimes \Lambda)) \subset C^{-\infty}_{E_u^*}(\M;E\otimes \Lambda)$. Additionally, it has a continuous extension to $C^{-\infty}_{E_u^*}(\M;E\otimes \Lambda)$.
\end{lemma}
\begin{proof}
Recall that $\Pi_0(\lambda):\mc H_{NG}\Lambda \to\mc H_{NG}\Lambda$ has been defined as the spectral projector at $z=0$ of $F(\lambda):\mc H_{NG}\Lambda \to\mc H_{NG}\Lambda$, it has finite rank. Since $F(\lambda)$ and its Fredholmness do not depend on the choice of $N$, $G$, as long as $\lambda\in \mathcal{F}_{NG,A_0}$, neither does its spectral projector at $0$. The image of $\Pi_0(\lambda)$ is thus contained in the intersection of the $\mathcal{H}_{N'G'}\Lambda$ such that $\lambda\in \mathcal{F}_{N'G',A_0}$. 

Let us show that this intersection is contained in $C^{-\infty}_{E_u^*}(\M;E\otimes \Lambda)$. We thus take $u$ in all the $\mathcal{H}_{N'G'}$ such that $\lambda\in \mathcal{F}_{N'G',A_0}$. By Proposition~\ref{prop:escape_fct_exist} for an arbitrary cone $\Gamma'_\tu{reg}$ disjoint from $E^\ast_u$, there exists an escape function $G'$ for $A_0$ compatible with $c'_X$ and $\Gamma'_{E_0^*}\subset \Gamma_0$ such that
 microlocally on $\Gamma'_\tu{reg}$, $\mc H_{N'G'}$ is contained in the standard Sobolev space $H^{N'/2}(\mc M; E)$. In particular, taking $N'$ arbitrarily large, $\lambda\in \mathcal{F}_{N'G',A_0}$ and $\WF(u) \cap \Gamma'_{\tu{reg}} = \emptyset$. Since $\Gamma'_{\tu{reg}}$ was arbitrary, $\WF(u)\subset E^\ast_u$.

To prove that $\Pi_0(\lambda)$ has a continuous extension to $C^{-\infty}_{E_u^*}(\M;E\otimes \Lambda)$, it suffices to observe that $C^{-\infty}_{E_u^*}(\M;E\otimes \Lambda)$ is also contained in the union of the all the $\mathcal{H}_{N'G'}$ such that 
$\lambda\in\mathcal{F}_{N'G',A_0}$. This follows from Definition~\ref{def:escape_fct},(\ref{it:escfct_homog_cond}), since we know that in a conic neighborhood around $E_u^*$ we have $m(x,\xi) \leq -1/4$. As a consequence, $\Pi_0(\lambda)$ is a linear operator from $C^{-\infty}_{E_u^*}(\M;E\otimes \Lambda)$ to $C^{-\infty}(\mathcal{M},E\otimes \Lambda)$. It is continuous as it has finite rank. 
\end{proof}

To finish the proof of Proposition \ref{prop:D_Eu}, it suffices to apply a variation of the sandwiching trick presented in the proof of Lemma \ref{lemma:discrete}. Indeed, since $\Pi_0(\lambda)$ is a bounded projector on $C^{-\infty}_{E_u^*}(\M;E\otimes \Lambda)$, commuting with $d_{\X+\lambda}$, $u\mapsto \Pi_0(\lambda)u$ factors to a surjective map
\begin{equation}\label{eq:intrinsic_isom2}
\ker_{C^{-\infty}_{E_u^*}\Lambda} d_{\mathbf X+\lambda}/\ran_{C^{-\infty}_{E_u^*}\Lambda}d_{\mathbf X+\lambda} \to \ker_{\ran \Pi_0(\lambda)} d_{\mathbf X+\lambda}/\ran_{\ran \Pi_0(\lambda)}d_{\mathbf X+\lambda}
\end{equation}
We need to show the injectivity of this map. This will follow from the fact that $C^{-\infty}_{E_u^*}(\M;E\otimes \Lambda)$ is contained in the union of the $\mathcal{H}_{N'G'}\Lambda$ such that $\lambda\in\mathcal{F}_{N'G',A_0}$. We consider $u\in C^{-\infty}_{E_u^*}(\M;E\otimes \Lambda)$ such that $d_{\X+\lambda}u=0$, and $[\Pi_0(\lambda)u]=0$, i.e, $\Pi_0(\lambda)u = d_{\X + \lambda} \Pi_0(\lambda) v$ for some $v\in C^{-\infty}_{E_u^*}(\M;E\otimes \Lambda)$. Since $u$ belongs to some $\mathcal{H}_{N'G'}\Lambda$, we then write $\tilde{F}(\lambda) = F(\lambda) + \Pi_0(\lambda)$, and observe, just as in \eqref{Id-Pi0}, that
\[
\tilde{F}(\lambda)^{-1}Q(\lambda) d_{\X+\lambda} + d_{\X + \lambda} \tilde{F}(\lambda)^{-1} Q(\lambda) = {\rm Id} - \Pi_0(\lambda),
\] 
so that 
\[
u= d_{\X + \lambda} \big(\tilde{F}(\lambda)^{-1} Q(\lambda)u + \Pi_0(\lambda)v\big).
\]
It remains to check that $\tilde{F}^{-1}(\lambda) Q(\lambda)u\in C^{-\infty}_{E_u^*}(\M;E\otimes \Lambda)$. But, since $Q(\lambda)$ and $\tilde{F}^{-1}(\lambda)$ are bounded on each $\mathcal{H}_{N'G'}\Lambda$ such that $\lambda\in \mathcal{F}_{N'G',A_0}$, this is an element of each such  $\mathcal{H}_{N'G'}\Lambda$, so it is contained in the intersection thereof. We have seen in the proof of Lemma \ref{lem:P1_properties} that this intersection is contained in $C^{-\infty}_{E_u^*}(\M;E\otimes \Lambda)$.

Finally, note that the operator $F(\lambda): \mc H_{NG}\Lambda \to \mc H_{NG}\Lambda$ preserves the order in the Koszul complex, i.e. $F(\lambda): \mc H_{NG}\Lambda^j \to \mc H_{NG}\Lambda^j$, and all the subsequent constructions such as $\Pi_0(\lambda)$ do as well. The isomorphism $\Pi_0(\lambda)$ can thus be restricted to the individual cohomology $\ker_{C^{-\infty}_{E_u^*}\Lambda^j} d_{\mathbf X+\lambda}/\ran_{C^{-\infty}_{E_u^*}\Lambda^j}d_{\mathbf X+\lambda}$ and we have completed the proof of Proposition~\ref{prop:D_Eu}.
\end{proof}

\subsection{Discrete Ruelle-Taylor spectrum}\label{sec:discrete_spec}
In this section we show that the Ruelle-Taylor resonance spectrum of the admissible lift $\X:\a\mapsto {\rm Diff}^1(\M;E)$ of the Anosov action, for $E$ a Riemannian vector bundle,  
 is discrete in $\a_\C^*$. Our goal is to use  Lemma \ref{lemma:discrete}. In contrast to just obtaining the Fredholm property of the Taylor complex, this section requires to use a parametrix $Q(\lambda)$ in Proposition \ref{prop:Fredholm} that is more intrinsically related to the $\X$ action, in particular we shall construct $Q(\lambda)$ as a function of $(X_1,\dots,X_\kappa)=({\bf X}_{A_1},\dots,{\bf X}_{A_\kappa})$ if $A_j\in \mc{W}$ is an orthonormal basis for some scalar product $\cjg\cdot,\cdot\cjd$ on $\a$. This requires to use the slightly better escape function of Lemma \ref{unifG} that provides decay not only in a fixed direction $A_1\in \mathcal W$, but also for all other elements in a small neighborhood $\mc{U}$ of $A_1$.
Let us now fix an orthonormal basis $A_1,\dots,A_\kappa\in \mc{U}\subset \W$ of $\a$ in the positive Weyl chamber, and we denote the associated scalar product in $\a$ by $\cjg\cdot,\cdot \cjd$. In order to be able to use  Lemma \ref{lemma:discrete}, we will prove the following:
\begin{lemma}\label{lemmafordiscreteness}
For each  fixed $\lambda \in \mc F_{NG, A_0, \delta}$ there is a Lie algebra morphism ${\bf Q}(\lambda):\a\to \mc{L}(\mc{H}_{NG})\cap \mc{L}(C^\infty(\M;E))$ and  commuting with $\X(\lambda):={\bf X}+\lambda$ in the sense that $[\X_{A_j}(\lambda),{\bf Q}_{A_k}(\lambda)]=0$ for all 
$j,k$, such that 
\[d_{\X+\lambda}\delta_{{\bf Q}(\lambda)}+\delta_{{\bf Q}(\lambda)}d_{\X+\lambda}={\rm Id}+R(\lambda)+K(\lambda)\]
with $R(\lambda),K(\lambda)\in \mc{L}(\mc{H}_{NG}\Lambda)$, $\|R(\lambda)\|_{\mc{L}(\mc{H}_{NG})}<1/2$ and $K(\lambda)$ compact on $\mc{H}_{NG}\Lambda$. Moreover $R(\lambda),K(\lambda)$ are $\Lambda$-scalar. 
\end{lemma}
\begin{proof}
Let $T_j>0$ for $j=1,\dots,\kappa$, and consider $\chi_{j}\in C_c^\infty([0,\infty[; [0,1])$ non-increasing with $\chi_j=1$ in $[0,T_j]$ and $\supp \chi_j\in [0,T_j+1] $. Then we set
\[
Q_j'(\lambda) := \int_0^\infty e^{-t_j\X_{A_j}(\lambda)}\chi_{j}(t_j)dt_j
\]
and we make it act on $C^\infty(\M;E)\otimes \Lambda\a^*$ 
by $\widetilde{Q}_j(\lambda):u\otimes w\mapsto (Q_j'(\lambda)u)\otimes \iota_{A_j}\omega$.
As in Proposition \ref{prop:Fredholm}, we compute 
\[
\ d_{\X(\lambda)}\widetilde{Q}_j(\lambda)+\widetilde{Q}_j(\lambda)d_{\X(\lambda)}= {\rm Id}+R_j(\lambda), \quad R_j(\lambda)(u\otimes \omega):=\left(\int_0^\infty e^{-t_jX_{j}(\lambda)}u\chi'_j(t_j)dt_j \right)\otimes w.
\]
and note that $R_j(\lambda)=R_j'(\lambda)\otimes {\rm Id}$ is scalar. We thus have
\begin{equation}\label{dXQ+QdX0} 
d_{\X(\lambda)}Q(\lambda)+Q(\lambda)d_{\X(\lambda)}= F(\lambda), \quad F(\lambda):={\rm Id} -(-1)^{\kappa}\prod_{j=1}^\kappa R_j(\lambda).
\end{equation}
with $Q(\lambda):= \sum_{j=1}^\kappa (-1)^{j-1}\widetilde{Q}_j(\lambda) \prod_{k=1}^{j-1} R_k(\lambda)$.
First we observe that $Q(\lambda)=\delta_{{\bf Q}(\lambda)}$ is the divergence associated to  the Lie algebra morphism 
${\bf Q}(\lambda):\a \to \mc{L}(C^\infty(\M;E))$ defined by 
\[ {\bf Q}_{A_j}(\lambda)=(-1)^jQ'_j(\lambda)\prod_{k=1}^{j-1} R_k(\lambda).\]
We notice that ${\bf Q}_{A_j}(\lambda)$ commutes with ${\bf X}_{A_i}(\lambda)$ for each $i,j$. 
As in the proof of Proposition~\ref{prop:Fredholm}, we have that ${\bf Q}(\lambda)$ maps to $\mc{L}(\mc{H}_{NG})$  and $Q(\lambda)$ is bounded on $\mc{H}_{NG}\Lambda$: notice that here we use Lemma \ref{unifG} as it is important that the order function $m$ satisfies $X_{A_j}^Hm\leq 0$ for $|\xi|$ large enough and for all $j=1,\dots,\kappa$. We take $P$ microsupported in a neighbourhood of $E^\ast_u\oplus E^\ast_s$ and $\WF(P)$ in a sufficiently close conical neighbourhood of $E_u^*\oplus E_s^*$, as in the proof of Proposition~\ref{prop:Fredholm}, and follow the arguments given there, which were based on Lemma \ref{smallnorm}: if $T_j:=T$ is chosen large enough (as in proof of Proposition~\ref{prop:Fredholm}) 
\[
\begin{split}
\prod_{k=1}^\kappa R_k(\lambda)P=& \int_{[T,T+1]^\kappa} e^{-\sum_{j=1}^\kappa t_j\X_{A_j}(\lambda)}P\prod_{j=1}^\kappa \chi'_j(t_j)dt\\
=&  \int_{[T,T+1]^\kappa} (R(t,\lambda)+K(t,\lambda))\prod_{j=1}^\kappa \chi'_j(t_j)dt,
\end{split}
\] 
where $\|R(t,\lambda)\|_{\mc{L}(\mc{H}_{NG})}\prod_j \|\chi_j'\|_{L^\infty}\leq 1/2 $ and $K(t,\lambda)$ is compact on $\mc{H}_{NG}$ for all $t\in [T,T+1]^\kappa$ (both depend on $N,G$). 
This shows that the operator $(F(\lambda)-{\Id})P$ decomposes as $(F(\lambda)-{\Id})P=R(\lambda)+K_1(\lambda)$ 
with $\|R(\lambda)\|_{\mc{L}(\mc{H}_{NG}\Lambda)}<1/2$ and $K_1(\lambda)$ compact on $\mc{H}_{NG}\Lambda$.
Next, we claim that using that  $P\in \Psi^0(\M)$ is scalar with ${\rm WF}({\rm Id}-P)$ not intersecting a conic neighborhood of $E_u^*\oplus E_s^*$, we see that $K_2(\lambda):=(F(\lambda)-{\rm Id})({\rm Id}-P)$ is a compact operator on $\mc{H}_{NG}\Lambda$. Indeed,  let us first take a microlocal partition of $({\rm Id}-P)$ so that $({\rm Id}-P)-\sum_{k=1}^\kappa P_k\in \Psi^{-\infty}(\M)$ 
with $P_k\in \Psi^{0}(\M)$ and ${\rm WF}(P_k)$ not intersecting a conic neighborhood of the characteristic set $\{(x,\xi)\in T^*\M\, |\, \xi(X_{A_k})=0\}$.
Let us show that $R_k(\lambda)P_k$ is compact on $\mc{H}_{NG}$: first,  
\begin{equation}\label{Rkpkcomp}
R_k(\lambda)P_k\X_{A_k}(\lambda)=\int_{T}^{T+1} e^{-t_k{\bf X}_{A_k}(\lambda)}P_k \chi''_k(t_k)dt_k+
R_k(\lambda)[P_k,\X_{A_k}] \in \mc{L}(\mc{H}_{NG}),\end{equation}
where we used that $[P_k,\X_{A_k}]\in \Psi^0(\M)$ and that $e^{-t_k{\bf X}_{A_k}(\lambda)}$ is bounded on $\mc{H}_{NG}$.
Since $\X_k(\lambda)$ is elliptic near ${\rm WF}(P_k)$, we can construct a parametrix 
$Z_k(\lambda)\in \Psi^{-1}(\M)$ so that $\X_{A_k}(\lambda)Z_k(\lambda)-P'_k\in\Psi^{-\infty}(\M)$
for some $P_k'\in \Psi^0(\M)$ with $P_k'P_k-P_k\in \Psi^{-\infty}(\M)$. 
We thus obtain that
\[R_k(\lambda)P_k\X_{A_k}(\lambda)Z_k(\lambda)-R_k(\lambda)P_k\in \Psi^{-\infty}(\M),\] 
but $Z_k(\lambda)$ being compact on $\mc{H}_{NG}$, we get that $R_k(\lambda)P_k$ is compact on $\mc{H}_{NG}$  using \eqref{Rkpkcomp}.
Next, we write 
 \[
 \Big(\prod_{k=1}^\kappa R_k(\lambda)\Big)({\rm Id}-P)-\sum_{j=1}^\kappa  \Big(\prod_{k=1}^\kappa R_k(\lambda)\Big) P_j \in \Psi^{-\infty}(\mc M).
 \]
The former operator is compact since all the $R_k(\lambda)$ are bounded on $\mc{H}_{NG}$ and commute with each other and $R_k(\lambda)P_k$ is compact.
Putting everything together we deduce that $F(\lambda)$ has the desired properties by setting $K(\lambda):=K_1(\lambda)+K_2(\lambda)$. 
 \end{proof}
 \textbf{Remark}: we notice that in the proof above, it is sufficient to take only one of the $T_j$ to be large while the others can be small, as this is sufficient to get the norm estimate $\|R(\lambda)\|_{\mc{L}(\mc{H}_{NG})}<1/2$.
 
As a corollary, using Lemma \ref{lemma:discrete} and Lemma \ref{lem:spectrum-and-eigenvalues}, we deduce the following: 
\begin{prop}\label{prop:zero_cohomology_nonempty}
For an admissible lift of an Anosov action $\X$, the Ruelle-Taylor resonance spectrum is a discrete subset of $\a_\C^*$. Moreover, $\lambda\in \mc F_{NG, A_0}\cap\sigma_{\T, \mc H_{NG}}(-\mathbf X)$ if and only if there is $u\in \mc{H}_{NG}$ such that 
\[
(\X +\lambda)u=0.
\]
\end{prop}
This completes the proof of Theorem~\ref{theoHn}. In the scalar case (i.e. $E$ is the trivial bundle)  we will show in Corollary~\ref{noresonancesRe>0} below that part 3) of Theorem~\ref{theoHn} can be sharpened using the dynamical parametrix $Q(\lambda)$ in Lemma~\ref{lemmafordiscreteness}  (the same argument also works for admissible lifts under the condition $\|e^{-t\X_{A}}f\|_{\mc{L}(L^\infty)}\leq C$ for all $t\in \R$): 
\begin{cor}\label{noresonancesRe>0}
Let $X$ be an Anosov action. Then one has 
\[
  \mathrm{Res}_X \subset \bigcap_{A\in \mathcal W} \{\lambda\in\a_\C^*~|~\Re(\lambda(A))\leq 0\}.
 \]
\end{cor}
\begin{proof} Let $A\in \W$ and assume that $\lambda\in \a^*_\C$ satisfies ${\rm Re}(\lambda(A))>0$, then we will show that $\lambda$ can not be a Ruelle-Taylor resonance.
 We use the parametrix $Q(\lambda)$ of Lemma \ref{lemmafordiscreteness} 
with $A_1:=A$ and $(A_j)_j \in \W^\kappa$ forming a basis of $\a$ with $A_j$ in an arbitrarily small neighborhood of $A_1$ so that ${\rm Re}(\lambda(A_j))>0$ for all $j$. We get that \eqref{dXQ+QdX0} holds with $F(\lambda)$ having discrete spectrum near $z=0$. Let $\Pi_0(\lambda)$ be the spectral projector of the Fredholm operator $F(\lambda)$ at $z=0$, which can be written 
\begin{equation}\label{projector} 
\Pi_0(\lambda) = \frac{1}{2\pi i}\int_{|z|=\eps}(z{\rm Id}-F(\lambda))^{-1}dz
\end{equation}
for some small enough $\eps>0$. We notice that for $f\in L^\infty(\M)$, we have    
\[
\begin{split}
\|({\rm Id}-F(\lambda))f\|_{L^\infty}\leq & \int_{(\R^+)^\kappa} \|e^{-\sum_{j=1}^\kappa t_j\X_{A_j}(\lambda)}f\|_{L^\infty}\prod_{j=1}^\kappa(-\chi'_j(t_j))dt_1\dots dt_\kappa\\
\leq & \|f\|_{L^\infty}e^{-\sum_{j=1}^\kappa T_j\lambda(A_j)} \int_{(\R^+)^\kappa} \prod_{j=1}^\kappa(-\chi'_j(t_j))dt_1\dots dt_\kappa\\
= & \|f\|_{L^\infty}e^{-\sum_{j=1}^\kappa T_j\lambda(A_j)}.
\end{split}
\]
This shows that by choosing the $T_j>0$ (that had been introduced in Lemma~\ref{lemmafordiscreteness}) large enough, $\|({\rm Id}-F(\lambda))\|_{\mc{L}(L^\infty)}<1/2$. In particular $F(\lambda)$ is invertible on $L^\infty$ and therefore $\Pi_0(\lambda)=0$ since the expression \eqref{projector} holds
also as a map $C^\infty(\M)\to C^{-\infty}(\M)$. This ends the proof. 
\end{proof}
Let us end the section with a statement about joint Jordan blocks for an admissible lifts $\mathbf{X}$: Therefore given $\alpha \in \N^\kappa$ we define $\mathbf X^\alpha(\lambda):= \prod_{j=1}^\kappa (\mathbf X_{A_j}+\lambda_j)^{\alpha_j}$.
\begin{prop}\label{prop:jordan_blocks}
 For any Ruelle-Taylor resonance $\lambda\in \Res_{\mathbf X}$ there is $J\in \N^*$ which is the minimal integer such that, whenever for some $u\in C^{-\infty}_{E_u^*}(\mc M)$ and $k\in \N^*$ one has $\mathbf X^\beta(\lambda) u =0$ for all $|\beta|=k$ then $\mathbf X^\alpha(\lambda) u =0$ for all $|\alpha|=J$. Moreover the space of \emph{generalized joint resonant states} is the finite dimensional space given by 
\begin{equation}\label{generalizedjointresonant}
\{u\in C^{-\infty}_{E_u^*}(\mc M)\, |\, \mathbf X^\alpha(\lambda)u=0 \tu{ for all }|\alpha|=J\} \subset \ran \Pi_0(\lambda)
\end{equation}
where $\Pi_0(\lambda)$ is the spectral projector of $F(\lambda)$ at $z=0$, defined in \eqref{projector}.
\end{prop}
\begin{proof}
 Let $\mathcal H_{NG}$ be an anisotropic Sobolev space such that $\lambda\in \mathcal F_{NG,A_0}$. We construct the parametrix from \eqref{dXQ+QdX0} 
 \[
Q(i\lambda) d_{\mathbf X+\lambda} + d_{\mathbf X+\lambda} Q(\lambda) = {\rm Id} - R(\lambda)\otimes {\rm Id},
\]
for some appropriate choice of basis $A_1,\ldots,A_j\in\W\subset \a$, and writing $\psi_j:=-\chi_j'\in C_c^\infty((0,\infty))$, $\mathbf X_j := \mathbf X_{A_j}$ and $\lambda_j:=\lambda(A_j)$ we write
\begin{equation}\label{eq:Eq-Rlambda}
R(\lambda) = \prod_{j=1}^\kappa \int e^{-t_j (\mathbf X_j+\lambda_j)} \psi_j(t_j) dt_j.
\end{equation}
We denote by $\Pi_0(\lambda):\mc H_{NG}\to\mc H_{NG}$ the spectral projector on the generalized eigenspace of $R(\lambda)$ for the eigenvalue 1. Note that it commutes with ${\bf X}_j$ for all $j$, since $R(\lambda)$ does.

We now show by induction that for any $u\in C^{-\infty}_{E_u^*}(\mc M)$ with $\mathbf X^\alpha(\lambda) u = 0$ for all $|\alpha|=k$ we have $u\in \ran(\Pi_0(\lambda))\subset \mc H_{NG}$.  
The start of the induction is simply deduced from \eqref{eq:Eq-Rlambda}: for $u\in \mathcal H_{NG}$ with $(X_j+\lambda_j)u=0$ we deduce 
\[
 R(\lambda) u = \prod_{j=1}^\kappa\int \psi_j(t_j) e^{-t_j(X_j+\lambda_j)} u\, dt_j = u, 
\]
thus $u\in \ran(\Pi_0(\lambda))$.  Next we show that if the property is satisfied at step $k$ then it is at step $k+1$: if for $u \in C^{-\infty}_{E_u^*}(\mc M)$, $\mathbf X^\alpha(\lambda) u=0$ for all $|\alpha|=k+1$ then for any $|\beta|=k$ we have $\mathbf X_{A_j}(\lambda)(\mathbf X^\beta(\lambda) u) =0$ for $j=1,\ldots,\kappa$.
 Thus, by the the same argument as above we know $\mathbf X^\beta(\lambda) u \in \ran(\Pi_0(\lambda))$. As $[\Pi_0(\lambda),\mathbf X^\beta(\lambda)]=0$ we conclude $\mathbf X^\beta(\lambda)(\Pi_0(\lambda) u -u)= 0$. 
 Consequently, by induction hypothesis, $\Pi_0(\lambda) u -u \in \ran(\Pi_0(\lambda))$ and the claim follows. 
The statement of the proposition follows because $\ran(\Pi_0(\lambda))$ is a finite dimensional $\mathbf X_{A_i}(\lambda)$ invariant subspace. 
\end{proof}

We also notice that the non-triviality of the space \eqref{generalizedjointresonant} (with $J$ minimal) implies that $\lambda$ is a Ruelle-Taylor resonance, since for $u$ in this space, there is an $\alpha$ with $|\alpha|=J-1$ such that $v:=\X^\alpha(\lambda)u\not=0$ satisfies $v\in {\rm Res}_{\X,\Lambda^0}(\lambda)$.
We also note that equality in \eqref{generalizedjointresonant} does in general not hold. One rather has:
\begin{prop}\label{prop:ranPi}
 If $\Pi_0(\lambda)$ is the spectral projector of $R(\lambda)$ from $\eqref{eq:Eq-Rlambda}$ then 
 \[
  \ran \Pi_0(\lambda) = \bigoplus_{\eta\in \Res_\mathbf X,~~ \prod_j \hat\psi_j(-i(\lambda_j-\eta_j))=1} \{u\in C^{-\infty}_{E_u^*}(\mc M)\, |\, \mathbf X^\alpha(\eta)u=0 \tu{ for all }|\alpha|=J\}
 \]
 where $J\in \N$ is the integer from Proposition \ref{prop:jordan_blocks}.
\end{prop}

\begin{proof}
 First note that $\ran \Pi_0(\lambda)$ is finite dimensional and $\mathbf X_j$ invariant, thus we can decompose the space into joint generalized eigenstates. If $\eta$ is such a joint eigenvalue then Proposition~\ref{prop:zero_cohomology_nonempty} implies that $\eta$ is also a Ruelle-Taylor resonance. Now let $u\in \ran \Pi_0(\lambda)$ be a joint eigenstate of ${\bf X}$ with eigenvalue $\eta$ then by \eqref{eq:Eq-Rlambda}
 \[
  R(\lambda) u = \prod_{j=1}^\kappa \int e^{-t_j (\lambda_j-\eta_j)} \psi_j(t_j) dt_j  = \prod_{j=1}^\kappa \hat\psi_j(-i(\lambda_j-\eta_j))u.
 \]
Thus $u$ is an eigenstate of $R(\lambda)$ but as $u$ is also required to be in the generalized eigenspace of eigenvalue $1$, we deduce that
$\prod_j\hat\psi_j(-i(\lambda_j-\eta_j))=1$. This shows, that the left hand side is contained in the right hand side. 

For the converse inclusion we note that any joint resonant state $(X_j+\eta_j)u=0$ whose joint resonance fulfills $\prod_j \hat\psi_j(-i(\lambda_j-\eta_j))=1$ is an eigenstate of $R(\lambda)$ with eigenvalue $1$ and thus contained in $\ran(\Pi_0(\lambda))$. For the generalized eigenstates of higher order we argue as above in Proposition~\ref{generalizedjointresonant} by induction. 
\end{proof}

\section{The leading resonance spectrum}\label{sec:resonance-at-zero}
In this section we study the leading resonance spectrum, i.e. those resonances with vanishing real part and show that they give rise to particular measures and are related to mixing properties of the Anosov action. In this section the bundle $E$ will be trivial.

\subsection{Imaginary Ruelle-Taylor resonances in the non-volume preserving case.}\label{sec:ia_non_volume}

In this section, we investigate the purely imaginary Ruelle-Taylor 
resonances and in particular the resonance at $0$ for the action on functions. We assume that the Anosov action $X$ does not necessarily preserve a smooth invariant measure. We choose a basis $A_1,\dots,A_\kappa$ of $\a$, with dual basis $(e_j)_j$ in $\a^*$, and set $X_j:=X_{A_j}$, and we use $dv_g$ the smooth Riemannian probability measure on $\M$. 
Let us choose $i\lambda \in i \a^\ast$ purely imaginary and fix non-negative functions $\chi_j\in C_c^\infty(\R^+)$, equal to $1$ on a large interval $[0,T_j]$, with $\chi_j'\leq 0$ and use the parametrix $Q(i\lambda)$ in the divergence form from Lemma \ref{lemmafordiscreteness} so that
by \eqref{dXQ+QdX0} 
\[
Q(i\lambda) d_{X+i\lambda} + d_{X+i\lambda} Q(i\lambda) = {\rm Id} - R(i\lambda)\otimes {\rm Id},
\]
and writing $\psi_j:=-\chi_j'\in C_c^\infty((0,\infty))$ and $\lambda_j:=\lambda(A_j)$
\begin{equation}\label{eq:Eq-R}
R(i\lambda) = \prod_{j=1}^\kappa \int e^{-t_j (X_j+i \lambda_j)} \psi_j(t_j) dt_j.
\end{equation}
We proved that $R(i\lambda)$ has essential spectral radius $<1$ in the anisotropic space $\mc{H}_{NG}$, and the resolvent $(R(i\lambda)-z)^{-1}$ is meromorphic outside $|z|<1-\eps$ for some $\eps$, and the poles in $|z|>1-\eps$ are the eigenvalues of $R(i\lambda)$.
Moreover, for $f\in L^\infty$, one has 
\begin{equation}\label{Linfty_estimate} 
\|R(i\lambda)f\|_{L^\infty}\leq \|f\|_{L^\infty}\prod_{j=1}^\kappa \int_{\R} \psi_j(t_j) dt_j= \|f\|_{L^\infty}.
\end{equation}
Since $R(i\lambda)$ is bounded, for $|z|$ large enough one has on $\mc{H}_{NG}$
\begin{equation}\label{resolventofR}
(z-R(i\lambda))^{-1} = z^{-1}\sum_{k\geq 0} z^{-k}R(i\lambda)^k
\end{equation}
but the $L^\infty$ estimate \eqref{Linfty_estimate} shows that this series 
converges in $\mc{L}(L^\infty)$ and is analytic for $|z|>1$. 
Therefore, using the density of $C^\infty(\M)$ in $\mc{H}_{NG}$, we deduce that $R(i\lambda)$ has no eigenvalues in $|z|>1$. We will use the notation $\cjg u,v\cjd$ for the distributional pairing associated to the Riemannian measure $dv_g$ fixed on $\M$ that also extends to a complex bilinear pairing $\mc{H}_{NG}\times \mc{H}_{-NG}\to \C$; in particular if $u,v\in L^2(\M)$, this is simply $\int_{\M} uv \, dv_g$. Accordingly, we also write $\cjg u,v\cjd_{L^2}$ for the 
pairing  $\int_{\M}u\bar{v}\, dv_g$ and its sesquilinear extension to the pairing $\mc{H}_{NG}\times \mc{H}_{-NG}\to \C$.

The next three lemmas (Lemma~\ref{lem:nojordan}) characterize the spectral projector of $R(i\lambda)$ onto the possible eigenvalue 1. Keep in mind that by Lemma~\ref{lem:critereFredholm} this spectral projector is closely related to the Ruelle Taylor resonant states. Finally in Proposition~\ref{prop:SRBmeasures} we will use the knowledge about this spectral projector to characterize the leading resonance spectrum and to define physical measures.
\begin{lemma}\label{lem:nojordan}
Let $\lambda\in i\a^*$. If $\tau$ is an eigenvalue of $R(i\lambda)$ with modulus $1$, it has no associated Jordan block, i.e. $(z-R(i\lambda))^{-1}$ has at most a pole of order $1$ at $z=\tau$.
\end{lemma} 
\begin{proof}
We take $u\in \mc{H}_{NG}$ such that $(R(i\lambda)-z)^{-1} u$ has a pole of order $> 1$ at $z=\tau$. 
By density of $C^\infty$ in $\mc{H}_{NG}$, we can always assume that $u$ is smooth. 
Denoting by $\psi^{(k)}=\psi*\dots*\psi$ ($k$-th convolution power), we can write
\[
R(i\lambda)^k  = \prod_{j=1}^\kappa \int_{\R}e^{-t_j X_j(i\lambda) } \psi_j^{(k)} (t_j) dt_j .
\]
Note that $R(0)1=1$.
We take $v$ another smooth function, then 
\[\begin{split}
|\cjg R(i\lambda)^k u, v\cjd|  =& \Big| \int_{\R^\kappa} \prod_{j=1}^\kappa \psi_j^{(k)}(t_j)e^{-i\sum t_j\lambda_j} 
\Big(\int_{\M}  v  e^{-\sum t_j X_j} u \, dv_g\Big)\, dt_1 \dots dt_{\kappa} \Big| \\
\leq &  |v|_{L^\infty} |u|_{L^\infty} R(0)^k1=|v|_{L^\infty} |u|_{L^\infty}.
\end{split}
\]
We deduce that for $|z|>1$
\[
|\cjg z (z-R(i\lambda))^{-1} u, v\cjd| \leq \sum_{k=0}^\infty |z|^{-k} |v|_{L^\infty} |u|_{L^\infty} = |u|_{L^\infty}|v|_{L^\infty}(1-|z|^{-1})^{-1}.
\]
 This is in contradiction with the assumption that $\tau$ is a pole is of order $>1$.
\end{proof}
Then we can prove the following 
\begin{lemma}\label{lem:tau=1}
For $\lambda=\sum_{j=1}^\kappa \lambda_je_j\in \a^*$, $R(i\lambda)$ has an eigenvalue of modulus $1$ on $\mc{H}_{NG}$ if and only if $i\lambda$ is a Ruelle-Taylor resonance. In that case, 
the only eigenvalue of modulus $1$ of $R(i\lambda)$ in $\mc{H}_{NG}$ is $\tau=1$ and the 
eigenfunctions of $R(i\lambda)$ at $\tau=1$ are the joint Ruelle resonant states of $X$ at $\lambda$. Moreover, if $\Pi(i\lambda)$ is the spectral projector  of $R(i\lambda)$ at $\tau=1$, one has, as bounded operators in $\mc{H}_{NG}$,
\begin{equation}\label{eq:projector_as_lim_Rk}
\lim_{k\to \infty} R(i\lambda)^k =\Pi(i\lambda).
\end{equation}
\end{lemma}
\begin{proof}First, if $i\lambda$ is a Ruelle-Taylor resonances, Proposition~\ref{prop:jordan_blocks} implies that $R(i\lambda)$ has 1 as an eigenvalue and the resonant states are included in the range of the spectral projector of $R(i\lambda)$ at 1. 

Conversely, let $\Pi(i\lambda)$ be the spectral projector of $R(i\lambda)$ at $\tau\in \mathbb{S}^1:$ it commutes with the $X_j$, so we can use Lemma \ref{lem:taylor_finite_dim} to decompose $\ran \Pi(i\lambda)$ in terms of joint eigenspaces for $X_j$. 
Let $u$ be a joint-eigenfunction of $X_j$ in $\ran \Pi(i\lambda)$, with $X_ju=\zeta_ju$. By Lemma \ref{lem:nojordan}, $R(i\lambda)$ has no Jordan block at $\tau$, thus $u\in \mc{H}_{NG}$ is a non-zero eigenfunction of $R(i\lambda)$ with eigenvalue $\tau\in \mathbb{S}^1$. Then
\[
\tau u=R(i\lambda)u  = u \int_{\R^\kappa} \prod_{j=1}^\kappa e^{-t_j( \zeta_j + i \lambda_j)}\psi_j(t_j) dt_j = u \prod_{j=1}^\kappa \hat{\psi}_j(\lambda_j - i \zeta_j ).
\]
For $\tau$ to have modulus $1$, we need $\prod_{j=1}^\kappa |\hat{\psi}_j(\lambda_j - i \zeta_j)|=1$. But since $\int_{\R} \psi_j=1$ and the  $\zeta_j$'s have non-negative real part, 
\[\begin{split} 
|\hat{\psi}_j(\lambda_j - i \zeta_j)|\leq \int_{\R}e^{-t{\rm Re}(\zeta_j)}\psi_j(t)dt\leq 1
 \end{split}\]
so ${\rm Re}(\zeta_j)=0$ and $|\hat{\psi}_j(\lambda_j -i \zeta_j)|=1$ for all $j$.
But then there is $\alpha\in\R$ so that
$1=\int_\R \psi_j(t)=\int_{\R}\cos(t(\lambda_j+{\rm Im}(\zeta_j))+\alpha)\psi_j(t)dt$
and thus $\cos(t(\lambda_j+{\rm Im}(\zeta_j))+\alpha)=1$ on $\supp(\psi_j)$ since $\psi_j\geq 0$, but this implies that $\zeta_j = -i \lambda_j$ and $\alpha\in 2\pi\mathbb Z$. Then we get $\tau = 1$. In particular, 
\[R(i\lambda)=\Pi(i\lambda)+K(i\lambda)\]
with $K(i\lambda)\Pi(i\lambda)=\Pi(i\lambda)K(i\lambda)=0$, and $K(i\lambda)$ having spectral radius $r<1$ on $\mc{H}_{NG}$, thus satisfying that for all $\eps>0$, there is $n_0$  large so that for all $n\geq n_0$ 
\[ \|K(i\lambda)^n\|_{\mc{L}(\mc{H}_{NG})}\leq (r+\eps)^n.\]
We can chose $r+\eps<1$, which implies that 
\begin{equation}\label{boundRila}
\forall n\geq n_0, \,\, R(i\lambda)^n=\Pi(i\lambda)+K(i\lambda)^n \to \Pi(i\lambda) \textrm{ in }\mc{L}(\mc{H}_{NG})
\end{equation} 
proving \eqref{eq:projector_as_lim_Rk}.

To conclude the proof, we want to prove that $(X_j+i\lambda_j)\Pi(i\lambda)=0$ for all $j=1,\dots,\kappa$. By the discussion above, $0$ is the only joint eigenvalue 
of $(X_1+i\lambda_1,\dots,X_\kappa+i\lambda_\kappa)$ on $\ran \Pi(i\lambda)$, i.e. there is $J>0$ such that $\prod_{j=1}^\kappa (X_j+i\lambda_j)^{\alpha_j}\Pi(i\lambda)=0$ for all multi-index $\alpha\in\N^\kappa$ with length $|\alpha|=J$. We already know that $R$ has no Jordan block, and we want to deduce that this is also true for the $X_j$'s. By Proposition \ref{prop:jordan_blocks}, we get
\[ \ran \Pi(i\lambda)=\{u\in C^{-\infty}_{E_u^*}(\M)\, |\, \prod_{j=1}^\kappa (X_j+i\lambda_j)^{\alpha_j}u=0,\,\forall \alpha\in \N^\kappa, |\alpha|=J \}.\]
In particular this space does not depend on the choice of $\chi_j$ (and thus $\psi_j$).
The operator $e^{-\sum_jt_j(X_j+i\lambda_j)}: \ran \Pi(i\lambda)\to \ran \Pi(i\lambda)$ is represented by a finite dimension matrix $M(t)$ with $t=(t_1,\dots,t_\kappa$), and $R(i\lambda)|_{\ran \Pi(i\lambda)}={\rm Id}$ (since $R(i\lambda)$ has no Jordan block), thus
\[ {\rm Id}= \int_{\R^\kappa} M(t)\psi_j(t_j)dt_1\dots dt_\kappa \]
for all choices of $\chi_j$ (and $\psi_j=-\chi_j'$). We can thus take, for $T=(T_1,\dots,T_\kappa)$, the family $\psi_j$ converging to the Dirac mass $\delta_{T_j}$ and we obtain $M(T)={\rm Id}$. This shows that $M(t)={\rm Id}$ for all $t\in \R_+^\kappa$ large enough such that Lemma~\ref{lemmafordiscreteness} can be applied and therefore $(X_j+i\lambda_j)\Pi(i\lambda)=0$ for all $j$. This implies that $\ran \Pi(i\lambda)$ is exactly the space of Ruelle resonant states for $X$ at $i\lambda$. 
\end{proof}
From what we have shown in Lemma~\ref{lem:tau=1}, we deduce that we can write the spectral projector as $\Pi(i\lambda)f=\sum_{k=1}^J v_k\cjg f,w_k\cjd_{L^2}$ with $v_k\in \mc{H}_{NG}$ spanning the space of joint Ruelle resonant states of the resonance $i\lambda$ and $w_k\in \mc{H}_{NG}^*\simeq \mc{H}_{-NG}$. Recall that we have shown that the space of joint Ruelle resonant states (i.e. the range of $\Pi(i\lambda)$) is intrinsic, i.e. does not depend on the precise form of the parametrix.
But surely the operator $R(i\lambda)$ depends on the choice of the cutoff functions $\psi_j$ (see \eqref{eq:Eq-R}) and thus also $\Pi(i\lambda)$ might depend on that choice.
In order to see that this is not the case, let us consider $X_j^* = -X_j +\textup{div}_{v_g}(X_j)$ which are the adjoints with respect to the fixed measures $v_g$. Note that by the commutativity of the $X_j$, the operators $X_j^*$ also commute and are  admissible operators (in the sense of Definition~\ref{def:admissible_lift}) for the inverted Anosov action $\tau^-(a):=\tau(-a)$ which is obviously again an Anosov action (with the same positive Weyl chamber after swaping the stable and unstable bundles). Therefore we can apply the results of Section~\ref{sec:Taylor-pseudors} to the admissible operators $X_j^*$, in particular they have discrete joint spectrum on the spaces $\mathcal H_{-NG}$. 
Using $(X_j+i\lambda_j)\Pi(i\lambda)=0$ and the fact that $[X_j,\Pi(i\lambda)]=0$ we deduce that $(X_j^*-i\lambda_j)w_k=0$ and thus all $w_k$, $k=1,\ldots, J$ are joint resonant states of the $X_j^*$. 
We can even see that they span the space of joint resonant states: one can perform the same parametrix construction Lemma~\ref{lemmafordiscreteness} to $X_j^*$
\[
 Q_{X^*}(i\lambda) d_{X^*+i\lambda} + d_{X^*+i\lambda} Q_{X^*}(i\lambda) = {\rm Id} - R_{X^*}(i\lambda)\otimes {\rm Id},
\]
and if we choose the same cutoff functions as in the parametrix for $X_j$ at the beginning of this section we find
\begin{equation}\label{eq:Eq-R'}
R_{X^*}(i\lambda) = \prod_{j=1}^\kappa \int e^{-t_j (X_j^*+i \lambda_j)} \psi_j(t_j) dt_j.
\end{equation}
In particular $R_{X^*}(-i\lambda) = (R_{X}(i\lambda))^*$ as bounded operators on $\mc{H}_{-NG}$,  where the adjoint is defined by 
$\cjg R_{X}(i\lambda)f,f'\cjd_{L^2}=\cjg f,(R_{X}(i\lambda))^*f'\cjd_{L^2}$ for all $f\in \mc{H}_{NG}$, $f'\in \mc{H}_{-NG}$. 
If $\Pi_{X^*}(i\lambda)$ is the spectral projector of $R_{X^*}(i\lambda)$ onto the eigenvalue 1 then we obtain $\Pi_{X^*}(-i\lambda)f =\Pi_X(i\lambda)^*f= \sum_{k=1}^Jw_k\cjg f,v_k\cjd_{L^2}$ with adjoint defined as above.
By Lemma~\ref{lem:critereFredholm} the space of joint resonant states $(X_j^*-i\lambda)w=0$,  $w\in\mathcal H_{-NG}$ is in the range of $\Pi_{X^*}(-i\lambda)$, consequently the $w_j$ span the space of joint resonant states of $X^*$ with joint resonance $-i\lambda$ and the $w_j$ span the space of joint resonant states of the resonance $i\lambda$. Putting everything together, we have:\begin{lemma}\label{lem:form_of_Pi}
 Let $\lambda\in \a^*$ such that $i\lambda$ is a Ruelle-Taylor resonance of $X$. Then $-i\lambda$ is also a Ruelle-Taylor resonance of $X^*$ and spaces of joint resonant states have the same dimension. If $v_1,\ldots,v_J\in C^{-\infty}_{E_u^*}(\mc M)$ and $w_1,\ldots,w_J\in C^{-\infty}_{E_s^*}(\mc M)$ are such that they span the space of joint resonant states of $X$ at $i\lambda$ and $X^*$ at $-i\lambda$ respectively and fulfill $\langle v_j,w_k\rangle_{L^2}=\delta_{jk}$, then we can write $\Pi(i\lambda)=\sum_{k=0}^Jv_k \cjg \cdot,w_k\cjd_{L^2}$. In particular $\Pi(i\lambda)$ depends only on the $X_j$ but not on the choice of $R(i\lambda)$.
\end{lemma}
We can now identify resonant states on the imaginary axis with some particular invariant measures. 
\begin{prop}\label{prop:SRBmeasures}~
\begin{enumerate}
	\item For each $v\in C^\infty(\M,\R^+)$, the map
\[
\mu_v :  C^\infty(\M)\ni u\mapsto \cjg\Pi(0)u,v\cjd 
\]
is a non-negative Radon measure with mass $\mu_v(\M)=\int_{\M}v\, dv_g$, invariant by $X_j$ for all $j=1,\dots,\kappa$ in the sense $\mu_v(X_j u)=0$ for all $u\in C^\infty(\mc{M})$. 
\item\label{it:space_ofSRB} The space 
\[{\rm span}\{\mu_v \,|\, v\in C^\infty(\M,\R^+)\}= \Pi(0)^*(C^\infty(\M))\]
is a finite dimensional subspace of $C^{-\infty}_{E_s^*}(\M)$ and it is precisely the space spanned by all finite measures $\mu$ with $\WF({\mu})\subset E_s^*$ that are invariant under the Anosov action. Here $\Pi(0)^*:\mc{H}_{-NG}\to \mc{H}_{-NG}$ is a bounded projector for all $N\gg 1$.

\item Let $f\in L^1(\mc{W};[0,1])$ with compact support contained in $\bbar{\mc{W}}$ and $\int_{\mc{W}}f >0$. Then for any $u,v\in C^\infty(\M)$  
\begin{equation}\label{Birkhoff_sum}
\mu_v(u)= \lim_{T\to \infty} \frac{1}{T^\kappa \int_{\mc{W}} f} \int_{A\in \mc{W}} f(\frac{A}{T})\cjg e^{-X_A}u,v\cjd dA
\end{equation}
where $dA$ is the Lebesgue-Haar measure on $\a$.

\item Similarly, for $\lambda\in \a^*$, $v\in C^\infty(\mc M)$ the map
\[
\mu^{\lambda}_v : C^\infty(\M)\ni u\mapsto \cjg\Pi(i\lambda)u,v\cjd 
\]
is a complex valued measure. The measures are flow equivariant in the sense that $\mu_v^\lambda(X_ju) = -i\lambda_j\mu_v^\lambda(u)$ and the set $\{\mu_v^{\lambda} \,|\, v\in C^\infty(\M)\}$ is finite dimensional and coincides with the space of finite complex measures $\mu$ with $\WF(\mu)\subset E^*_s$ which are equivariant in the above sense. 
\item \label{it:ac} Let $v_1,v_2\in C^\infty(\mc M, \R^+)$ with $v_1\leq Cv_2$ for some $C>0$ and $i\lambda \in i\a^*$ a Ruelle-Taylor resonance. Then $\mu_{v_1}^\lambda$ is absolutely continuous with bounded density with respect to $\mu_{v_2}=\mu_{v_2}^0$. In particular any $\mu_v^\lambda$ is absolutely continuous with respect to $\mu_1$.
\end{enumerate}
\end{prop}

\begin{proof}
First $R(0)1=1$ is clear and $X$ has a Taylor-Ruelle resonance at $\lambda=0$ by Lemma~\ref{lem:spectrum-and-eigenvalues}. If $u,v\in C^\infty(\M)$ are non-negative, we have 
$a_k:=\cjg R(0)^ku,v\cjd\geq 0$ and
\[
\lim_{k\to \infty} \cjg R(0)^ku,v\cjd = \cjg\Pi(0)u,v\cjd\geq 0.
\]
Note also that for each $k$, and each $u\in C^\infty(\mc{M})$ non-negative   
\[
\forall x\in \M, \quad 0\leq (R(0)^ku)(x)\leq (R(0)^k1)\|u\|_{C^0}\leq \|u\|_{C^0}.
\]
This implies that for each $v\in C^\infty$ with $v\geq 0$, $\mu^k_v: u\mapsto \cjg R(0)^ku,v\cjd$ is a Radon measure with finite mass $\mu^k_v(\M)=\int_{\M} v \, dv_g$ and thus 
$\mu_v$ is as well. The invariance of $\mu_v$ is a direct consequence of Lemma~\ref{lem:form_of_Pi}. The same holds for property (\ref{it:space_ofSRB}). The invariance of the space spanned by these measures with respect to $X_j$ follows from $\Pi(0)X_j=X_j\Pi(0)=0$, obtained  by Lemma \ref{lem:tau=1}.

Let us next show that for an arbitrary Ruelle-Taylor resonance $i\lambda\in i\a^*$ we get complex measures $\mu_v^\lambda$ and in the same turn prove the absolute continuity statement (\ref{it:ac}).
We consider $u\in C^\infty(\mc M)$, $v_1, v_2\in C^\infty(\mc M, \R^+)$ with $v_1\leq v_2$ and get for all $k$
\[
|\cjg R(i\lambda)^ku,v_1\cjd|\leq \cjg R(0)^k|u|, v_1\cjd \leq \cjg R(0)^k|u|, v_2\cjd
\]
thus $|\mu_{v_1}^\lambda(u)|\leq \mu_{v_2}(|u|)$. This proves that $\mu_{v_1}^\lambda$ is a complex measure. A priori, $\mu_{v_1}^\lambda$ is absolutely continuous with respect to $\mu_{v_2}$, so it has a $L^1$ density $f$ with respect to $\mu_{v_2}$. This density is actually bounded by $1$ or equivalently,
\begin{equation}\label{eq:borel-set}
|\mu_{v_1}^\lambda(\mc{A})| \leq \mu_{v_2}(\mc{A})
\end{equation}
 for every Borel set $\mc{A}$. If $\mc{A}$ is a closed set, we can find a sequence of smooth functions $g_n$, valued in $[0,1]$, which converges simply to the characteristic function of $\mc{A}$. By dominated convergence $\mu_{v_1}^\lambda (g_n) \to \mu_{v_1}^\lambda(\mc{A})$, and likewise $\mu_{v_2}(g_n)\to \mu_{v_2}(\mc{A})$. Taking products of sequences of functions, or sequences $1-g_n$, and a diagonal argument, we see that the set of Borel sets $\mc{A}$ for which \eqref{eq:borel-set} holds contains closed sets, and is stable by countable intersection, and complement. It is thus equal to the whole tribe of Borel sets, and the proof of $\|f\|_{L^\infty}\leq 1$ is complete.

Let us finally show \eqref{Birkhoff_sum}. For each $f\in L^1(\a;[0,1])$ with
$\supp(f)\subset \bbar{\mc{W}}$ being compact and $\int f>0$, we want to prove that 
\begin{equation}\label{limittoshow}
C_f(T):= \frac{1}{T^\kappa \int_{\mc{W}} f} \int_{A\in \mc{C}_T}f(\frac{A}{T})\cjg e^{-X_A}u,v\cjd dA \to \mu_v(u) \textrm{ as }T\to +\infty.
\end{equation} 
Assume that \eqref{limittoshow} is satisfied for a dense set in $L^1(\mc{W})$ of compactly supported functions $F:=(f_i)_{i\in I}
\in C_c(\mc{W})$. Then, for all $\delta>0$ small, one can find a sequence $f_{i(n)}\in F$ so that 
$\int_{\mc{W}} |f_{i(n)}/(\int_{\mc{W}} f_{i(n)})-f/\int_{\mc{W}}f|<\delta$. Then 
\[ |C_f(T)-C_{f_{i(n)}}(T)|\leq \int_{\mc{W}}|\cjg e^{-TX_A}u,v\cjd|. \Big|\frac{f(A)}{\int f}-\frac{f_{i(n)}}{\int f_{i(n)}}\Big| dA \leq \delta \|u\|_{L^\infty} \|v\|_{L^\infty}.\]
which implies \eqref{limittoshow} for $f$ by our assumption on $f_{i}$ and since $\delta$ is arbitrarily small. 
We shall then show \eqref{limittoshow} for functions of the form $\omega(t_1)q(\bar{t}/t_1)$ 
if $(t_1,\bar{t})\in \R_+\times \R^{\kappa-1}$ are coordinates associated to bases of vectors in small cones $\mc{C}$ with closure 
contained in $\mc{W}\cup \{0\}$, and $\omega\in C_c^\infty(0,1)$, $q\in C_c^\infty(\R^{\kappa-1})$ such that 
$\supp(\omega(t_1)q(\bar{t}/t_1))\subset \mc{C}$.

We fix a small open cone $\mc{C}\subset \W$ with arbitrarily 
small conic section with closure contained in $\mc{W}$ and  choose a basis $(A_j)_{j=1}^\kappa$ of $\a$ so that $A_j\in \mc{C}$. Up to rescaling $A_j$ by some fixed large $T>0$, we can assume that Lemma \ref{lemmafordiscreteness} applies with $T_j=1/2$ in the construction of ${\bf Q}(\la)$ and $R(\la)$.
We then identify $\a\simeq \R^\kappa$ by identifying the canonical basis $(e_j)_j$ of $\R^\kappa$ with $(A_j)_j$, and we define a scalar product on $\a$ by deciding that $A_j$ are orthonormal. 
We let  $\Sigma=\mc{C}\cap \{A_1+\sum_{j=2}^\kappa t_jA_j |\, t_j\in \R\}$ be a hyperplane section of the cone $\mc{C}$.
Choose $\psi\in C_c^\infty((-1/2,1/2))$ non-negative  even with $\int_\R\psi=1$, and for each $\sigma\in \R^\kappa$, 
define $\psi_\sigma(t):=\prod_{j=1}^\kappa\psi(t_j-\sigma_j)$. The operators $Q(0),R(0)$ constructed in Lemma \ref{lemmafordiscreteness} can be defined, for $\sigma$ close to $e_1$, with the cutoff function $\chi_j$ so that $-\chi'_j(t_j)=\psi(t_j-\sigma_j)$, we then denote $Q_\sigma,R_\sigma$ the corresponding operators, which in turn are locally uniform in $\sigma$. 
Then $\mu_v$ is given by $\mu_v(u)=\lim_{k\to \infty}\cjg R_\sigma(0)^ku,v\cjd$ locally uniformly in $\sigma$. This means that  viewing $\Sigma$ as an open subset of $e_1+\R^{\kappa-1}$ containing $e_1$, taking any $q\in C_c^\infty(\Sigma)$ with $\int_{\R^{\kappa-1}}q(\bar{t})d\bar{t}=1$  and any $\omega\in C_c^\infty((0,1))$ with $\int_0^1 \omega=1$, we have for $\sigma(\bar{t}):=(1,\bar{t})\in \Sigma$
\begin{equation}\label{formulesurmu_v(u)} 
\mu_v(u)=\lim_{N\to \infty}\int_{\R^{\kappa-1}}\frac{1}{N}\sum_{k=1}^{N}\omega\big(\frac{k}{N}\big)\cjg R_{\sigma(\bar{t})}^ku,v\cjd  q(\bar{t})d\bar{t}.\end{equation}
Indeed, one can write, if $\supp(\omega)\subset (\eps,(1-\eps))$ for some $\eps>0$ 
\[\begin{split} 
& \int_{\R^{\kappa-1}}\frac{1}{N}\sum_{k=1}^{N}\omega\big(\frac{k}{N}\big)\cjg R_{\sigma(\bar{t})}^ku,v\cjd  q(\bar{t})d\bar{t}-\mu_v(u)\\  
& = \int_{\R^{\kappa-1}}q(\bar{t})\Big(\frac{1}{N}\sum_{k\in (\eps N,(1-\eps)N)}\omega\big(\frac{k}{N}\big)(\cjg R_{\sigma(\bar{t})}^ku,v\cjd-\mu_v(u)) -\mu_v(u)\Big(1- \frac{1}{N}\sum_{k=1}^{N}\omega\big(\frac{k}{N}\big)\Big)\Big)d\bar{t}
\end{split}\]
and we use $\frac{1}{N}\sum_{k=1}^{N}\omega\big(\frac{k}{N}\big)\to \int_0^1 \omega =1$ as $N\to \infty$, and $\cjg R_\sigma^ku,v\cjd \to \mu_v(u)$ uniformly in $\sigma$ as $k\to \infty$, so that we get the result by dominated convergence.
 
Recall that $R_\sigma^k$ is given by the expression
\[
R_\sigma^ku  =  \int_{\R^\kappa}(e^{-\sum_{j=1}^\kappa t_j X_j}u) \psi_{\sigma}^{(k)}(t)dt.
\]
where $\psi_\sigma^{(k)}$ is the $k$-th convolution of $\psi_{\sigma}$. To prove \eqref{limittoshow} for a dense subset of 
$L_{\rm comp}^1(\mc{W})$ it suffices to combine \eqref{formulesurmu_v(u)} and the following Lemma (since the space of 
finite sums of functions of the form $\tilde{\omega}(t_1)q(\bar{t}/t_1)$ is dense in $L^1_{\rm comp}(\mc{W})$):
\begin{lemma}\label{approximationRtoBirkhoff} 
With $\omega,q$ as above, and let $\widetilde{\omega}(r):=r^{1-\kappa}\omega(r)$ the following limit holds as $N\to \infty$
\[ \int_{\R^{\kappa-1}}\frac{1}{N}\sum_{k=1}^{N}\omega\big(\frac{k}{N}\big)\cjg R_{\sigma(\bar{t})}^ku,v\cjd  q(\bar{t})d\bar{t}-\frac{1}{N^\kappa}\int_{0}^N
\int_{\R^{\kappa-1}}\cjg e^{-\sum_{j}t_jX_j}u,v\cjd  \widetilde{\omega}\big(\frac{t_1}{N}\big)q\Big(
\frac{\bar{t}}{t_1}\Big)dt_1d\bar{t}\to 0.
\]
\end{lemma}
\end{proof} 
\begin{proof}[Proof of Lemma~\ref{approximationRtoBirkhoff}]
Since for $u\in C^\infty(\M)$, $\|e^{-\sum_{j=1}^\kappa t_j X_j}u\|_{L^\infty}\leq \|u\|_{L^\infty}$, it suffices to show that 
\[ 
\R_+\times \R^{\kappa-1}\ni t=(t_1,\bar{t}\, )\mapsto \int_{\R^{\kappa-1}}\frac{1}{N}\sum_{k=1}^{N}\omega(\tfrac{k}{N}) \psi_{\sigma(\theta)}^{(k)}(t)q(\theta)d\theta-\frac{1}{N^{\kappa}}\widetilde{\omega}(\tfrac{t_1}{N})q(\tfrac{\bar{t}}{t_1})
\]
converges for $N\to\infty$ towards $0$ in $L^1(\R^\kappa)$. Let $\eps>0$ small so that $\supp(\omega)\subset (\eps,1-\eps)$. 
Scaling $t\to tN$, the above convergence statement is equivalent to show that 
\[
f_N(t):=N^{\kappa-1}\sum_{k\in \mathbb \Z\cap (\epsilon N, (1-\epsilon)N)}\omega\big(\frac{k}{N}\big) \int_{\R^{\kappa-1}}\psi_{\sigma(\theta)}^{(k)}(tN)q(\theta)d\theta
\] 
is such that $\lim_{N\to \infty}\|f_N-h\|_{L^1(\R^\kappa)}=0$ if $h(t):=\frac{1}{T}\widetilde{\omega}(\frac{t_1}{T})q\big(\frac{T\bar{t}}{t_1}\big)$. First, 
notice that $\supp(\psi_{\sigma(\theta)}^{(k)}(N\cdot))\subset B(0,2)$ for each $k\leq N$, and $\int \psi^{(k)}_{\sigma(\theta)}(t)dt=1$. 
It then suffices to prove that  $f_{N}$
converges in $L^2(\R^\kappa)$ to $h$. We proceed using the Fourier transform, writing $\xi=(\xi_1,\bar{\xi})$
\[
\hat{f}_{N}(\xi)=\frac{1}{N}\sum_{k=\eps N}^{(1-\eps)N}\omega\big(\frac{k}{N}\big)\Big(\hat{\psi}_0\big(\frac{\xi}{N}\big)\Big)^k e^{-i\frac{k}{N}\xi_1} \hat{q}\big(\frac{k}{N}\bar{\xi}\big).
\]
First, for $\xi$ fixed, since $\hat{\psi}_0(\xi)=1+\mc{O}(|\xi|^2)$ for small $\xi$ (since $\int \psi=1$ and $\int t\psi(t)dt=0$ by assumption on $\psi$), one has the following pointwise convergence (using Riemann sums)
\begin{equation}\label{eq:fnhxi}
\lim_{N\to \infty}\hat{f}_{N}(\xi)=\int_{\R}\omega(t_1)\hat{q}(t_1\bar{\xi})e^{-it_1\xi_1}dt_1=\hat{h}(\xi).
\end{equation}
To prove $L^2$ convergence, we use that, since $\psi\geq 0$ is smooth and satisfies $\int \psi=1$ and $\int t\psi(t)dt=0$, there is $c_0>0$ such that
\begin{equation}\label{boundpsi_0}
|\hat{\psi}_0(\xi)|\leq (1+c_0|\xi|^2)^{-1}.
\end{equation} 
Indeed, it suffices to prove $|\hat{\psi}(s)|\leq (1+c_0s^2)^{-1}$ for some $c_0>0$: but  there is $\eps>0$ so that this is true for some $c_0=c_0(\eps)$ when $|s|<\eps$ small (by Taylor expansion at $s=0$ and since  $|\hat{\psi}(s)|<1$ for $s\not=0$) and for $|s|>1/\eps$ large by integration by parts ($\psi\in C_c^\infty$), while for $|s|\in [\eps,1/\eps]$ there is $c_0'(\eps)>0$ so that $|\hat{\psi}(s)|\leq (1+c'_0(\eps)s^2)^{-1}$ by the fact that $|\hat{\psi}|\leq 1-\delta_\eps$ on  $[\eps,1/\eps]$ for some $\delta_\eps>0$.
Thus for $\delta>0$ arbitrarily small, there is $C>0$ depending only on $\|\hat{q}\|_{L^\infty}, \|\omega\|_{L^\infty}$ such that
\[
\int_{|\xi|\geq  N^{1/2+\delta}}|\hat{f}_{N}(\xi)|^2d\xi \leq CN^\kappa \int_{|\xi|\geq N^{-1/2+\delta}}|\hat{\psi}_0(\xi)|^{2\eps N}d\xi \leq CN^\kappa e^{-c_0\eps N^{2\delta}}\int|\hat{\psi}_0(\xi)|d\xi\to 0 
\]
where the second inequality holds for large $N$ and the limit is for $N\to\infty$.
Next, we will show that for $\delta>0$ small and $\ell\in\N$, there is $N_\ell,C_\ell$ such
that for all $N\geq N_\ell$
\begin{equation}\label{hatFnbound} 
\forall \xi, |\xi|\in [1,N^{1/2+\delta}],\quad  |\hat{f}_N(\xi)|\leq C_\ell(|\xi|^{-\ell}+ N^{-\ell(\frac{1}{2}-\delta)}).
\end{equation}
This will prove the convergence of $\hat{f}_N$ to $\hat{h}$ in $L^2$, since for all $n>0$, there is $T_n$ and $N_n>0$  such that for $N\geq N_n$ 
\[ \int_{|\xi|\geq T_n}|\hat{f}_N(\xi)-\hat{h}(\xi)|^2\, d\xi\leq 1/n\]
and, using dominated convergence and \eqref{eq:fnhxi}, 
\[ 
\lim_{N\to\infty}\int_{\R^\kappa}|\hat{f}_N(\xi)-\hat{h}(\xi)|^2\textbf{1}_{[0,T_n]}(|\xi|)d\xi=0.
\]
We next show \eqref{hatFnbound}. We will use a discrete integration by parts to gain decay in $\hat{f}_N(\xi)$ in the $\xi_1$ variable. For $\rho\in C_c^\infty((0,1))$, we can define some sequences 
$a_k^m$, $b_k^m$ for $k\in\Z$ and $m\in\N$ by induction. First, $b_k^0:=e^{-i\frac{k}{N}\xi_1}$, $a_k^0=\rho(k/N)$ for $k\in \Z$. Next, for $m\geq 1$, $k\in \Z$,
 \[
b_k^m:= b_k^{m-1}\frac{e^{-i\frac{\xi_1 }{N}}}{1-e^{-i\frac{\xi_1}{N}}}=e^{-i\frac{k}{N}\xi_1}\Big(\frac{e^{-i\frac{\xi_1}{N}}}{1-e^{-i\frac{\xi_1}{N}}}\Big)^{m}
, \quad a^m_k:= a^{m-1}_k-a^{m-1}_{k+1}.
\]
Note also that $a_k^m=0$ for $k<\eps N-m$ and $k>(1-\eps)N$ and remark that $b_k^m=b^{m+1}_{k-1}-b_k^{m+1}$. Thus, 
\[
\sum_{k\in \Z} a^m_k b_k^m = \sum_{k\in\Z} a^m_k (b^{m+1}_{k-1} -b^{m+1}_{k}) = -\sum_{k\in\Z} b^{m+1}_k (a^m_k - a^m_{k+1}) = -\sum_{k\in\Z} b^{m+1}_k a^{m+1}_k.
\]
Since $\rho$ is smooth, a Taylor expansion gives for each $m$ a constant $C_m>0$ such that for $N\geq 1$, $|a_k^m|\leq C_m\|\rho\|_{C^m}N^{-m}$.
Up to increasing the value of $C_m$, we also have that $|b_k^m|\leq C_m N^m/|\xi_1|^m$ for $|\xi_1|/N\leq \pi/2$. We deduce that for each 
$m$, there exists $C_m>0$ such that for all $N$ large enough, 
\begin{equation}\label{bounddiscreteIPP}
\Big| \sum_{k=1}^N\rho(k/N)e^{-i\xi_1\frac{k}{N}}\Big|\leq N \min(C_m\|\rho\|_{C^m}|\xi_1|^{-m}, \|\rho\|_{L^\infty}).
\end{equation}  
Now, take $\rho(x):=\hat{q}(x\bar{\xi})\omega(x)e^{x N\log(\hat{\psi}_0(\xi/N))}$, which exists since $|\hat{\psi}_0(\xi/N)-1|\leq 1/2$ for $N$ large enough (since $|\xi|/N$ is assumed small).
Using that $q\in C_c^\infty(\R^{\kappa-1})$, for all $\ell\in\N$, there is $C_{\ell,m}>0$ such that for all $x\in \supp(\omega)$
\begin{equation}\label{boundplmrho}
|\partial_x^{m}(\hat{q}(x\bar{\xi})\omega(x))|\leq C_{\ell,m}(1+|\bar{\xi}|^2)^{-\ell+m}.
\end{equation}
Since $|N\log(\hat{\psi}_0(\xi/N))|\leq C|\xi|^2/N$ for some uniform $C>0$  if $|\xi|/N$ is small (by using \eqref{boundpsi_0} 
and $\hat{\psi}_0(\xi)=1+\mc{O}(|\xi|^2)$), there is $C,C'$ so that for all $n$ and $|\xi|/N$ small enough
\[|\pl_x^n e^{xN\log(\hat{\psi}_0(\xi/N))}|\leq (C\frac{|\xi|^2}{N})^n e^{Cx\frac{|\xi|^2}{N}}\leq C'(C\frac{|\xi|^2}{N})^n \]
and we finally obtain, by combining this bound with \eqref{boundplmrho} (choosing $\ell=2m-n+j$ for the $m-n$ derivatives), the following bound : for all $m,j$ there is $C_{m,j},C'_{m,j}>0$ so that for $N$ large enough 
\begin{equation}\label{rhoCm}
\|\rho\|_{C^m}\leq C_{m,j}\sum_{n=0}^m(1+|\bar{\xi}|^2)^{-j} (C\frac{|\xi|^2}{N})^{n}\leq C'_{m,j}(1+|\bar{\xi}|^2)^{-j}(1+\frac{|\xi|^{2m}}{N^m}).
\end{equation}  
Combining this bound (by taking $j=m$) with the bound \eqref{bounddiscreteIPP}, this implies that for all $m$, there is $C_m,C_m'>0$ so that for all $N$ large enough and $|\xi|\in [1,N^{1/2+\delta}]$ and $|\xi_1|\geq 1/2$
\[ |\hat{f}_N(\xi)|\leq C_{m}\frac{1+|\xi|^{2m}N^{-m}}{
(1+|\bar{\xi}|^2)^{m}(1+|\xi_1|)^m}\leq C'_{m}\Big(\frac{1}{|\xi|^m}+\frac{|\xi|^m}{N^m}\Big)\leq C'_m\Big(\frac{1}{|\xi|^m}+\frac{1}{N^{m(\frac{1}{2}-\delta)}}\Big) \]
which shows \eqref{hatFnbound} in the case $|\xi_1|\geq 1/2$. In the case $|\xi|\in [1,N^{1/2+\delta}]$ and $|\xi_1|\leq 1/2$, one has $|\bar{\xi}|\geq |\xi|-1/2$ and thus the $\|\rho\|_{L^\infty}$ bound in \eqref{bounddiscreteIPP} together with the bound \eqref{rhoCm} for $m=0$ and $j$ large give
\[ |\hat{f}_N(\xi)|\leq C_{j}(1+|\bar{\xi}|^2)^{-j} \leq C'_{j} (1+|\xi|^2)^{-j}\]
which again shows \eqref{hatFnbound} in the case $|\xi_1|\leq 1/2$.
\end{proof}
As noted in the introduction, we will call \emph{physical measures} the measures $\mu_v$, and $\mu_1$ will be called the \emph{full physical measure}. 
 
\subsection{Imaginary Ruelle-Taylor resonances for volume preserving actions}\label{sec:Ruelle_at_zero}
In this section, we are going to study the dimensions of the Ruelle-Taylor resonance at $\lambda=0$ in the case where there is a smooth measure preserved by the action.
First, we want to prove 
\begin{prop}\label{smoothnessat0}
Assume that there is a smooth invariant measure $\mu$ for the action, i.e. $\mc{L}_{X_{A}}\mu=0$ for each $A\in\W$. Then, for each $\lambda\in i\a^*$ imaginary, there is an injective map
\begin{equation}\label{isomreson-smooth}
 \ker_{C^{-\infty}_{E_u^*}\Lambda^j} d_{X+\lambda}/\ran_{C^{-\infty}_{E_u^*}\Lambda^{j}} d_{X+\lambda}\to 
\ker_{C^{\infty}\Lambda^j} d_{X+\lambda}/\ran_{C^{\infty}\Lambda^{j}} d_{X+\lambda}.
\end{equation}
\end{prop}
\begin{proof}
We shall use an argument inspired by \cite[Section 6]{FRS08} in rank $1$. As in  \cite[Section 6]{FRS08}, it will be be technically convenient to use a semiclassical Weyl quantization ${\rm Op}_h$ and semiclassical wavefront set, the semiclassical parameter  being denoted $h>0$. The interested reader can consult the book \cite{Zwo12} 
for the details on semiclassical calculus, and \cite[Appendix E]{DZ19} that sumarizes all the necessary results used here. For a shorter summary 
see also the appendix of the article \cite{DZ16a}. We will use the classes of semiclassical operator $\Psi_h^{k}(\mc{M})$ (see \cite[Section 14.2]{Zwo12}), the semiclassical wavefront set ${\rm WF}_h(u)$ (resp. ${\rm WF}_h(A)$) of a distribution $u$ (resp. of an operator $A\in \Psi_h^k(\mc{M})$) depending on a small parameter $h>0$ (see \cite[Section 8.4.2]{Zwo12} or \cite[Sections E.2.1, E.2.3]{DZ19}). The semiclassical wavefront set is a subset of the fiber radial compactification $\bbar{T}^*\mc{M}$ of the cotangent bundle $T^*M$, see \cite[Section E.1.3]{DZ19}.

Fix a basis $A_1,\dots,A_\kappa \in\W$ close to $A_1$ and write $\lambda_j:=\lambda(A_j)$ and $X_{A_j}(\lambda):=X_{A_j}+\lambda_j$. Let $T_j>0$ for $j=1,\dots,\kappa$, let $\varepsilon>0$ be small and consider $\chi_{j}\in C_c^\infty([0,\infty[ ; [0,1])$ non-increasing with $\chi_j=1$ in $[0,T_j]$ and $\supp \chi_j\in [0,T_j+\varepsilon] $.
We use the  parametrix $Q(\lambda)=\delta_{{\bf Q}(\lambda)}$ of the proof of Lemma \ref{lemmafordiscreteness} and get \eqref{dXQ+QdX0} with those $\chi_j$. 
 As in the proof of Lemma \ref{lemmafordiscreteness}, $F(\lambda)-{\rm Id}=R(\lambda)+K(\lambda)$ with $K(\lambda)$ compact on $\mc{H}_{NG}$ and $\|R(\lambda)\|<1/2$, and by the Remark following Lemma \ref{lemmafordiscreteness}, we can choose $T_1>0$ large and $T_j>0$ small for $j=2,\dots \kappa$ so that this still holds. Using  Lemma \ref{lem:P1_properties}, 
 we deduce $\ran(\Pi_0(\lambda))\subset C^{-\infty}_{E_u^*}(\M;\Lambda\a^*)$ if $\Pi_0(\lambda)$ is the spectral projector of $F(\lambda)$ at $z=0$. 
We will show that the range of the spectral projector $\Pi_0(\lambda)$ at $z=0$ of $F(\lambda)$ actually satisfies
\begin{equation}\label{ranP1smooth}
\ran \Pi_0(\lambda) \subset C^\infty(\M;\Lambda \a^*).
\end{equation}
Since $F(\lambda)$ is a scalar operator, we can work on scalar valued distributions, and we shall then identify $F(\lambda)$ with an operator $\mc{H}_{NG}\to \mc{H}_{NG}$ for some $N>0$ large enough, and fixed.

Using Lemma~\ref{lem:nojordan}, $z=1$ is at most a pole of order $1$ of $({\rm Id}-F(\lambda)-z)^{-1}$, so that each $u\in \ran(\Pi_0(\lambda))$ satisfies $F(\lambda)u=0$.
Then let $u\in \mc{H}_{NG}$ such that $F(\lambda)u=0$. 

Recall from \cite[eq (2.6)]{DZ16a} that ${\rm WF}(u)={\rm WF}_h(u)\cap T^*\M\setminus \{0\}$.
We are now going to show that ${\rm WF}_h(u)\cap \{(x,\xi)\in E_u^*\, |\, |\xi|\in [c_1,c_2]\}=\emptyset$ for some $0<c_1<c_2$ by using the equation $F(\lambda)u=0$, the propagation of semiclassical wavefront sets (Egorov theorem \cite[Theorem 11.12]{Zwo12}) and the explicit expression of $F(\lambda)$ in terms of the propagators $e^{-tX_{A_j}(\lambda)}$.

\begin{figure}
\centering
\def\svgwidth{0.7\linewidth}
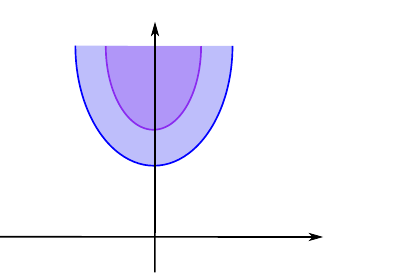
\caption{\label{fig:PS_regions} Schematic sketch of the phase space regions appearing in the proof of Proposition~\ref{smoothnessat0}.}
\end{figure}

For $T_1>0$ large enough but fixed and $T_2,\dots,T_\kappa$ small enough, one can find a closed neighborhood $W_u$ of 
$E_u^*\cap \partial \bbar{T}^*\M$ in the fiber radial compactification of $T^*\M$, which is conic for $|\xi|$ large, $0<c_1<c_2$ such that for all $t_1\in [T_1/2,T_1+\varepsilon]$ and all $t_j\in[0,T_j+\varepsilon]$ when $j\geq 2$ we have
\[   W_u\subset e^{-\sum_{j=1}^\kappa t_jX_{A_j}^H}(W_u) \textrm{ and } \{(x,\xi)\in E_u^*\, |\, |\xi|\in [c_1,c_2]\}\subset e^{-\sum_j t_jX_{A_j}^H}(W_u)\setminus W_u.\]
First we choose $b_0\in S^0(\M)$ with the following properties: first,
\[ 
b_0\geq 0, \quad b_0(x,\xi)=1 \textrm{ in } T^*\M\setminus W_u, \quad b_0(x,\xi)=0 \textrm{ in } e^{\frac{T_1}{2}X^H_{A_1}}(W_u),
\]
then, for each $t=(t_1,\dots,t_\kappa)$ with $t_j\in [T_j,T_j+\varepsilon]$ the symbol
\[
0\leq b_0(x,\xi)-b_0(e^{\sum_{j=1}^\kappa t_jX^H_{A_j}}(x,\xi))
\]
is equal to $1$ on $\{(x,\xi)\in E_u^*\, |\, |\xi|\in [c_1,c_2]\}$,
and third, by a partition of unity we can insure that there is $c_0\in S^0(\M)$ such that $b_0^2+c_0^2=1$. If $B_0=\Op_h(b_0)$ and $C_0=\Op_h(c_0)$ are the corresponding operators then 
$B_0B_0^* + C_0C_0^* ={\rm Id} + hR$ with $R \in\Psi_h^{-1}(\mc M)$. As in the construction of a parametrix, one can modify the symbols by lower order terms and get symbols $b, c$ such that for $B=\Op_h(b)$ and $C=\Op_h(c)$ we have ${\rm Id}-B^*B-C^*C\in h^\infty\Psi_h^0(\M)$. Indeed, for $S\in \Psi_h^{-1}(\mc M)$ with real principal symbol (using $\Op_h(q)-\Op_h(q)^*\in h\Psi_h^0(\mc{M})$ if $q\in S^0(\mc{M})$ 
is real-valued), we have
\[(B_0+hS)(B_0+hS)^* + (C_0+hS)(C_0+hS)^* - {\rm Id} - h(R + 2S(B_0+C_0)) \in h^2\Psi_h^{-2}(\M)\] 
and, as $B_0+C_0$ is semiclassically elliptic, we can invert it microlocally and find $s\in S^0(\mc{M})$ so that $S=\Op_h(s)$ satisfies 
$R + 2S(B_0+C_0) \in h\Psi^0(\mc{M})$ and we gain a power of $h$ if we correct $B_0,C_0$ by $hS$. This argument then can be iterated. Note furthermore that the regions where $b_0=0$ respectively $b_0=1$ are still valid for $b$.

Note that the escape function $G$ can be chosen so that the order function
$m\geq 0$ in the region $T^*\M\setminus W_u$ for $|\xi|$ large enough. 
Since $u\in \mc{H}_{NG}$, we thus have $Bu\in L^2$. Let $\widetilde{\chi}\in C^\infty(\R^\kappa)$
be given by $\widetilde{\chi}(t)=(-1)^\kappa\prod_{j=1}^\kappa \chi_j'(t_j)\geq 0$ for $t\in \R^\kappa$. Recalling that $F(\la)$ is the operator introduced in \eqref{dXQ+QdX0}, we can write, using the semiclassical Egorov Lemma  in its simple form of a coordinate change \cite[Proposition E19]{DZ19}, 
\[
\begin{split} 
Bu=B({\rm Id}-F(\lambda))u=& \int_{(\R^+)^\kappa} Be^{-\sum_{j=1}^\kappa t_jX_{A_j}(\lambda)}u\widetilde{\chi}(t)dt_1\dots dt_\kappa\\
=& 
\int_{(\R^+)^\kappa} e^{-\sum_{j=1}^\kappa t_jX_{A_j}(\lambda)}B_{t}u\widetilde{\chi}(t)dt_1\dots dt_\kappa
\end{split}
\]
with $B_{t}-\Op_h(b\circ e^{\sum_j t_jX^H_{A_j}})\in h\Psi_h^{-1}(\M)$ and ${\rm WF}_h(B_{t})\subset e^{-\sum_j t_jX^H_{A_j}}({\rm WF}_h(B))$. This gives, using $\|e^{-t_jX_{j}(\lambda)}\|_{\mc{L}(L^2)}=1$ by the fact that $\mu$ is invariant, and using Cauchy-Schwarz
\[
\|Bu\|^2_{L^2}=\|B({\rm Id}-F(\lambda))u\|^2_{L^2}\leq \int_{(\R^+)^\kappa} \|B_tu\|^2_{L^2}\widetilde{\chi}(t)dt \int_{(\R^+)^\kappa}\widetilde{\chi}(t)dt
\]
We can then write, since $\int_{(\R^+)^\kappa}\widetilde{\chi}(t)dt=1$,
\begin{equation}\label{ineqBtBu} 
\int_{(\R^+)^\kappa} (\|Bu\|^2_{L^2}-\|B_tu\|^2_{L^2})\widetilde{\chi}(t) dt\leq 0. 
\end{equation}
Next, recall that $ \supp \widetilde{\chi}(t) \subset 
\prod_{j\geq 1} [T_j,T_j+\eps]$. We claim that for $t\in \supp\tilde\chi$ there is $e_t\in S^0(\M;[0,1])$ such that 
we have $B^*B-(B_t^*B_t+E_t^*E_t)\in h^\infty \Psi_h^0(\M)$ for $E_t:=\Op_h(e_t)$ and such that 
$e_t(x,\xi) = 1+ \mathcal O(h)$ in the region $\{(x,\xi)\in E_u^*\, |\, |\xi|\in [c_1,c_2]\}$. Indeed, $E_t$ is microlocally equal to $C_t:=e^{\sum_{j=1}^\kappa t_jX_{A_j}(\lambda)}Ce^{-\sum_{j=1}^\kappa t_jX_{A_j}(\lambda)}$ on $\WF_h(B_t)$ and to $B$ on the complement of ${\rm WF}_h(B_t)$.
 This implies, thanks to \eqref{ineqBtBu},  
\[  \int_{(\R^+)^\kappa}\|E_tu\|^2_{L^2}\widetilde{\chi}(t)dt=\mc{O}(h^\infty). \]
There is $f,g_t\in S^0(\M;[0,1])$ with $f=1+\mathcal O(h)$ on $\{(x,\xi)\in E_u^*\, |\, |\xi|\in [c_1,c_2]\}$ with $f$ independent of $t$, such that for $t\in \supp\tilde\chi$
$E_t^*E_t-(F^*F+G_t^*G_t)\in h^\infty \Psi_h^0(\M)$, where $F=\Op_h(f)$ and 
$G_t= \Op_h(g_t)$. We thus obtain 
\[  \|Fu\|^2_{L^2} \leq \int_{(\R^+)^\kappa}\|E_tu\|^2_{L^2}\widetilde{\chi}(t)dt+\mc{O}(h^\infty)=\mc{O}(h^\infty), \]
which implies that ${\rm WF}_h(u)\cap \{(x,\xi)\in E_u^*\, |\, |\xi|\in [c_1,c_2]\}=\emptyset$. We then conclude that ${\rm WF}(u)\cap E_u^*=\emptyset$, which also shows that $u\in C^\infty$ and \eqref{ranP1smooth}. 

Then we define the following map
\begin{equation}\label{injectsmooth}
 \mathcal I: \Abb
 {\ker_{\ran \Pi_0(\lambda)} d_{X(\lambda)}/\ran_{\ran \Pi_0(\lambda)}d_{X(\lambda)}}
 {\ker_{C^\infty\Lambda} d_{X(\lambda)}/\ran_{C^\infty\Lambda}d_{X(\lambda)}}
 {u + \ran_{\ran \Pi_0(\lambda)}d_{X(\lambda)}}{ u + \ran_{C^\infty \Lambda}d_{X(\lambda)} }
\end{equation}
which is well defined since $\ran \Pi_0(\lambda)\subset C^\infty\Lambda$. We claim that this map is injective:  let $u= d_{X(\lambda)}v\in \ran \Pi_0(\lambda)$ with $v\in C^\infty \Lambda^j$, then we need to show that 
$u= d_{X(\lambda)}w$ for some $w\in \ran \Pi_0(\lambda)$. But is suffices to use $[d_{X(\lambda)},\Pi_0(\lambda)]=0$ to see that $u=\Pi_0(\lambda)u=d_{X(\lambda)}\Pi_0(\lambda)v$.
This proves the claim and concludes the proof of the lemma by using also the isomorphism \eqref{eq:intrinsic_isom2}.
\end{proof}
\begin{lemma}
Assume that there is a smooth invariant measure $\mu$ for the action, i.e. $\mc{L}_{X_{A}}\mu=0$ for each $A\in\a$, and that ${\rm supp}(\mu)=\mc{M}$. Then the periodic tori are dense in $\mathcal{M}$.
\end{lemma}

\begin{proof}
Since $\mathcal{M}$ is compact, the measure is finite, so we can apply Poincar\'e's recurrence theorem: almost every point $x$ of $\mathcal{M}$ is recurrent, i.e. its orbit comes back infinitely close to $x$ infinitely many times (and this for each direction of the action).  Katok-Spatzier \cite[Theorem 2.4]{KaSp94} proved a closing lemma for Anosov actions: there is $C,\delta>0$ such that whenever there is $x\in\mathcal{M}$ and $t\in\mc W$ with $d(\tau(t)x,x)<\delta$, $\|t\|>C$, then there is a periodic torus for the action at distance at most $\frac{1}{\delta}d(\tau(t)x,x)$ from $x$.
\end{proof}

\begin{prop}\label{resonanceat0}
Assume that there is a smooth invariant measure $\mu$ for the action, with ${\rm supp}(\mu)=\mc{M}$. Then 
\[ 
\dim \left(\ker_{C^{-\infty}_{E_u^*}\Lambda^j} d_{X}/\ran_{C^{-\infty}_{E_u^*}\Lambda^{j}} d_{X}\right)= \dim \Lambda^j\a^*={\binom{\kappa}{j}} 
\]
and the cohomology space is generated by the constant forms $e_{i_1}'\wedge \dots\wedge e_{i_\kappa}'$ if $(e_j')_j$ is a basis of $\a^*$. 
\end{prop}
\begin{proof}
In the proof of Proposition~\ref{smoothnessat0} with $\lambda=0$, we have defined an operator $F(0)$ that is Fredholm on $\mc{H}_{NG}$ and $\Pi_0(0)$ is its spectral projector at $z=0$, with ${\rm ran}(\Pi_0(0))\subset C^\infty(\M)$. Recall also that $F(0)$ is scalar and can thus be considered as an operator on functions. Let us show that ${\rm ran}(\Pi_0(0))=\R$ consists only of constants under our assumptions. Pick $u\in C^\infty(\M)$ such that $F(0)u=0$. Let $x\in \M$ belong to a closed orbit in the Weyl chamber, i.e. $\varphi_{t_0}^{X_A}(x)=x$ for some $A\in \W$ and $t_0>0$. Then it is a classical result that the orbit $T_x:=\{\varphi_s^{X_{\tilde A}}(x)\, |\, s\in \R, \tilde{A}\in \a\}$ is a closed $\kappa$-dimensional torus  (a proof compatible with the present notation can e.g. be found in \cite[Lemma 3.1]{BGW21}). It is isomorphic to $\R^\kappa/\Z^\kappa$ by the map
\[
\psi_x: t \in\R^\kappa\mapsto \tau\Big(\sum_{j=1}^\kappa t_j A'_j\Big)(x)
\]
for some basis $A'_i\in\a$.  Note that $\psi_x^*( e^{\sum_\ell s_\ell X_{A_\ell'}} u)(t) = \psi_x^*u(t+s)$.
Let us restrict the identity $F(0) u= 0$ or equivalently $R(0)u=u$ on $T_x$. We can decompose $v:=\psi_x^*u$ into Fourier series
\[
t\in \R^\kappa, \quad v(t)=\sum_{k\in\Z^\kappa}e^{2i\pi k.t}v_k.
\]
Recall that the basis for which $R(0)$ was constructed in Proposition~\ref{smoothnessat0} was denoted by $A_1,\ldots,A_\kappa\in \mathfrak a$. We can express this basis in terms of the basis $A'_j$ of the periodic torus via some base change matrix $A_j=\sum_{i}M_{ij}A_i'$ (using 
$\sum_{\ell=1}^\kappa s_\ell A_{\ell}=\sum_{\ell,i} s_\ell M_{i \ell}A'_{i}$) the identity $Ru(x)=u(x)$ implies
\[
\begin{split}
\sum_{k\in\Z^\kappa}e^{2i\pi k.t}v_k(\psi_x^*R(0)u)(t) =&  \int_{(\R^+)^\kappa}(\psi_x^*u)(t-Ms)\widetilde \chi(s)ds \\
= & \sum_{k\in\Z^\kappa}v_ke^{2i\pi k.t}\int_{(\R^+)^\kappa}e^{-2i\pi k.Ms}\widetilde{\chi}(s)ds
\end{split}\]
with $M=(M_{ij})_{ij}$ real valued and $\tilde{\chi}$ was defined in the proof of Proposition \ref{smoothnessat0}. 
This shows that for each $k\in\Z^\kappa$, 
\[
v_k=0 \textrm{ or } \int_{(\R^+)^\kappa}(e^{-2i\pi k.Ms}-1)\widetilde{\chi}(s)ds=0.
\]
Using that $\widetilde{\chi}\geq 0$ and $\widetilde{\chi}(s)>0$ in some open set and that $M$ is invertible, we see that 
either $v_k=0$ or $k=0$, i.e. $v(t)=v(0)$ is constant. Therefore $u$ is constant on each periodic torus. Since $u$ is smooth and the periodic tori are dense, this implies that $d_Xu=0$ and $u(\varphi_t^{X_A}(x))=u(x)$ for each $x\in \M$, $t\in \R$ and $A\in\a$. Taking $A\in \W$, there is $\nu>0$ such that for each $t>0$ large enough so that $|d\varphi_t^{X_A}v|\leq e^{-\nu t}|v|$ for each $v\in E_s$. Thus 
\[
|du_x(v)|=|du_{\varphi_t^{X_A}(x)}d\varphi_t^{X_A}(x)v|\leq \|du\| e^{-\nu t}|v|.
\]
Letting $t\to \infty$, we conclude that $du|_{E_s}=0$. The same argument with $t<0$ shows that 
$du|_{E_u}=0$ and therefore $du=0$. Since $F(0)1=0$, this shows that, when viewed as an operator on $\Lambda\a^*$, $\ran \Pi_0(0)$ is exactly the space of constant forms. We can then use the isomorphism \eqref{eq:intrinsic_isom2} to conclude the proof since it is direct to see that constant forms $e_{i_1}'\wedge\dots\wedge e_{i_j}'$ form a basis of $\ker d_{X}/{\rm Im}\, d_X$ on ${\rm Im}(\Pi_0(0))$ (as $d_X|_{\ran \Pi_0(0)}=0$). 
\end{proof}

Note that in their paper \cite{KaSp94} Katok-Spatzier study the first cohomology group of dynamical systems and show that any smooth cocycle is smoothly conjugated to a constant function \cite[Theorem 2.9 a)]{KaSp94}. In our language of Taylor complexes, this result implies that for standard Anosov actions $\dim ( \ker_{C^{\infty}\Lambda^1} d_{X}/\ran_{C^{\infty}\Lambda^1} d_{X}) = \kappa$ and is spanned by the constant forms. Combining this fact with Proposition \ref{resonanceat0}, we obtain:
\begin{cor}
If the Anosov $\R^\kappa$-action is standard in the sense of \cite{KaSp94}, then  the map \eqref{isomreson-smooth} is an isomorphism for $j=1$.
\end{cor}

\subsection{Ruelle-Taylor resonances and mixing properties}
In this section we do not assume anymore that a volume measure is preserved, and want to establish the following relation of Ruelle-Taylor 
resonances and mixing properties.
\begin{prop}\label{prop:mixing}
 Let $X$ be an Anosov action on $\mc M$ then the following are equivalent:
 \begin{enumerate}
  \item \label{it:weak} There is a direction $A_0\in \a$ 
  such that $\varphi_t^{X_{A_0}}$ is weakly mixing with respect to the full physical measure $\mu_1$.
  \item \label{it:res} $0$ is the only Ruelle-Taylor resonance on $i\a^*$ and there is a unique normalized physical measure $\mu_1$.
  \item \label{it:strong} For each $A\in \mc W$,
  $\varphi_t^{X_A}$ is strongly mixing with respect to the full physical measure $\mu_1$. 
  \end{enumerate}
\end{prop}
\begin{proof}
 Obviously (\ref{it:strong}) $\Rightarrow$ (\ref{it:weak}). So let us prove (\ref{it:weak}) $\Rightarrow$ (\ref{it:res}): Assume that there is either a non-zero Ruelle-Taylor resonance $i\lambda\in i\a^*$ or a non-unique normalized physical measure. Then by Proposition~\ref{prop:SRBmeasures}(\ref{it:ac}) there is a non-constant bounded density $f\in L^\infty(\mc M, \mu_1)$ with $X_Af= i\lambda(A)f$ for all $A\in \a$ (setting $\lambda=0$ if the density comes from the non-uniqueness of the physical measure). As $f$ is non-constant there exists $g\in L^\infty(M,\mu_1)$ with $\int g\,d\mu_1 =0$ but $\int gf\,d\mu_1\neq 0$. With these two functions, the correlation function
 \[
  C_{f,g}(t;A_0):=\int_{\mc M} g (\varphi_{-t}^{X_{A_0}})^*f\,d\mu_1 - \int_{\mc M}g\, d\mu_1\int_{\mc M}f\,d\mu_1 = e^{-i\lambda(A_0)t}\int_{\mc M} gf\,d\mu_1
 \]
 so $\varphi_t^{X_{A_0}}$ is not weakly mixing.
 
 We will now prove (\ref{it:res})$\Rightarrow$(\ref{it:strong}) using the regularity of a joint spectral measure: Let us first introduce these measures. We consider the space $L^2(M,\mu_1)$. Since the measure $\mu_1$ is flow-invariant, the flow acts as unitary operators on $L^2(M,\mu_1)$. In particular, for each $A\in \a$, $X_A$ is skew-adjoint when acting on $L^2(M,\mu_1)$ with domain
\[
\mathcal{D}(X_A) = \left\{ u\in L^2(M,\mu_1)\ \middle|\ \lim_t \frac{1}{t}(e^{tX_A}u-u) \text{ exists}\right\} = \{u \in L^2(M,\mu_1)\ |\ X_A u \in L^2(M,\mu_1)\}.
\]
Additionally, since the flows commute, the $X_A$ are \emph{strongly commuting}, so that we can apply the joint spectral theorem -- see Theorem 5.21 in \cite{Schmudgen}. There exists a Borel, $L^2(M,\mu_1)$-projector valued, measure $\nu$ on $\a^*$ such that for $u\in L^2(M,\mu_1)$,
\[
u = \int_{\a^*} d\nu(\vartheta)u,\qquad  X_A u = \int_{\a^*} i\vartheta(A) d\nu(\vartheta)u \text{ for all }A\in \mathfrak a.
\]
We will prove the following regularity result of these measures below:
\begin{lemma}\label{lem:spec_meas_WF}
Let $X$ be an Anosov action. Assume that there is no non-zero purely imaginary Ruelle-Taylor resonance and a unique normalized physical measure. Then for any $f,g\in C^\infty(\mc M)$ with $\int_{\mc M} f\,d\mu_1 = \int_{\mc M} g\, d\mu_1 = 0$ we consider $\nu_{f,g}(\theta):= \langle \nu(\theta)f,g\rangle_{L^2(\mc M, \mu_1)}$ which are finite complex valued measures on $\a^*$. Then the analytic wavefront set
\footnote{See \cite[§3.3]{Fol89} for the definition and basic properties of the analytic wavefront set. In our proof, we need to use a non-quadratic phase, and only the quadratic case is treated in Folland; however this is just a slight technical hurdle, as mentionned by Folland at the start of p. 160. For completeness, however, we will refer to \cite{Sjo82} (in French).} \footnote{The usual $C^\infty$ wavefront set is contained in the analytic wavefront set, i.e $\WF\subset\WF_a$; see Theorem 3.22 in Folland.}
$\WF_a(\nu_{f,g})\subset \a^*\times \a$ fulfills
\[
\WF_a(\nu_{f,g}) \cap (\a^*\times \mathcal W) = \emptyset.
\]
\end{lemma}
Before proving this Lemma let us show that it implies (\ref{it:strong}). Take $A_0\in \mc{W}$, $f,g$ as in the above Lemma, then the spectral theorem yields
\[
 C_{f,g}(t;A_0) = \int_{\mc M} g(\varphi_{-t}^{X_{A_0}})^*fd\mu_1 = \int_{\a^*} e^{-i\vartheta(A_0)t} d\nu_{f,\overline g}(\vartheta).
\]
Given any $\varepsilon>0$, using the fact that $\nu_{f,\overline g}$ is finite, there is a cutoff function $\chi_K \in C_c^\infty(\a^*,[0,1])$ equal to 1 on a sufficiently large compact set $K\subset \a^*$ such that $|\int_{\a^*} e^{-i\vartheta(A_0)t} (1-\chi_K) d\nu_{f,\overline g}(\vartheta)|\leq \varepsilon/2$ uniformly in $t$. Furthermore by the fact that the wavefront set is empty in the direction of the Weyl chamber $\mathcal W$ we deduce that there is $T$ such that $|\int_{\a^*} e^{-i\vartheta(A_0)t} \chi_K d\nu_{f,\overline g}(\vartheta)|\leq \varepsilon/2$ for any $t>T$ thus $\lim_{t\to\infty} C_{f,g}(t,A_0)=0$. The passage to arbitrary $L^2(\mc M, \mu_1)$ functions follows by the density of the smooth functions.
\end{proof}

\begin{proof}[Proof of Lemma~\ref{lem:spec_meas_WF}]
Let us pick any $A_0\in \mc W$ and a basis $A_1,\ldots,A_\kappa\in \mc W$ such that these elements span an open cone around $A_0$. With this basis we identify the joint spectral measure with a measure on $\R^\kappa$. Recall the definition of $R_\sigma(i\lambda)$ from the proof of Proposition~\ref{prop:SRBmeasures} which was based on the choice of an even, positive $\psi\in C^\infty((-1/2,1/2))$ with $\int\psi=1$ and some $\sigma\in \R_+^\kappa$. Using the spectral theorem we calculate for any $f,g\in L^2(\mc M, \mu_1)$
\begin{equation}\label{eq:Rk_spec}
 \langle R_\sigma(i\lambda)^k f, g\rangle_{L^2(\mc M, \mu_1)} = 
 \int_{\R^\kappa} \hat\Psi(\vartheta+\lambda)^k e^{-ik\sigma(\vartheta+\lambda)} d\nu_{f,g}(\vartheta)
\end{equation}
where $\Psi(t):=\prod_{j=1}^\kappa\psi(t_j)$. Now let us define the following closed subspaces $\mc H_{NG, 0}:=\{u\in \mc H_{NG}| \int ud\mu_1=0\}\subset \mc H_{NG}$. Note that these are well defined for sufficiently large $N$ because $\mu_1\in \mc H_{-NG}$. Furthermore from the invariance of $\mu_1$ under the Anosov actions the spaces $\mc H_{NG,0}$ are preserved under $R_\sigma(i\lambda)$. Now the assumption that there is no imaginary Ruelle-Taylor resonance except zero and that there is a unique normalized physical measure imply (combining the findings of Section~\ref{sec:ia_non_volume}) that $R_\sigma(i\lambda)$ has a spectral radius $<1$ on $H_{NG, 0}$
for any $\lambda\in \R^\kappa$, and $\sigma\in \R_+^\kappa$ sufficiently large. Thus there are $C_{\sigma,\lambda}, \varepsilon_{\sigma,\lambda}>0$, locally uniformly in $\sigma,\lambda$ such that $\|R_\sigma(i\lambda)^k\|_{\mc H_{NG,0}}\leq C_{\sigma,\lambda}e^{-\varepsilon_{\sigma,\lambda} k}$. Now let $f,g$ be as in the assumption of our Lemma, then we can estimate
\[
 \langle R_\sigma(i\lambda)^k f, g\rangle_{L^2(\mc M,\mu_1)} \leq 
 \|R_\sigma(i\lambda)^kf\|_{\mc H_{NG,0}}\|g\mu_1\|_{\mc H_{-NG}} \leq C_{f,g,\sigma,\lambda}e^{-\varepsilon_{\sigma,\lambda}k}.
\]
Let us come back to the expression \eqref{eq:Rk_spec} involving the spectral measures. By the properties of $\psi$ we deduce that near zero $\hat \Psi(\xi)   = \exp(-S(\xi))$ with some analytic function $S(\xi) = a|\xi|^2 + \mathcal O(|\xi|^4)$. Furthermore for any $\delta>0$, there is $\varepsilon_2>0$ such that $\hat\Psi(\xi)<e^{-\varepsilon_2}$ for $|\xi|>\delta$. Choosing a cutoff function $\chi\in C_c^\infty((-3\delta,3\delta)^\kappa)$ with $\chi(\xi)=1$ for $|\xi|<2\delta$ we get by the boundedness of $\nu_{f,g}$ for an arbitrary fixed $\lambda_0\in \R^\kappa$
\[
 \left|\langle R_\sigma(i\lambda)^k f,g\rangle_{L^2(\mc M,\mu_1)} - \int_{\R^\kappa} \hat\Psi(\vartheta+\lambda)^k e^{-ik\sigma(\vartheta+\lambda)} \chi(\vartheta+\lambda_0) d\nu_{f,g}(\vartheta)\right|\leq Ce^{-\varepsilon_2k}
\]
uniformly for $\sigma\in \R_+^\kappa$, $|\lambda-\lambda_0|<\delta$. Putting everything together we get
\[
 \left|\int_{\R^\kappa} e^{-kS(\vartheta+\lambda)-ik\sigma(\vartheta+\lambda)} \chi(\vartheta+\lambda_0) d\nu_{f,g}(\vartheta)\right|\leq \tilde Ce^{-\tilde \varepsilon k}
\]
with $\tilde C,\tilde \varepsilon>0$ locally uniform in $|\lambda-\lambda_0|<\delta$ and $\sigma\in \R_+^\kappa$. In the expression on the left hand sice, we recognize a semi-classical Fourier-Bros-Iagolnitzer (FBI) transform of the distribution $d\nu_{f,g}$ at parameters $(\lambda,\sigma)$, with $h=1/k$. That it decays exponentially as $h\to 0$, uniformly in $\lambda$ close to $\lambda_0$ and uniformly in $\sigma\in\R_+^\kappa$ is the definition that
\[
\WF_a(d\nu_{f,g}) \cap \{\lambda_0\}\times \R_+^\kappa = \emptyset.
\]
Here $\WF_a$ is the \emph{analytic} wavefront set, see the \cite[Définition 6.1, Proposition 6.2]{Sjo82}.

Recall furthermore that we had identified $\a^*\cong\R^\kappa$ by the above choice of the basis $A_j$, thus our result implies that there is no analytic wavefront set in $\a^* \times \{\sum c_jA_j|c_j\in \R_+\}$, but as the $A_j$ span an arbitrary subcone of $\mathcal W$ we also get the absence of analytic wavefront set in $\a^*\times\mathcal W$ and we have completed the proof of Lemma~\ref{lem:spec_meas_WF}.
\end{proof}

\appendix

\section{Tools from microlocal analysis}\label{sec:microlocal}

We recall here some essentials of microlocal analysis. In the paper, we are working with pseudodifferential operators acting on 
$C^\infty(\mc M; E)\otimes \Lambda \a_\C^* \cong C^\infty( \mc M; E\otimes \Lambda \a_\C^*)$. 
Note that by fixing an arbitrary scalar product on $\a^*$ the bundle $E\otimes \Lambda:=E\otimes \Lambda \a_\C^* \to \mc M$ is again a Riemannian bundle. We will therefore introduce notations for pseudodifferential operators on general Riemannian bundles $E\to \mc M$ over a compact Riemannian manifold $\mc M$. Only when we want to exploit some specific structures of $E\otimes \Lambda$, will we refer to this particular bundle. 

For more details we refer to standard references such as \cite{GS94}. For the details concerning the anisotropic calculus we refer to \cite{FRS08}.
\begin{Def}\label{def:symbols}
 Let $k\in\R$, $1/2<\rho\leq 1$. Then the \emph{standard symbol space} $S^k_\rho(\mc M; E)$ is the space of
 functions $a\in C^\infty(T^*\mc M; \tu{End}(E))$, for which in any local chart $U\subset \R^n$ of $\mc M$ and any local trivialization of the bundle, for any compact set $K\subset U$ and any two multiindices $\alpha,\beta\in \N^n$
 \[
  \sup_{(x,\xi)\in T^*U,x\in K} \|\partial_x^\alpha \partial_\xi^\beta a(x,\xi)\|\langle\xi\rangle^{-(k - \rho|\beta|+(1-\rho)|\alpha|)} <\infty. 
 \]
 Given a zeroth order symbol $m(x,\xi) \in S^0_1(\M)$ then the \emph{anisotropic symbol space} $S^{m(x,\xi)}_\rho(\mc M; E)$ is the space of
 functions $a\in C^\infty(T^*\mc M; \tu{End}(E))$ for which in any local chart $U\subset \R^n$ for any compact set $K\subset U$ and any two multiindices $\alpha,\beta\in \N^n$
 \[
  \sup_{(x,\xi)\in T^*U,x\in K} \|\partial_x^\alpha \partial_\xi^\beta a(x,\xi)\|\langle\xi\rangle^{-(m(x,\xi) - \rho|\beta|+(1-\rho)|\alpha|)} <\infty. 
 \]
We furthermore set\footnote{Note that $\cap_{k>0}S_\rho^{-k}(\mc M; E)$ is independent of $1/2<\rho \leq 1$ and we therefore drop the index in the notation of $S^{-\infty}(\mc M; E)$.}
\[ \begin{gathered}
  S^{-\infty}(\mc M;E) := \cap_{k>0}S_\rho^{-k}(\mc M; E),\quad S^\infty(\mc M;E):=\cup_{k>0}S_\rho^{k}(\mc M; E),\\
  S^{m+}_\rho(\mc M; E) := \cap_{\varepsilon>0}S^{m+\varepsilon}_\rho(\mc M; E),\quad S^{m}_{\rho-}(\mc M; E) := \cup_{\varepsilon>0}S^{m}_{\rho-\varepsilon}(\mc M; E).
 \end{gathered}\]
 \end{Def}
Note that by setting $m(x,\xi) = k\in \R$ the standard symbols are a special case of anisotropic symbols. We will therefore mostly introduce the notation in the anisotropic setting as it comprises the standard symbols as a special case. Furthermore, note that $x\mapsto \Id_{E_x}$ is a global smooth section of $\tu{End}(E)\to\mc M$ and multiplication with this section yields a canonical embedding 
$S_\rho^{\infty}(\mc M) \hookrightarrow S_\rho^{\infty}(\mc M; E)$. We will denote symbols in the image of this embedding as the space of \emph{scalar symbols}. 
 
After fixing a finite atlas and a suitable partition of unity of $\mc M$ one can define a quantization  (see e.g. \cite[E.1.7]{DZ19}) that associates to any $a\in S_\rho^\infty(\mc M; E)$ a continuous operator $\Op(a): C^\infty(\mc M; E)\to C^\infty(\mc M; E)$ which extends to a continuous operator $\Op(a): C^{-\infty}(\mc M; E)\to C^{-\infty}(\mc M; E)$. We denote by $\Psi^{-\infty}(\mc M;E)$ the space of smoothing operators $A:C^{-\infty}(\mc M; E)\to C^\infty(\M; E)$. The quantization has the property that $\Op(S^{-\infty}(\mc M; E))\subset \Psi^{-\infty}(\mc M; E)$. We say that $A\in \Psi^{m}_\rho(\mc M; E)$
iff there is $a\in S^{m}_\rho(\mc M; E)$ such that $A-\Op(a)\in\Psi^{-\infty}(\mc M; E)$. 
When $\rho=1$, we will drop the $\rho$ index and write $S^m(\M;E)$ and 
$\Psi^m(\M;E)$ instead of $S_1^m(\M;E)$ and $\Psi_1^m(\M;E)$.

With any $A\in \Psi_\rho^{m}(\mc M; E)$ one can associate its principal symbol 
\[
\sigma^{m}_p(A)\in S_\rho^{m}(\mc M; E) /S_\rho^{m-2\rho +1}(\mc M; E).
\]
The principal symbol is an inverse to $\Op$ in the sense that 
\[
 \sigma^{m}_p\circ \Op :S_\rho^{m}\to S_\rho^{m} /S_\rho^{m-2\rho +1} \tu{ and } \Op\circ \, \sigma^{m}_p :\Psi_\rho^{m}\to \Psi_\rho^{m} /\Psi_\rho^{m-2\rho +1}
\]
are simply the projections on the respective quotients. 
\begin{example}\label{exmpl:principal_symbols}
 Any $k$-th order differential operator $P$ with smooth coefficients  on the bundle $E\to \mc M$ is
 in $\Psi^k_1(\mc M; E)$ and a representative of its principal symbol $\sigma_p^k(P)$ can be calculated by
 \[
  \left[\sigma_p^k(P)(x,\xi)\right] u(x) =\lim_{t\to\infty}t^{-k} \left[e^{-it\phi} P(e^{it\phi} u)\right](x),
 \]
where $u\in C^\infty(\mc M; E)$ and $\phi\in C^\infty(\mc M)$ is a phase function with $d\phi(x) =\xi$ (see e.g. \cite[(6.4.6')]{Hoe03}).
As a direct consequence we get:
\begin{enumerate}
 \item For any vector field $X\in C^\infty(\mc M; T^*\M)\subset \Psi^1_1(\mc M)$ we have $\sigma_p^1(X)(x,\xi) = i\xi(X(x))$. 
 \item If $\mathbf X:\a\to \OpDiff^1(\mc M; E) \subset \Psi^1_1(\mc M, E)$ is an admissible lift of an Anosov action, then for all $A\in \a$ one finds that the principal symbol
 $\sigma_p^1(\mathbf X_A)(x,\xi) = i\xi(X_A(x)) \Id_{E_x}$ is scalar.
 \item In order to express the principal symbol of the exterior derivative
 $d_\mathbf X \in \Psi^1_1(\mc M, E\otimes\Lambda\a_\C^*)$ of $X$, let us consider the smooth map 
 $T^*\mc M \ni (x,\xi) \mapsto \xi(X_\bullet(x)) \in \Lambda^1\a^*$. With its help we calculate for $v\in E_x, \omega\in \Lambda \a^*$
 \[
  \sigma_p^1(d_\X)(x,\xi) (v\otimes \omega) = i v\otimes \left(\xi(X_\bullet(x))\wedge\omega\right).
 \]
(Thus $\sigma_p^1(d_\X)$ is scalar on the $E$-component but not on the $\Lambda \a^*$-component as it increases the order of differential forms.) 
\end{enumerate}
\end{example}

\begin{prop}\label{prop:pseudo_product}
 Let $A\in \Psi^{m_1(x,\xi)}_{\rho}(\mc M;  E)$ and $B\in \Psi^{m_2(x,\xi)}_{\rho}(\mc M; E)$, then $AB \in \Psi^{m_1+m_2}_{\rho}(\mc M; E)$ and 
 $\sigma_p^{m_1+m_2}(AB) = \sigma_p^{m_1}(A)\sigma_p^{m_2}(B) \mod S^{m_1+m_2-2\rho+1}_\rho(\M;E)$.
\end{prop}

\begin{Def}
 Given $a\in S^{m(x,\xi)}_\rho(\mc M; E)$, we define its \emph{elliptic set} to be the open cone
 $\Ell^{m(x,\xi)}(a)\subset T^*\mc M\setminus\{0\}$ which consists of all $(x_0,\xi_0)\in T^*\mc M\setminus\{0\}$ for which there is a $C>0$ and a function $\chi\in C^\infty(T^*\mc M)$, positively homogeneous of degree zero for $|\xi|\geq C$, and $\chi(x_0,C\xi_0/|\xi_0|)>0$, such that $a(x,\xi)\in \tu{End}(E_x)$ is invertible for all $(x,\xi)\in \tu{supp}(\chi)$ and  $\chi a^{-1}\in S_\rho^{-m(x,\xi)}(\mc M; E)$. 
 We call $a\in S_\rho^{m(x,\xi)}(\mc M, E)$ \emph{elliptic} if $\Ell^{m(x,\xi)}(a)= T^*\mc M\setminus\{0\}$.
\end{Def}

As a direct consequence of the chain rule for derivatives and the symbol estimates we get
\begin{lemma}
If $a\in S_\rho^{m(x,\xi)}(\mc M; E)$ is a scalar symbol, then $(x_0,\xi_0) \in \Ell^{m(x,\xi)}(a)$ if there exists  an open cone $\Gamma\subset T^*\mc M$ containing $(x_0,\xi_0)$ and $C>0$ such that 
\[
 |a(x,\xi)|\geq\frac{1}{C} \langle\xi\rangle^{m(x,\xi)} \tu{ for all } (x,\xi)\in \Gamma\cap\{|\xi|>C\}.
\]
\end{lemma}
One checks that for $a\in S_\rho^{m(x,\xi)}(\mc M; E)$ and $r\in S_\rho^{m(x,\xi)-\varepsilon}(\mc M; E)$ one has $\Ell^{m(x,\xi)}(a)=\Ell^{m(x,\xi)}(a+r)$ which allows to define the elliptic set of an operator
$A\in \Psi^{m(x,\xi)}(\mc M; E)$ via its principal symbol $\Ell^{m(x,\xi)}(A) := \Ell^{m(x,\xi)}(\sigma^{m(x,\xi)}_p(A))$.
\begin{Def}
 Given $A=\Op(a)\, {\rm mod}\, \Psi^{-\infty}(\mc M; E)$, we define its \emph{wavefront set} to be the closed cone 
 $\WF(A)\subset T^*\mc M\setminus\{0\}$ which is the complement of all $(x_0,\xi_0) \in T^*\mc M\setminus \{0\}$ for which there is an open cone $\Gamma\subset T^*\mc M$ around $(x_0,\xi_0)$ such that for all $N>0,\alpha,\beta\in\N^n$ there is $C_{N,\alpha,\beta}$ such that
 \[
  \|\partial_x^\alpha \partial_\xi^\beta a(x,\xi)\| \leq C_{N, \alpha,\beta}\langle\xi\rangle^{-N} \tu{ for all }(x,\xi)\in \Gamma.
 \]
\end{Def}

The wavefront set has the following property for the product of two pseudodifferential operators $A, B\in \Psi_\rho^\infty(\mc M; E)$:
\[
 \WF(AB) \subset \WF(A)\cap \WF(B).
\]
We will crucially use the following constructions of microlocal parametrices. 
\begin{lemma}\label{prop:elliptic_param}
 If $A\in \Psi^{m_1(x,\xi)}_\rho(\mc M; E), B\in \Psi^{m_2(x,\xi)}_\rho(\mc M; E)$ and $\WF(B)\subset \Ell^{m_1(x,\xi)}(A)$, then there is $Q\in \Psi_\rho^{m_2(x,\xi)-m_1(x,\xi)}$ with $\WF(Q)\subset \WF(B)$ such that
 \[
  AQ -B\in \Psi^{-\infty}(\mc M; E).
 \]
If furthermore $A$ and $B$ are holomorphic families of operators, then $Q$ can be chosen to be holomorphic as well.
\end{lemma}
As a consequence of Lemma~\ref{prop:elliptic_param}, if $A\in \Psi_\rho^{m_1}(\mc M; E)$ and $B\in \Psi_\rho^{m_2}(\mc M; E)$, then
\begin{equation}
 \label{eq:ell_wf_estimate}
 \Ell^{m_1}(A)\cap\WF(B) \subset \WF(AB).
\end{equation}

We also have the following particular case of Egorov's lemma. 
\begin{lemma}
 \label{prop:egorov}
 Let $F\in \Diffeo(\M)$ be a smooth diffeomorphism 
and let $\widetilde{F}\in \Diffeo(E)$ be a lift of $F$, i.e. $\widetilde{F}$ acts linearly in the fibers and $\pi\circ\widetilde{F}=F\circ\pi$ for $\pi:E\to \M$ the fiber projection. Define the \emph{transfer operator}
\[
L_F:C^\infty(\M;E)\to C^\infty(\M;E),\quad (L_Fu)(x):=\widetilde{F}^{-1}(F(x),u(F(x))).
\] 
Then for each $A\in \Psi_\rho^m(\M;E)$, we have $L_FAL_F^{-1}\in \Psi_\rho^{m\circ \Phi}(\M;E)$  
with $\Phi(x,\xi):=(F(x),(dF^{-1})^T\xi)$ and 
\[
\sigma_{p}^{m\circ \Phi}(L_FAL_{F}^{-1})(x,\xi)=\widetilde{F}^{-1}(F(x),\cdot)\circ\sigma_p^{m}(A)(\Phi(x,\xi))\circ\widetilde{F}(x,\cdot).
\]
\end{lemma}
\begin{prop}[$L^2$-boundedness]
 \label{prop:l2-bound}
 Let $A\in \Psi^0_\rho(\mc M; E)$, then $A$ can be extended from an operator on $C^\infty(\mc M; E)$ to a bounded operator on $L^2(\mc M; E)$. Furthermore, for any 
\[
C>\limsup_{|\xi|\to\infty}\|\sigma_p^0(A)(x,\xi)\|,
\]
there exists a decomposition $A=K+R$, where $K\in \Psi^{-\infty}(\mc M; E)$ is a smoothing and hence $L^2$-compact operator and $\|R\|_{L^2\to L^2}\leq C$. If $A_t$ is a smooth family in $\Psi^0_\rho(\mc M; E)$ for $t\in [t_1,t_2]$, the decomposition $A_t=R_t+K_t$ can be chosen so that $t\mapsto R_t$ and $t\mapsto K_t$ are continuous in $t$.
\end{prop}
\begin{proof}
See \cite[Lemma 14]{FRS08} for the proof. The dependence in $t$ is straightforward from the proof.
\end{proof}

We conclude this appendix by mentioning that one can use a small semiclassical parameter
$h>0$ in the quantization, in which case we shall write $\Op_h$, by using the expression in a local chart 
\[
\Op_h(a)f(x)= \frac{1}{(2\pi h)^n}\int e^{\frac{i(x-x')\xi}{h}}a(x,\xi)f(x')d\xi dx'
\] 
if $a$ is supported in a chart. We do not use this semiclassical quantization except in the subsection \ref{sec:Ruelle_at_zero} and we refer to \cite[Appendix E]{DZ19} for the results on semiclassical pseudodifferential operators that we will use. The class 
One of the advantages is that one can get the estimate 
$\|\Op_h(a)\|_{L^2\to L^2}\leq \sup_{x,\xi} |a(x,\xi)|+\mc{O}(h)$ for small $h>0$ and if $a\in S^0(\M;E)$.

\providecommand{\href}[2]{#2}


\end{document}